\documentclass[12pt]{article}
\usepackage[margin=1in]{geometry}  
\usepackage{graphicx}              
\usepackage{amsmath}               
\usepackage{amsthm}                
\usepackage[ansinew]{inputenc}
\usepackage{array,color}
\usepackage{amsxtra}
\usepackage{amstext}
\usepackage{latexsym}
\usepackage{dsfont}
\usepackage[all]{xy}
\usepackage{amsmath,amsfonts,amssymb}

\usepackage{marginnote}

\usepackage{cite}

\usepackage{hyperref}
\hypersetup{
colorlinks   = true,
citecolor    = black,
linkcolor =black,
urlcolor=black,
}

\hypersetup{hidelinks}

\usepackage{authblk}

\newcommand{\email}[1]{%
    \normalsize\href{mailto:#1}{\color{black}{#1} }}

\allowdisplaybreaks

\makeatletter
\newcommand{\subjclass}[2][2020]{%
  \let\@oldtitle\@title%
  \gdef\@title{\@oldtitle\footnotetext{#1 \emph{Mathematics subject classification}: #2}}%
}
\newcommand{\keywords}[1]{%
  \let\@@oldtitle\@title%
  \gdef\@title{\@@oldtitle\footnotetext{\emph{Keywords}: #1.}}%
}
\makeatother

\newtheorem{thm}{Theorem}[section]
\newtheorem{cor}[thm]{Corollary}

\newtheorem{prop}[thm]{Proposition}
\theoremstyle{definition}
\newtheorem{defn}[thm]{Definition}
\theoremstyle{remark}
\newtheorem{rmk}[thm]{Remark}
\theoremstyle{remark}
\newtheorem{rems}[thm]{Remarks}
\newtheorem{ex}[thm]{Example}
\newtheorem{exes}[thm]{Examples}
\numberwithin{equation}{section}

\title{Constructions of BiHom-X algebras and bimodules of some BiHom-dialgebras}

\author[1,2]{Ismail Laraiedh}
\author[3]{Sergei Silvestrov}
\affil[1]{\Affilfont Departement of Mathematics, Faculty of Sciences,
\authorcr \Affilfont Sfax University, Box 1171, 3000 Sfax, Tunisia}
\affil[2]{Departement of Mathematics,
\authorcr \Affilfont College of Sciences and Humanities Al Quwaiiyah,
\authorcr \Affilfont Shaqra University, Kingdom of Saudi Arabia
\authorcr \Affilfont
\email{ismail.laraiedh@gmail.com}, \email{ismail.laraiedh@su.edu.sa}}
\affil[3]{\Affilfont Division of Mathematics and Physics,
\authorcr \Affilfont School of Education, Culture and Communication,
\authorcr \Affilfont M\"{a}lardalen University, Box 883, 72123 V{\"a}ster{\aa}s, Sweden
\authorcr \Affilfont
\email{sergei.silvestrov@mdh.se}}


\date{}

\begin{document}
\maketitle

\abstract{
The aim of this paper is to introduce and to give some constructions results of BiHom-X algebras by using Yau's twisting, Rota Baxter and Some elements of centroids. Next, we define the bimodules of BiHom-left symmetric dialgebras, BiHom-associative dialgebras and BiHom-tridendriform algebras. A sequence of this kind of bimodules can be constructed. Their matched pairs are also introduced and related relevant properties are given.
}

\footnote[0]{{\it Corresponding authors}:
Ismail Laraiedh, Sergei Silvestrov}

\section{Introduction}
The theory of Hom-algebras has been initiated in \cite{HartwigLarSil:defLiesigmaderiv, LarssonSilvJA2005:QuasiHomLieCentExt2cocyid,LarssonSilv:quasiLiealg} motivated by quasi-deformations of Lie algebras of vector fields, in particular q-deformations of Witt and Virasoro algebras. Hom-Lie algebras and more general quasi-Hom-Lie algebras were introduced first by Hartwig, Larsson and Silvestrov in  \cite{HartwigLarSil:defLiesigmaderiv} where a general approach to discretization of Lie algebras of vector fields using general twisted derivations ($\sigma$-deriva\-tions) and a general method for construction of deformations of Witt and Virasoro type algebras based on twisted derivations have been developed. The general quasi-Lie algebras, containing the quasi-Hom-Lie algebras and Hom-Lie algebras as subclasses, as well their graded color generalization, the color quasi-Lie algebras including color quasi-hom-Lie algebras, color hom-Lie algebras and their special subclasses the quasi-Hom-Lie superalgebras and hom-Lie superalgebras, have been first introduced in \cite{HartwigLarSil:defLiesigmaderiv,LarssonSilvJA2005:QuasiHomLieCentExt2cocyid,LarssonSilv:quasiLiealg,LSGradedquasiLiealg,LarssonSilv:quasidefsl2,SigSilv:CzechJP2006:GradedquasiLiealgWitt}.
Subsequently, various classes of Hom-Lie admissible algebras have been considered in \cite{ms:homstructure}. In particular, in \cite{ms:homstructure}, the Hom-associative algebras have been introduced and shown to be Hom-Lie admissible, that is leading to Hom-Lie algebras using commutator map as new product, and in this sense constituting a natural generalization of associative algebras as Lie admissible algebras leading to Lie algebras using commutator map. Furthermore, in \cite{ms:homstructure}, more general $G$-Hom-associative algebras including Hom-associative algebras, Hom-Vinberg algebras (Hom-left symmetric algebras), Hom-pre-Lie algebras (Hom-right symmetric algebras), and some other Hom-algebra structures, generalizing $G$-associative algebras, Vinberg and pre-Lie algebras respectively, have been introduced and shown to be Hom-Lie admissible, meaning that for these classes of Hom-algebras, the operation of taking commutator leads to Hom-Lie algebras as well. Also, flexible Hom-algebras have been introduced, connections to Hom-algebra generalizations of derivations and of adjoint maps have been noticed, and some low-dimensional Hom-Lie algebras have been described.
In Hom-algebra structures, defining algebra identities are twisted by linear maps.
Since the pioneering works \cite{HartwigLarSil:defLiesigmaderiv,LarssonSilvJA2005:QuasiHomLieCentExt2cocyid,
LarssonSilv:quasiLiealg,LSGradedquasiLiealg,LarssonSilv:quasidefsl2,ms:homstructure}, Hom-algebra structures have developed in a popular broad area with increasing number of publications in various directions.
Hom-algebra structures include their classical counterparts and open new broad possibilities for deformations, extensions to Hom-algebra structures of representations, homology, cohomology and formal deformations, Hom-modules and hom-bimodules, Hom-Lie admissible Hom-coalgebras, Hom-coalgebras, Hom-bialgebras, Hom-Hopf algebras, $L$-modules, $L$-comodules and Hom-Lie quasi-bialgebras, $n$-ary generalizations of BiHom-Lie algebras and BiHom-associative algebras and generalized derivations, Rota-Baxter operators, Hom-dendriform color algebras, Rota-Baxter bisystems and covariant bialgebras, Rota-Baxter cosystems, coquasitriangular mixed bialgebras, coassociative Yang-Baxter pairs, coassociative Yang-Baxter equation and generalizations of Rota-Baxter systems and algebras, curved $\mathcal{O}$-operator systems and their connections with tridendriform systems and pre-Lie algebras, BiHom-algebras, BiHom-Frobenius algebras and double constructions, infinitesimal BiHom-bialgebras and Hom-dendriform $D$-bialgebras, Hom-algebras have been considered
\cite{AbdaouiMabroukMakhlouf,
AmmarEjbehiMakhlouf:homdeformation,
AttanLaraiedh:2020ConstrBihomalternBihomJordan,
Bakayoko:LaplacehomLiequasibialg,
Bakayoko:LmodcomodhomLiequasibialg,
BakBan:bimodrotbaxt,
BakyokoSilvestrov:HomleftsymHomdendicolorYauTwi,
BakyokoSilvestrov:MultiplicnHomLiecoloralg,
BenMakh:Hombiliform,
BenAbdeljElhamdKaygorMakhl201920GenDernBiHomLiealg,
CaenGoyv:MonHomHopf,
ChtiouiMabroukMakhlouf1,
ChtiouiMabroukMakhlouf2,
DassoundoSilvestrov2021:NearlyHomass,
EbrahimiFardGuo08,
GrMakMenPan:Bihom1,
HassanzadehShapiroSutlu:CyclichomolHomasal,
HounkonnouDassoundo:centersymalgbialg,
HounkonnouHoundedjiSilvestrov:DoubleconstrbiHomFrobalg,
HounkonnouDassoundo:homcensymalgbialg,
kms:narygenBiHomLieBiHomassalgebras2020,
Laraiedh1:2021:BimodmtchdprsBihomprepois,
LarssonSigSilvJGLTA2008,
LarssonSilvJA2005:QuasiHomLieCentExt2cocyid,
LarssonSilv:quasidefsl2,
LarssonSigSilvJGLTA2008:QuasiLiedefFttN,
LarssonSilvestrovGLTMPBSpr2009:GenNComplTwistDer,
MaMakhSil:CurvedOoperatorSyst,
MaMakhSil:RotaBaxbisyscovbialg,
MaMakhSil:RotaBaxCosyCoquasitriMixBial,
MaZheng:RotaBaxtMonoidalHomAlg,
MabroukNcibSilvestrov2020:GenDerRotaBaxterOpsnaryHomNambuSuperalgs,
Makhl:HomaltHomJord,
Makhlouf2010:ParadigmnonassHomalgHomsuper,
MakhloufHomdemdoformRotaBaxterHomalg2011,
MakhSil:HomHopf,
MakhSilv:HomDeform,
MakhSilv:HomAlgHomCoalg,
MakYau:RotaBaxterHomLieadmis,
RichardSilvestrovJA2008,
RichardSilvestrovGLTbnd2009,
SaadaouSilvestrov:lmgderivationsBiHomLiealgebras,
ShengBai:homLiebialg,
Sheng:homrep,
SigSilv:GLTbdSpringer2009,
SilvestrovParadigmQLieQhomLie2007,
SilvestrovZardeh2021:HNNextinvolmultHomLiealg,
QSunHomPrealtBialg,
Yau:ModuleHomalg,
Yau:HomEnv,
Yau:HomHom,
Yau:HombialgcomoduleHomalg,
Yau:HomYangBaHomLiequasitribial,
YauHomMalcevHomalternHomJord}.

In this paper we introduce and give some results on constructions of
BiHom-X algebras using Yau's twisting, Rota Baxter and some elements of centroids.
The bimodules of BiHom-left symmetric dialgebras, BiHom-associative dialgebras and BiHom-tridendriform algebra are defined, and it is shown that a sequence of this kind of bimodules can be constructed. Their matched pairs are also introduced and related relevant properties are given.
In Section \ref{sec:Yaustwistinggeneralization}, we provide some results on constructions of BiHom-X algebras. Section \ref{sec:homleftsymcolordialg} contains definitions and some key results about bimodules of BiHom-associative algebras and BiHom-left-symmetric algebras, and matched pairs of BiHom-left symmetric and BiHom-associative dialgebras. In section \ref{sec:homtridendriformcoloralgebras}, devoted to bimodules of BiHom-tridendriform algebras, definitions and some constructions of BiHom-dendriform and 
BiHom-tridendriform algebras and the concepts of bimodules and matched pairs of BiHom-tridendriform algebra are investigated.

\section{Constructions of BiHom-X algebras}
\label{sec:Yaustwistinggeneralization}
Throughout this paper, all graded vector spaces are assumed to be over a field $\mathbb{K}$ of characteristic different from $2$.

In this section, we provide some results on constructions of BiHom-X algebras.
\begin{defn}
A BiHom-algebra is a $(n+3)$-tuple $(A, \mu_1, \dots, \mu_n,\alpha,\beta)$ in which $A$ is a linear space,
$\mu_i : A\otimes A \rightarrow A$ $(i=1, \dots, n)$ are bilinear maps,
 and $\alpha,\beta : A\rightarrow A$ are linear maps, called the twisting maps.
 If in addition,  $$\alpha\circ\mu_i=\mu_i\circ(\alpha\otimes \alpha),~\beta\circ\mu_i=\mu_i\circ(\beta\otimes \beta), \quad (i=1, \dots, n),$$
 the BiHom-algebra $(A, \mu_1, \dots, \mu_n,\alpha,\beta)$ is said to be multiplicative.
 \end{defn}
\begin{defn}
Let $(A, \mu_1, \dots, \mu_n,\alpha,\beta)$ be a BiHom-algebra. Then
\begin{enumerate}
\item
A BiHom-subalgebra of $(A, \mu_1, \dots, \mu_n,\alpha,\beta)$ is a linear subspace $H$ of $A$, which is closed for the multiplication $\mu_i$ $(i=1, \dots, n)$, and invariant by $\alpha$ and $\beta$, that is, $\mu_i(x,y)\in H,~\alpha(x)\in H$ and $\beta(x)\in H$ for all $x,y\in H$. If furthermore $\mu_i(x,y)\in H$ and $\mu_i(y,x)\in H$ for all $(x,y)\in A\times H,$ then $H$ is called a two-sided BiHom-ideal of $A$.
  \item $(A, \mu_1, \dots, \mu_n,\alpha,\beta)$ is said to be regular if $\alpha$ and $\beta$ are algebra automorphisms.
  \item $(A, \mu_1, \dots, \mu_n,\alpha,\beta)$ is said to be involutive if $\alpha$ and $\beta$ are two involutions,  that is $\alpha^2=\beta^2=id$.
\end{enumerate}
 \end{defn}

\begin{defn}
Let $(A, \mu_1, \dots, \mu_n,\alpha,\beta)$ and $(A', \mu'_1, \dots, \mu'_n,\alpha',\beta')$ be two BiHom- algebras. Then
a linear map $f:A\longrightarrow A^{'}$ is said to be a BiHom-algebras morphism if the following conditions hold:
$$\begin{array}{llll}&&f  \circ \mu_i= \mu_i^{'} \circ(f \otimes f),~\forall i=1,\dots,n,\\&&f\circ \alpha=\alpha^{'} \circ f,\\&&f\circ \beta=\beta^{'}\circ  f,\end{array} $$
as illustrated, respectively,  by the following commutative diagrams:
  $$
\xymatrix{
A\otimes A \ar[d]_{f\otimes f }\ar[rr]^{\mu_i}
               && A  \ar[d]^{f}  \\
A'\otimes A' \ar[rr]^{\mu'_i}
               && A' },\quad \xymatrix{
 A \ar[d]_{ f }\ar[rr]^{\alpha}
               && A  \ar[d]^{f}  \\
 A' \ar[rr]^{\alpha'}
              && A'  },\quad \xymatrix{
A \ar[d]_{ f }\ar[rr]^{\beta}
              && A  \ar[d]^{f}  \\
A' \ar[rr]^{\beta'}
             && A',  }
$$
for all $i=1,\dots,n$.
\end{defn}
Denote by $\Gamma_{f}=\{x+f(x);~~x\in A\}\subset A\oplus A'$  the graph of a linear map $f:A\longrightarrow A^{'}$.

\begin{defn}
A BiHom-algebra $(A, \mu_1, \dots, \mu_n,\alpha,\beta)$ is called a BiHom-X algebra if the axioms defining the structure of $X$ are linear combination of the terms
of the form
$\mu_j\circ(\mu_i\otimes \beta)$ or $\mu_j\circ(\alpha\otimes\mu_i)$.
\end{defn}

\begin{prop}\label{bylv}
Let $(A,\mu_1^{A}, \dots, \mu_n^{A},\alpha_{1},\beta_{1})$ and $(B,\mu_1^{B}, \dots, \mu_n^{B},\alpha_{2},\beta_{2})$ be BiHom-X algebras. Then, there is a BiHom-X algebra $(A\oplus B, \mu_1^{A\oplus B},\dots, \mu_n^{A\oplus B}, \alpha,\beta),$ where for all $i=1,\dots,n,$ the bilinear maps $\mu_{i}^{A\oplus B}:(A\oplus B)^{\times 2}\longrightarrow (A\oplus B)$ are given by
 $$
 \mu_{i}^{A\oplus B}(a_1+b_1,a_2+b_2):=\mu_{i}^A(a_1,a_2)+\mu_{i}^B (b_1,b_2), \forall \ a_1,a_2\in A,\ \forall \ b_1,b_2\in B,
$$
 and the linear  maps $\alpha$ and $\beta$ are given, for all $(a,b)\in A\times B$, by
$$ \begin{array}{lll}
\alpha(a+b)&:=& \alpha_{1}(a)+\alpha_{2}(b),\\
\beta(a+b)&:=& \beta_{1}(a)+\beta_{2}(b).
 \end{array}$$
 \end{prop}
\begin{proof}
For any $a_1,b_1,c_1\in A$, $a_2,b_2,c_2\in B$ and $1\leq i,j \leq n$,
\begin{multline*}
\mu_{i}^{A\oplus B}(\mu_{j}^{A\oplus B}(a_1+a_2,b_1+b_2),\beta(c_1+c_2))= \\
\mu_{i}^{A\oplus B}(\mu_{j}^{A}(a_1,b_1)+\mu_{j}^{B}(a_2,b_2),\beta_1(c_1)+\beta_2(c_2))\\
=\mu_{i}^{A}(\mu_{j}^{A}(a_1,b_1),\beta_1(c_1))+\mu_{i}^{B}(\mu_{j}^{B}(a_2,b_2),\beta_2(c_2))
\end{multline*}
Similarly,
\begin{align*}
\mu_{i}^{A\oplus B}(\alpha(a_1+a_2),&\mu_{j}^{A\oplus B}(b_1+b_2,c_1+c_2))= \\ &\mu_{i}^{A}(\alpha_1(a_1),\mu_{j}^{A}(b_1,c_1))+\mu_{i}^{B}(\alpha_2(a_2),\mu_{j}^{B}(b_2,c_2)).
\qedhere
\end{align*}
\end{proof}

\begin{prop}
Let $(A,\mu_{1}^A,\dots,\mu_{n}^A,\alpha_{1},\beta_{1})$ and $(B,\mu_{1}^{B},\dots,\mu_{n}^{B},\alpha_{2},\beta_{2})$  be BiHom-X algebras.
Then a linear map
$\varphi: A\rightarrow B$ is a morphism from the BiHom-X algebra $(A,\mu_{1}^A,\dots,\mu_{n}^A,\alpha_{1},\beta_{1})$ to the BiHom-X algebra $(B,\mu_{1}^{B},\dots,\mu_{n}^{B},\alpha_{2},\beta_{2})$ if and only if its graph
$\Gamma_{\varphi}$ is a BiHom-X subalgebra of $(A\oplus B, \mu_{1}^{A\oplus B},\dots,\mu_{n}^{A\oplus B}, \alpha_{1}+\beta_{1},\alpha_{2}+\beta_{2}).$
\end{prop}
\begin{proof}
  Let $\varphi: (A,\mu_{1}^A,\dots,\mu_{n}^A,\alpha_{1},\beta_{1})\longrightarrow (B,\mu_{1}^{B},\dots,\mu_{n}^{B},\alpha_{2},\beta_{2})$ be a morphism of BiHom-X algebras.
Then for all $u, v\in A$ and $1\leq i \leq n$,
$$ \mu_{i}^{A\oplus B}((u+\varphi(u),v+\varphi(v))=(\mu_{i}^A(u,v)+\mu_{i}^B(\varphi(u),\varphi(v)))=(\mu_{i}^A(u,v)+\varphi(\mu_{i}^A(u,v))).$$
Thus the graph $\Gamma_{\varphi}$ is closed under the multiplication $\mu_{i}^{A\oplus B}.$
Furthermore, $\varphi\circ\alpha_{1}=\alpha_{2}\circ\varphi$ yields $(\alpha_{1}\oplus\alpha_{2})(u, \varphi(u)) = (\alpha_{1}(u),
\alpha_{2}\circ\varphi(u)) = (\alpha_{1}(u), \varphi\circ\alpha_{1}(u)).$ In the same way, $(\beta_{1}\oplus\beta_{2})(u, \varphi(u)) = (\beta_{1}(u),
\beta_{2}\circ\varphi(u)) = (\beta_{1}(u), \varphi\circ\beta_{1}(u)),$
which implies that $\Gamma_{\varphi}$ is closed under $\alpha_{1}\oplus\alpha_{2}$ and $\beta_{1}\oplus\beta_{2}$. Thus $\Gamma_{\varphi}$ is a BiHom-X subalgebra of
$(A\oplus B, \mu_{1}^{A\oplus B},\dots,\mu_{n}^{A\oplus B}, \alpha_{1}+\beta_{1},\alpha_{2}+\beta_{2}).$

Conversely, if the graph $\Gamma_{\varphi}\subset A\oplus B$ is a
BiHom-X subalgebra of
$$(A\oplus B, \mu_{1}^{A\oplus B},\dots,\mu_{n}^{A\oplus B}, \alpha_{1}+\beta_{1},\alpha_{2}+\beta_{2}),$$
then for all $1\leq i \leq n$,
$$\mu_{i}^{A\oplus B}((u+ \varphi(u)), (v+ \varphi(v)))=(\mu_{i}^A (u, v) + \mu_{i}^ B(\varphi(u), \varphi(v)) )\in\Gamma_{\varphi},$$
which implies that
$$\mu_{i}^ B(\varphi(u), \varphi(v))=\varphi(\mu_{i}^ A(u, v)).$$
Furthermore, $(\alpha_{1}\oplus\alpha_{2})(\Gamma_{\varphi})\subset\Gamma_{\varphi},~(\beta_{1}\oplus\beta_{2})(\Gamma_{\varphi})\subset\Gamma_{\varphi}$ implies
$$(\alpha_{1}\oplus\alpha_{2})(u+\varphi(u))=(\alpha_{1}(u)+ \alpha_{2}\circ\varphi(u)) \in\Gamma_{\varphi},$$
$$(\beta_{1}\oplus\beta_{2})(u+ \varphi(u))=(\beta_{1}(u)+ \beta_{2}\circ\varphi(u)) \in\Gamma_{\varphi},$$
equivalent to the conditions $\alpha_2\circ\varphi(u)=\varphi\circ\alpha_1(u)$ and $\beta_2\circ\varphi(u)=\varphi\circ\beta_1(u),$ that is $ \alpha_1\circ\varphi=\varphi\circ\alpha_2$ and $\beta_1\circ\varphi=\varphi\circ\beta_2.$ Therefore, $\varphi$ is a
morphism of BiHom-X algebras.
\end{proof}
\begin{thm}\label{bk8}
Let $(A_1,\mu_{1}^{A_1},\dots,\mu_{n}^{A_1},\alpha_{1},\beta_{1})$ and $(A_2,\mu_{1}^{A_2},\dots,\mu_{n}^{A_2},\alpha_{2},\beta_{2})$ be some BiHom-X algebras.
Then $A=A_1\otimes A_2$ is endowed with a BiHom-X algebra structure for twising maps
$\alpha,\beta: A\rightarrow A$ and the product $\ast_i : A\otimes A\rightarrow A$ defined
for any $a_1,b_1,c_1\in A_1$, $a_2,b_2,c_2\in A_2$ and $1\leq i\leq n$ by
\begin{eqnarray*}
 \alpha(a_1\otimes a_2)&:=&\alpha_1(a_1)\otimes\alpha_2(a_2),\\
  \beta(a_1\otimes a_2)&:=&\beta_1(a_1)\otimes\beta_2(a_2),\\
(a_1\otimes a_2)\ast_{i}(b_1\otimes b_2)&:=&\mu_i^{A_1}(a_1,b_1)\otimes \mu_i^{A_2}(a_2, b_2).
\end{eqnarray*}
\end{thm}
\begin{proof}
For any $a_1,b_1,c_1\in A_1$, $a_2,b_2,c_2\in A_2$ and $1\leq i,j \leq n$,
\begin{multline*}
((a_1\otimes a_2)\ast_i(b_1\otimes b_2))\ast_j\beta(c_1\otimes c_2)=\\
(\mu_{i}^{A_1}(a_1,b_1)\otimes\mu_{i}^{A_{2}}(a_2,b_2))\ast_{j}\beta_1(c_1)\otimes \beta_{2}(c_2)\\
=\mu_{j}^{A_{1}}(\mu_{i}^{A_1}(a_1,b_1),\beta_1(c_1))\otimes\mu_{j}^{A_2}(\mu_{i}^{A_{2}}(a_2,b_2)),\beta_{2}(c_2)).
\end{multline*}
Similarly,
\begin{align*}
\alpha(a_1\otimes a_2)\ast_j((b_1\otimes b_2)\ast_i & (c_1\otimes c_2))= \\
&\mu_{j}^{A_{1}}(\alpha_1(a_1),\mu_{i}^{A_1}(b_1,c_1))\otimes\mu_{j}^{A_2}
(\alpha_2(a_2),\mu_{i}^{A_2}(b_2,c_2)).
\qedhere \end{align*}
\end{proof}
\begin{defn}
 Let $(A, \mu_1, \dots, \mu_n,\alpha,\beta)$ be a BiHom-algebra and $k\in\mathbb{N}^*$.
\begin{enumerate}
 \item [1)] The $kth$ derived
BiHom-algebra of type $1$ of $A$ is defined by
\begin{eqnarray}
 A_1^k=(A,  \mu^{(k)}_1=\mu_1\circ(\alpha^k\otimes\beta^{k}), \dots, \mu^{(k)}_n=\mu_n\circ(\alpha^k\otimes\beta^{k}),\alpha^{k+1},\beta^{k+1}).
\end{eqnarray}
\item [2)] The $kth$ derived BiHom-algebra of type $2$ of $A$ is defined by
\begin{equation}
\begin{array}{rl}
& A_2^k =(A,  \mu^{(2^k-1)}_1=\mu_1\circ(\alpha^{2^k-1}\otimes\beta^{2^k-1}), \dots, \\
 &\quad \quad \quad \quad \quad \mu^{(2^k-1)}_n =\mu_n\circ(\alpha^{2^k-1}\otimes\beta^{2^k-1}), \alpha^{2^k},\beta^{2^{k}}).
\end{array}
\end{equation}
\end{enumerate}
Note that $A_1^0=A_2^0=(A, \mu_1, \dots, \mu_n,\alpha,\beta)$ and
 $A^1_1=A_2^1=(A, \mu_1\circ(\alpha\otimes \beta), \dots,\mu_n\circ(\alpha\otimes \beta), \alpha^{2},\beta^{2})$.
 \end{defn}

\begin{defn}
 A BiHom-algebra $(A, \mu_1, \dots, \mu_n,\alpha,\beta)$ endowed with a linear map $R : A\rightarrow A$ such that $\alpha\circ R=R\circ\alpha,~~\beta\circ R=R\circ \beta$ and
\begin{eqnarray}
 \mu_i(R(x), R(y))=R\Big(\mu_i(R(x), y)+\mu_i(x, R(y))+\lambda\mu_i(x, y)\Big),\;\; i=1, \dots, n, \label{rb}
\end{eqnarray}
with $\lambda\in\mathbb{K}, x, y\in A$, is called a Rota-Baxter BiHom-algebra, and $R$ is called a Rota-Baxter operator of weight $\lambda$ on $A$.
\end{defn}
The below result allows to get BiHom-X algebras from either a BiHom-X algebra or an $X$-algebra.
\begin{thm}\label{gftp1}
Let $(A, \mu_1, \dots, \mu_n,\alpha,\beta)$ be a Rota-Baxter BiHom-X algebra and let
$\alpha',\beta' : A\rightarrow A$ be two endomorphisms of $A$ such that any two of the maps $\alpha,\beta,\alpha',\beta'$ commute.
Then, for any nonnegative integer $p$,
$$A_{\alpha^{'},\beta'}=(A, \mu_{\alpha',\beta'}^1=\mu_1\circ(\alpha'^{p}\otimes\beta'^{p}), \dots, \mu_{\alpha',\beta'}^n=\mu_n\circ(\alpha'^{p}\otimes\beta'^{p}), \alpha'^{p}\circ\alpha,\beta'^{p}\circ\beta)$$
is a Rota-Baxter BiHom-X algebra.

Moreover, let $(A', \mu'_1, \dots, \mu'_n,\gamma,\delta)$ be another BiHom-X algebra
and $\gamma',\delta' : A'\rightarrow A'$ be two endomorphisms such that any two of the maps $\gamma,\delta,\gamma',\delta'$ commute.
 If $f : A\rightarrow A'$ is a morphism of BiHom-X algebras that satisfies
 $f\circ\alpha'=\gamma'\circ  f,~f\circ\beta'=\delta'\circ  f$,
then $f : A_{\alpha',\beta'}\rightarrow A'_{\gamma',\delta'}$ is also a morphism of BiHom-X algebras.
\end{thm}
\begin{proof}
The proof of the first part follows from the following facts.

For any $x, y, z\in A, \; 1\leq i, j\leq n$,
\begin{eqnarray*}
\mu_{\alpha',\beta'}^i(\mu_{\alpha',\beta'}^j(x, y), (\beta'^{p}\circ\beta)(z))
&=&\mu_{\alpha',\beta'}^i(\mu_{\alpha',\beta'}^j(x, y), \beta'^{p}(\beta(z)))\\
&=&\mu_i\Big(\alpha'^{p}\mu_j(\alpha'^{p}(x),\beta'^{p}( y)), \beta'^{2p}(\beta(z))\Big) \\
&=&\mu_i(\mu_j(\alpha'^{2p}(x), \alpha'^{p}\beta'^{p}(y)), \beta(\beta'^{2p}(z)) \\
&=&\mu_i(\mu_j(X, Y), \beta(Z)),\\
\mu_{\alpha',\beta'}^i(\alpha'^{p}\circ\alpha(x), \mu_{\alpha',\beta'}^j(y, z))
&=&   \mu_{\alpha',\beta'}^i(\alpha(\alpha'^{p}(x)), \mu_{\alpha',\beta'}^j(y, z))\\
&=& \mu_i(\alpha'^{2p}\circ\alpha(x), \beta'^{p}(\mu_j(\alpha'^{p}(y), \beta'^{p}(z)))\\
&=& \mu_i(\alpha(\alpha'^{2p}(x)), \beta'^{p}(\mu_j(\alpha'^{p}(y), \beta'^{p}(z)))\\
&=& \mu_i(\alpha(\alpha'^{2p}(x)), \mu_j(\alpha'^{p}\beta'^{p}(y), \beta'^{2p}(z))\\
&=& \mu_i(\alpha(X), \mu_j(Y, Z)),
\end{eqnarray*}
where $X=\alpha'^{2p}(x),~Y=\alpha'^{p}\beta'^{p}(y)$ and $Z=\beta'^{2p}(z)$.

The Rota-Baxter identity (\ref{rb}) for $\mu_{\alpha',\beta'}^i$ is proved by
\begin{multline*}
\mu_{\alpha',\beta'}^i(R(x), R(y))=\mu_i(\alpha'^{p}(R(x)), \beta'^{p}(R(y)))
=\mu_i(R(\alpha'^{p}(x)), R(\beta'^{p}(y)))\\
=R\Big(\mu_i(R(\alpha'^{p}(x)), \beta'^{p}(y))+\mu_i(\alpha'^{p}(x), R(\beta'^{p}(y)))+\lambda\mu_i(\alpha'^{p}(x), \beta'^{p}(y))\Big)\\
=R\Big(\mu_{\alpha',\beta'}^i(R(x), y)+\mu_{\alpha',\beta'}^i(x, R(y))+\lambda\mu_{\alpha',\beta'}^i(x, y)\Big).
\end{multline*}
The second assertion follows from
\begin{align*}
f(\mu_{\alpha',\beta'}^i(x, y))&=f(\mu_i(\alpha'^{p}(x), \beta'^{p}(y)))=\mu'_i(f(\alpha'^{p}(x)), f(\beta'^{p}(y)))\\
&=\mu'_i(\gamma'^{p}(f(x)), \delta'^{p}(f(y)))=\mu_{\gamma',\delta'}^i(f(x), f(y)).
\qedhere
\end{align*}
\end{proof}
We have the following series of consequence of Theorem \ref{gftp1}.
\begin{cor}
Let $(A, \mu_1, \dots, \mu_n)$ be an $X$-algebra and $\alpha,\beta : A\rightarrow A$ be two endomorphisms of $A$.
Then $A_{\alpha,\beta}= (A, \mu_1\circ(\alpha\otimes\beta), \dots, \mu_n(\alpha\otimes\beta),\alpha,\beta)$ is a multiplicative  BiHom-X algebra.
\end{cor}
\begin{proof}
 Take $\alpha=\beta=Id$ and $p=1$ in Theorem \ref{gftp1}.
\end{proof}
\begin{cor}
Let $(A, \mu_1, \dots, \mu_n,\alpha,\beta)$ be a BiHom-X algebra. Then the $k$th derived BiHom-algebra of type $1$ and
the $k$th derived BiHom-algebra of type $2$ are  BiHom-X algebras.
\end{cor}
\begin{proof}
 It is sufficient to take $\alpha'=\alpha,~\beta'=\beta$, and  $p={k}$ and $p={2^k-1}$ respectively in Theorem \ref{gftp1}.
\end{proof}

Now we introduce the notion of centroids for BiHom-X algebras.
 \begin{defn}
 A Centroid of a BiHom-algebra $(A, \mu_1, \dots, \mu_n,\alpha,\beta)$ is a linear map $\gamma : A\rightarrow A$ such that $\gamma\circ\alpha=\alpha\circ\gamma$,
 $\gamma\circ\beta=\beta\circ\gamma$ and for any $1\leq i\leq n$ and $x, y\in A,$
 $$\gamma(\mu_i(x, y))=\mu_i(\gamma(x), y)=\mu_i(x, \gamma(y)).$$
\end{defn}

\begin{thm}
 Let $(A, \mu_1, \dots, \mu_n,R, \alpha,\beta)$ be a Rota-Baxter  BiHom-X algebra and $\gamma_1, \gamma_2 :A\rightarrow A$ be
 a pair of commuting elements of the centroid such that $$\gamma_i\circ R=R\circ\gamma_i, i=1,2.$$
Define bilinear maps $\mu^i_{\gamma} : A\times A\rightarrow A, i=1, \dots, n$ for any $x, y\in A$ by
$$\mu^i_{\gamma}(x, y):=\mu_i(\gamma_2\gamma_1(x), y).$$
Then
$A_{\gamma_1, \gamma_2}=(A, \mu^1_\gamma, \dots, \mu^n_\gamma,R, \gamma_1\alpha,\gamma_2\beta)$ is also a Rota-Baxter BiHom-X algebra.
\end{thm}
\begin{proof}
For any $x, y\in A$ and  $1\leq i,j\leq n$,
 \begin{align*}
\mu^i_\gamma(\mu^j_{\gamma}(x, y), \gamma_2\beta(z)) &= \mu^i_\gamma(\mu_j(\gamma_1\gamma_2(x),y), \gamma_2\beta(z)) \\
&=\mu_i(\gamma_1\gamma_2\mu_j(\gamma_1\gamma_2(x),y), \gamma_2\beta(z)) \\
&=\gamma^{2}_1\gamma^{3}_2\mu_i(\mu_j(x,y),\beta(z)), \\
\mu^i_\gamma(\gamma_1\alpha(x),\mu^j_{\gamma}(y, z)) &= \mu^i_{\gamma}(\gamma_1\alpha(x),\mu_j(\gamma_1\gamma_2(y),z)) \\
&=\mu_i(\gamma_1^{2}\gamma_2\alpha(x),\mu_j(\gamma_1\gamma_2(y),z)) \\
&=\gamma^{2}_1\gamma^{3}_2\mu_i(\alpha(x),\mu_j(y,z)),  \\
\mu^i_\gamma(R(x), R(y))
&= \mu_i(\gamma_1\gamma_2(R(x)), R(y))=\gamma_1\gamma_2\mu_i(R(x), R(y)) \\
&=\gamma_1\gamma_2R\Big(\mu_i(R(x), y)+\mu_i(x, R(y))+\lambda\mu_i(x, y)\Big) \\
&= R\Big(\gamma_1\gamma_2\mu_i(R(x), y)+\gamma_1\gamma_2\mu_i(x, R(y))+\lambda\gamma_1\gamma_2\mu_i(x, y)\Big) \\
&= R\Big(\mu_i(\gamma_1\gamma_2R(x), y)+\mu_i(\gamma_1\gamma_2(x), R(y))+\lambda\mu_i(\gamma_1\gamma_2(x), y)\Big) \\
&= R\Big(\mu^i_\gamma(R(x), y)+\mu^i_\gamma(x, R(y))+\lambda\mu^i_\gamma(x, y)\Big) .
\qedhere  \end{align*}
\end{proof}

\section{Bimodules and matched pairs of BiHom-left symmetric and BiHom-associative dialgebras}
\label{sec:homleftsymcolordialg}
In this section, we recall definitions and some key results about bimodules of BiHom-associative algebras \cite{GRAZIANI} and BiHom-left-symmetric algebras \cite{BenHassineChtiouiMabroukNcib19:CohomLiedeformBiHomleftsym}. Next, we introduce the notions of BiHom-left-symmetric dialgebra and BiHom-associative dialgebra and we give some related relevant properties.
\begin{defn}
A BiHom-module is a triple $(V,\alpha_V,\beta_V)$ consisting of a $\mathbb{K}$-vector space $V$ and
two linear maps $\alpha_V, \beta_V: V\longrightarrow V$  such that $\alpha_V\beta_V=\beta_V\alpha_V.$ A morphism
$f: (V,\alpha_V, \beta_V)\rightarrow (W,\alpha_W,\beta_W)$ of BiHom-modules is a linear map
 $f: V\longrightarrow W$ such that $f\alpha_V=\alpha_W f$ and
 $f\beta_V=\beta_W f.$
\end{defn}
\begin{defn}
A BiHom-associative algebra is a quadruple $(A,\mu,\alpha,\beta)$
consisting of a vector space $A$ on which the operation $\mu: A\otimes A\rightarrow A$ and $\alpha, \beta: A\rightarrow A$ are linear maps satisfying, for any $x, y, z \in A ,$
\begin{eqnarray}
\alpha\circ\beta &=& \beta\circ\alpha,\label{aca0}\\
\alpha\circ\mu(x, y)&=&\mu(\alpha(x),\alpha(y)),\label{aca0.0}\\
\beta\circ\mu(x, y)&=&\mu(\beta(x),\beta(y)),\label{aca0.00}\\
\mu(\alpha(x), \mu(y, z))&=&\mu(\mu(x, y), \beta(z)). \label{aca}
\end{eqnarray}
\end{defn}
\begin{rems} Clearly, there are the following connections between Hom-associative, BiHom-associative and
BiHom-$X$ algebra structures.
\begin{enumerate}
\item
A Hom-associative algebra $(A,\mu,\alpha)$ can be regarded as a BiHom-associative
algebra $(A,\mu,\alpha,\alpha)$.
\item
A BiHom-associative algebra is a BiHom-$X$ algebra.
\end{enumerate}
\end{rems}
\begin{ex}
Let $\{e_1,e_2\}$ be a basis of a $2$-dimensional vector space $A$ over $\mathbb{K}$. The following multiplication $\mu$ and maps $\alpha,\beta$ on $A$ define a BiHom-associative algebra:
\begin{alignat*}{4}
  \alpha(e_1)&=2e_1, & \qquad \alpha(e_2)&=-2a_1e_1+(1-a)e_2,\\
  \beta(e_1)&=2e_1, & \qquad \beta(e_2)&=-ae_1+(1-a) e_2, \\
  \mu(e_1,e_1)&=2e_1, & \qquad \mu(e_1,e_2)&= -ae_1+(1-a)e_2, \\
  \mu(e_2,e_1)&=-2ae_1+(a-1)e_2, & \qquad  \mu(e_2,e_2)&=2a^2 e_1+ae_2.
\end{alignat*}
where $a\in\mathbb{K}\backslash\{0\}$.
\end{ex}
\begin{defn}
Let $(A, \cdot, \alpha_{1}, \alpha_{2})$ be a BiHom-associative algebra, and let $(V, \beta_{1}, \beta_{2})$ be a BiHom-module. Let $ l, r: A \rightarrow gl(V) $ be two linear maps. The quintuple $(l, r, \beta_{1}, \beta_{2}, V)$ is called a bimodule of $A$ if, for all $ x, y \in  A, v \in V $,
$$\begin{array}{llllllll}
 l(x\cdot y)\beta_{2}(v)&=& l(\alpha_{1}(x))l(y)v,&& r(x\cdot y)\beta_{1}(v)&=& r(\alpha_{2}(y))r(x)v,\\
 l(\alpha_{1}(x))r(y)v &=& r(\alpha_{2}(y))l(x)v,&&
\beta_{1}(l(x)v)&=& l(\alpha_{1}(x))\beta_{1}(v),\\ \beta_{1}(r(x)v)&=& r(\alpha_{1}(x))\beta_{1}(v),&&
\beta_{2}(l(x)v) &=& l(\alpha_{2}(x))\beta_{2}(v),\\ \beta_{2}(r(x)v)&=& r(\alpha_{2}(x))\beta_{2}(v).
\end{array}$$

\end{defn}
\begin{prop}
Let $(l, r, \beta_{1}, \beta_{2}, V)$ be a bimodule of a BiHom-associative algebra $(A, \cdot, \alpha_{1}, \alpha_{2})$. Then, the direct sum $A \oplus V$ of vector spaces is a BiHom-associative algebra  with multiplication in $A\oplus V $, defined for all $ x_{1}, x_{2} \in  A, v_{1}, v_{2} \in V$, by
\begin{eqnarray*}
(x_{1} + v_{1}) \ast (x_{2} + v_{2})&:=& x_{1} \cdot x_{2} + (l(x_{1})v_{2} + r(x_{2})v_{1}),\cr
(\alpha_{1}\oplus\beta_{1})(x_{1} + v_{1})&:=& \alpha_{1}(x_{1}) + \beta_{1}(v_{1}),\cr (\alpha_{2}\oplus\beta_{2})(x_{1} + v_{1})&:=& \alpha_{2}(x_{1}) + \beta_{2}(v_{1}).
\end{eqnarray*}
\end{prop}
We denote such BiHom-associative algebra by $(A \oplus V, \ast, \alpha_{1} + \beta_{1}, \alpha_{2} + \beta_{2}),$ or $A \times_{l, r, \alpha_{1}, \alpha_{2}, \beta_{1}, \beta_{2}} V.$
\begin{ex}
Let $(A,\cdot,\alpha,\beta)$ be a BiHom-associative algebra. Then $(L,0,\alpha,\beta,A)$, $(0,R,\alpha,\beta,A)$ and $(L,R,\alpha,\beta,A)$ are bimodules of $(A,\cdot,\alpha,\beta)$, where $L(a)b=a\cdot b$ and $R(a)b=b\cdot a$ for all $a,b\in A$.
\end{ex}
\begin{thm}[\cite{HounkonnouHoundedjiSilvestrov:DoubleconstrbiHomFrobalg}]
\label{matched ass}
Let $(A,\cdot_A,\alpha_1,\alpha_2)$ and $(B,\cdot_B,\beta_1,\beta_2)$ be two BiHom-associative algebras. Suppose that there are linear maps $l_A,r_A:A\rightarrow gl(B)$
and $l_B,r_B:B\rightarrow gl(A)$ such that
$(l_A,r_A,\beta_1,\beta_2,B) \ \mbox{is a bimodule of}\ A,
(l_B,r_B,\alpha_1,\alpha_2,A) \ \mbox{is a bimodule of}\ B,$
and for any $x,y\in A,~a,b\in B$,
\begin{align}
\label{3}
     l_A(\alpha_1(x))(a\cdot_B b)&=l_A(r_B(a)x)\beta_2(b)+(l_A(x)a)\cdot_B\beta_2(b),
\\
\label{4}
     r_A(\alpha_2(x))(a\cdot_B b)&=r_A(l_B(b)x)\beta_1(a)+\beta_1(a)\cdot_B(r_A(x)b),
\\
\label{5}
     l_A(l_B(a)x)\beta_2(b)&+(r_A(x)a)\cdot_B\beta_2(b)\nonumber \\
     &-r_A(r_B(b)x)\beta_1(a) -\beta_1(a)\cdot_B(l_A(x)b)=0,
\\
\label{6}
     l_B(\beta_1(a))(x\cdot_A y)&=l_B(r_A(x)a)\alpha_2(y)+(l_B(a)x)\cdot_A\alpha_2(y),
\\
\label{7}
     r_B(\beta_2(a))(x\cdot_A y)&=r_B(l_A(y)a)\alpha_1(x)+\alpha_1(x)\cdot_A(r_B(a)y),
\\
\label{8}
     l_B(l_A(x)a)\alpha_2(y)&+(r_B(a)x)\cdot_A\alpha_2(y)\nonumber \\
     &-r_B(r_A(y)a)\alpha_1(x)-\alpha_1(x)\cdot_A(l_B(a)y)=0.
\end{align}
Then $(A,B,l_A,r_A,\beta_1,\beta_2,l_B,r_B,\alpha_1,\alpha_2)$ is called a matched pair of
BiHom-associative algebras. In this case, there is a BiHom-associative algebra structure on the direct sum
$A\oplus B$ of the underlying vector spaces of $A$ and $B$ given by
$$\begin{array}{llllll}
(x + a) \cdot (y + b)&:=&x \cdot_A y + (l_A(x)b + r_A(y)a)+a \cdot_B b + (l_B(a)y + r_B(b)x),\cr
(\alpha_{1}\oplus\beta_{1})(x + a)&:=&\alpha_{1}(x) + \beta_{1}(a),\cr (\alpha_{2}\oplus\beta_{2})(x + a)&:=&\alpha_{2}(x) + \beta_{2}(a).
\end{array}$$
\end{thm}
\begin{proof}
For any $x,y,z\in A$ and $a,b,c\in B$,
\begin{align*}
&(\alpha_1+\beta_1)(x+a)\cdot((y+b)\cdot(z+c))\\
&\quad =(\alpha_1(x)+\beta_1(a))[y\cdot_A z+l_B(b)z+r_B(c)y+b\cdot c+l_A(y)c+r_A(z)b)\\
&\quad =\alpha_1(x)\cdot_A(y\cdot_A z)+\alpha_1(x)\cdot_A l_B(b)z+\alpha_1(x)\cdot_A r_B(c)y+l_B(\beta_1(a))(y\cdot_A z)\\
&\quad \quad +l_B(\beta_1(a))l_B(b)z+l_B(\beta_1(a))r_B(c)y+r_B(b\cdot_B c)\alpha_1(x)+r_B(l_A(y)c)\alpha_1(x)\\
&\quad \quad +r_B(r_A(z)b)\alpha_1(x)+\beta_1(a)\cdot_B(b\cdot_B c)+\beta_1(a)\cdot_B l_A(y)c+\beta_A(\alpha_1(x))l_A(y)c\\
&\quad \quad +l_A(\alpha_1(x))r_A(z)b+r_A(y\cdot_A z)\beta_1(a)+r_A(l_A(b)z)\beta_1(a)+r_A(r_B(c)y)\beta_1(a),
\\
&((x+a)\cdot(y+b))\cdot(\alpha_2+\beta_2)(z+c)\\
&\quad=(x\cdot_A y+l_B(a)y+r_B(b)x+a\cdot_B b+l_A(x)b+r_A(y)a)\cdot(\alpha_2(z)+\beta_2(c))\\
&\quad=(x\cdot_A y)\cdot_A\alpha_2(z)+l_B(a)y\cdot_A\alpha_2(z)+r_B(b)x\cdot_A\alpha_2(z)+l_B(a\cdot_B b)\alpha_2(z)\\
&\quad\quad+l_B(l_A(x)b)\alpha_2(z)+l_B(r_A(y)a)\alpha_2(z)+r_B(\beta_2(c))(x\cdot_A y)+r_A(\beta_2(c))l_B(a)y\\
&\quad\quad +r_B(\beta_2(c))r_B(b)x+(a\cdot_B b)
\cdot_B\beta_2(c)+(l_A(x)b)\cdot_B\beta_2(c)+(r_A(y)a)\cdot_B\beta_2(c)\\
&\quad\quad+r_A(\alpha_2(z))(a\cdot_B b)+r_A(\alpha_2(z))(l_A(x)b)
+r_A(\alpha_2(z))(r_A(y)a)+l_A(x\cdot_A y)\beta_2(c)\\
&\quad\quad+l_A(l_B(a)y)\beta_2(c)+(r_B(b)x)\beta_2(c).
\end{align*}
Then \eqref{aca} and \eqref{3}-\eqref{8} yield
\begin{equation*}
(\alpha_1+\beta_1)(x+a)\cdot((y+b)\cdot(z+c))=((x+a)\cdot(y+b))\cdot(\alpha_2+\beta_2)(z+c).
\qedhere \end{equation*}
\end{proof}
We denote this BiHom-associative algebra by $A\bowtie^{l_A,r_A,\beta_1,\beta_2}_{l_B,r_B,\alpha_1,\alpha_2}B$.\\
\begin{defn}
A BiHom-left-symmetric algebra is a quadruple $(S, \ast, \alpha, \beta)$ consisting of a vector space $S$ on which the operation $\ast: S\otimes S\rightarrow S$ and $\alpha, \beta: S\rightarrow S$ are linear maps satisfying, for all $x, y, z\in S$,
\begin{align}
\alpha\circ\beta &=\beta\circ\alpha,\label{clsa0}\\
\alpha(x\ast y) &=\alpha(x)\ast\alpha(y),\label{clsa0.0}\\
\beta(x\ast y) &=\beta(x)\ast\beta(y),\label{clsa0.00}\\
(\beta(x)\ast \alpha(y))\ast\beta(z)&-\alpha\beta(x)\ast(\alpha(y)\ast z) \nonumber \\
&=(\beta(y)\ast \alpha(x))\ast\beta(z)-\alpha\beta(y)\ast(\alpha(x)\ast z).
\label{clsa}
\end{align}
 \end{defn}
\begin{defn}
Let $(S, \ast, \alpha_{1}, \alpha_{2})$ be a BiHom-left-symmetric algebra, and $(V, \beta_{1}, \beta_{2})$ be a BiHom-module. Let $ l, r: S \rightarrow gl(V) $ be two linear maps. The quintuple $(l, r, \beta_{1}, \beta_{2}, V)$ is called a bimodule of $S$ if for all $ x, y \in  S, v \in V $,
\begin{align*}
l(\alpha_{2}(x)\ast\alpha_{1}(y))\beta_{2}(v)&-l(\alpha_{1}\alpha_{2}(x))l(\alpha_{1}(y))v \\
&=l(\alpha_{2}(y)\ast\alpha_{1}(x))\beta_{2}(v)-l(\alpha_{1}\alpha_{2}(y))l(\alpha_{1}(x))v,\\
r(\alpha_{2}(x))r(\alpha_{1}(y))\beta_{2}(v)&-r(\alpha_{1}(y)\ast x)\alpha_{1}\alpha_{2}(v)\\
&=r(\alpha_{2}(x))l(\alpha_{2}(y))\beta_{1}(v)-l(\alpha_{1}\alpha_{2}(y))r(x)\beta_{1}(v),\\
\beta_{1}(l(x)v)= l(\alpha_{1}(x))\beta_{1}(v),& \quad \beta_{1}(r(x)v)= r(\alpha_{1}(x))\beta_{1}(v),\\
\beta_{2}(l(x)v)= l(\alpha_{2}(x))\beta_{2}(v),& \quad \beta_{2}(r(x)v)= r(\alpha_{2}(x))\beta_{2}(v).
\end{align*}
\end{defn}
\begin{prop}
Let $(l, r, \beta_{1}, \beta_{2}, V)$ be a bimodule of a BiHom-left-symmetric algebra $(S, \ast, \alpha_{1}, \alpha_{2})$. Then, the direct sum $S \oplus V$ of vector spaces is turned into a BiHom-left-symmetric algebra  by defining multiplication in $S \oplus V $ by
\begin{eqnarray*}
&&(x_{1} + v_{1}) \ast' (x_{2} + v_{2}):=x_{1} \ast x_{2} + (l(x_{1})v_{2} + r(x_{2})v_{1}),\cr
&&(\alpha_{1}\oplus\beta_{1})(x_{1} + v_{1}):=\alpha_{1}(x_{1}) + \beta_{1}(v_{1}),\cr&&(\alpha_{2}\oplus\beta_{2})(x_{1} + v_{1}):=\alpha_{2}(x_{1}) + \beta_{2}(v_{1}),
\end{eqnarray*}
for all $ x_{1}, x_{2} \in  S, v_{1}, v_{2} \in V$.
\end{prop}
We denote such a biHom-left-symmetric algebra by $(S\oplus V, \ast, \alpha_{1} + \beta_{1}, \alpha_{2} + \beta_{2}),$ or by
$S\times_{l, r, \alpha_{1}, \alpha_{2}, \beta_{1}, \beta_{2}} V.$

\begin{ex} Let $(S,\ast,\alpha,\beta)$ be a BiHom-left-symmetric algebra. Then $(L,0,\alpha,\beta,S)$, $(0,R,\alpha,\beta,S)$ and $(L,R,\alpha,\beta,S)$ are bimodules of $(S,\ast,\alpha,\beta)$, where $L(a)b=a\ast b$ and $R(a)b=b\ast a$ for all $a,b\in S$.
\end{ex}
\begin{thm}[\cite{Laraiedh1:2021:BimodmtchdprsBihomprepois}]
Let $(A,\ast_A,\alpha_1,\alpha_2)$ and $(B,\ast_{B},\beta_1,\beta_2)$ be two  BiHom-left-symmetric algebras. Suppose that there are linear maps $l_{\ast_A},r_{\ast_A}:A\rightarrow gl(B)$
and $l_{\ast_B},r_{\ast_B}:B\rightarrow gl(A)$ such that
\begin{align*}
&(l_{\ast_A},r_{\ast_A},\beta_1,\beta_2,B) \ \mbox{is a bimodule of}\ A, \\
&(l_{\ast_B},r_{\ast_B},\alpha_1,\alpha_2,A) \ \mbox{ is a bimodule of} \ B
\end{align*}
and, for any $x,y\in A,~a,b\in B$ with
\begin{align*}
\{\alpha_2(x),\alpha_1(y)\}_A&=\alpha_2(x)\ast_A\alpha_1(y)-\alpha_2(y)\ast_A\alpha_1(x),\\ (\rho_A\circ\alpha_2)\beta_1&=(l_{\ast_A}\circ\alpha_2)\beta_1-(r_{\ast_A}\circ\alpha_1)\beta_2,\\
\{\beta_2(a),\beta_1(b)\}_B&=\beta_2(a)\ast_B\beta_1(b)-\beta_2(b)\ast_B\beta_1(a), \\ (\rho_B\circ\beta_2)\alpha_1&=(l_{\ast_B}\circ\beta_2)\alpha_1-(r_{\ast_B}\circ\beta_1)\alpha_2,
\end{align*}
the following equalities hold:
\begin{align} 
    &r_{\ast_A}(\alpha_2(x))\{\beta_2(a),\beta_1(b)\}_B=r_{\ast_A}(l_{\ast_B}(\beta_1(b))x)\beta_1\beta_2(a)\nonumber\\
    &\quad -r_{\ast_A}(l_{\ast_B}(\beta_1(a)x)\beta_1\beta_2(b)+\beta_1\beta_2(a)\ast_Br_{\ast_A}(x)\beta_1(b)\nonumber\\
    &\quad -\beta_1\beta_2(b)\ast_B r_{\ast_A}(x)\beta_1(a),\label{Lie11}\\
    &l_{\ast_A}(\alpha_1\alpha_2(x))(\beta_1(a)\ast_B b)
    =(\rho_A(\alpha_2(x))\beta_1(a))\ast_B\beta_{2}(b)\nonumber\\
    & \quad-l_{\ast_A}(\rho_B(\beta_2(a))\alpha_1(x))\beta_2(b)
    +\beta_1\beta_2(a)\ast_B(l_{\ast_A}(\alpha_1(x))b)\nonumber\\
    &\quad+r_{\ast_A}(r_{\ast_B}(b)\alpha_1(x))\beta_1\beta_2(a),\label{Lie12}\\
    &r_{\ast_B}(\beta_2(a))\{\alpha_2(x),\alpha_1(y)\}_A=
    r_{\ast_B}(l_{\ast_A}(\alpha_1(y))a)\alpha_1\alpha_2(x)\nonumber\\
    &\quad-r_{\ast_B}(l_{\ast_A}(\alpha_1(x)a)\alpha_1\alpha_2(y)+\alpha_1\alpha_2(x)\ast_A r_{\ast_B}(a)\alpha_1(y)\nonumber\\
    &\quad-\alpha_1\alpha_2(y)\ast_A r_{\ast_B}(a)\alpha_1(x),\label{Lie13}\\
    &l_{\ast_B}(\beta_1\beta_2(a))(\alpha_1(x)\ast_A y) =(\rho_B(\beta_2(a))\alpha_1(x))\ast_A\alpha_{2}(y)\nonumber\\
    &\quad-l_{\ast_B}(\rho_A(\alpha_2(x))\beta_1(a))\alpha_2(y)+\alpha_1\alpha_2(x)\ast_A(l_{\ast_B}(\beta_1(a))y)\nonumber\\
    &\quad+r_{\ast_B}(r_{\ast_A}(y)\beta_1(a))\alpha_1\alpha_2(x).\label{Lie14}
\end{align}
Then for $(A,B,l_{\ast_A},r_{\ast_A},\beta_1,\beta_2,l_{\ast_B},r_{\ast_B},\alpha_1,\alpha_2)$, called a matched pair of BiHom-left-symmetric algebras, there exists a BiHom-left-symmetric algebra structure on the vector space $A\oplus B$ of the underlying vector spaces of $A$ and $B$ given by
$$\begin{array}{llllll}
(x + a) \ast (y + b)&:=&x \ast_A y + (l_{\ast_A}(x)b + r_{\ast_A}(y)a)+a \ast_B b + (l_{\ast_B}(a)y + r_{\ast_B}(b)x),\cr
(\alpha_{1}\oplus\beta_{1})(x + a)&:=&\alpha_{1}(x) + \beta_{1}(a),\cr (\alpha_{2}\oplus\beta_{2})(x + a)&:=&\alpha_{2}(x) + \beta_{2}(a).
\end{array}$$
\end{thm}
\begin{proof}
The proof is obtained in a similar way as for Theorem \ref{matched ass}.
\end{proof}
We denote this BiHom-left-symmetric algebra by $A\bowtie^{l_{\ast_A},r_{\ast_A},\beta_1,\beta_2}_{l_{\ast_B},r_{\ast_B},\alpha_1,\alpha_2}B$.\\
\subsection{BiHom-left-symmetric dialgebra}
\begin{defn}\label{gls}
A BiHom-left-symmetric dialgebra is a linear space $S$ equipped with a two bilinear products $\dashv,\vdash : S\times S\rightarrow S$ and two linear maps $\alpha,\beta: S\rightarrow S$ satisfying,
for all $x, y, z\in S$,
\begin{align}
\alpha\circ\beta &=\beta\circ\alpha,\label{als0}\\
\alpha(x\dashv y) &=\alpha(x)\dashv\alpha(y), \alpha(x\vdash y)=\alpha(x)\vdash\alpha(y),\label{als0.0}\\
\beta(x\dashv y) &=\beta(x)\dashv\beta(y), \beta(x\vdash y)=\beta(x)\vdash\beta(y),\label{als0.00}\\
\alpha(x)\dashv(y\dashv z) &=\alpha(x)\dashv(y\vdash z),\label{als1}\\
(x\vdash y)\vdash\beta(z) &=(x\dashv y)\vdash\beta(z),\label{als2}\\
\alpha\beta(x)\dashv(\alpha(y)\dashv z)&-(\beta(x)\dashv \alpha(y))\dashv \beta(z) \nonumber\\
&=\alpha\beta(y)\vdash(\alpha(x)\dashv z)-(\beta(y)\vdash \alpha(x))\dashv \beta(z),
\label{als3}\\
\alpha\beta(x)\vdash(\alpha(y)\vdash z)&-(\beta(x)\vdash \alpha(y))\vdash \beta(z) \nonumber\\
&=\alpha\beta(y)\vdash(\alpha(x)\vdash z)-(\beta(y)\vdash \alpha(x))\vdash \beta(z).
\label{als4}
\end{align}
\end{defn}
\begin{ex}
Any BiHom-associative algebra $(A,\mu,\alpha,\beta)$ is a BiHom-left-symmetric
dialgebra with $\dashv=\vdash=\mu$.
\end{ex}
\begin{rmk}
 Relation \eqref{als4} means that $(S, \vdash, \alpha,\beta)$ is a BiHom-left-symmetric algebra.
 So, any BiHom-left-symmetric dialgebra is a BiHom-left-symmetric algebra.
\end{rmk}
\begin{thm}\label{isma}
 Given two BiHom-left-symmetric
dialgebras $(S_1,\dashv_1, \vdash_1,\alpha_1,\beta_{1})$ and $(S_2,\dashv_2, \vdash_2,\alpha_{2},\beta_{2}),$ there is a BiHom-left-symmetric
dialgebra $(S_1\oplus S_2,\dashv=\dashv_1+\dashv_2, \vdash=\vdash_1+\vdash_2, \alpha_{1}+\alpha_{2},\beta_{1}+\beta_{2}),$ where the bilinear maps $\dashv,\vdash:(S_1\oplus S_2)^{\times 2}\longrightarrow (S_1\oplus S_2)$ are given for all $a_1,a_2\in S_1$ and $b_1,b_2\in S_2$ by
 \begin{align*}
 (a_1+b_1)\dashv(a_2+b_2) & :=a_1\dashv_1a_2+b_1\dashv_2 b_2,\\
  (a_1+b_1)\vdash(a_2+b_2) & :=a_1\vdash_1a_2+b_1\vdash_2 b_2,
 \end{align*}
 and the linear maps $\alpha_{1}+\alpha_{2},~\beta_{1}+\beta_{2}: (S_1\oplus S_2)\longrightarrow (S_1\oplus S_2)$ are given for all $(a,b)\in S_1\times S_2$ by
\begin{align*}
(\alpha_1+\alpha_2)(a+b)&:= \alpha_1(a)+\alpha_2(b),\\
(\beta_1+\beta_2)(a+b)&:= \beta_1(a)+\beta_2(b).
\end{align*}
 \end{thm}
\begin{proof}
We prove only the axiom \eqref{als1}, as others are proved similarly.
For any $a_{1},b_{1},c_{1}\in S_1$ and $a_2, b_2, c_2\in S_2$,
\begin{align*}
(\alpha_1+\alpha_2)(a_1+a_2) &\dashv((b_1+b_2)\vdash(c_1+c_2)) \\
&=(\alpha_1(a_1)+\alpha_2(a_2))\dashv((b_1+b_2)\vdash(c_1+c_2))\\
&=\alpha_1(a_1)\dashv_1(b_1\vdash_1 c_1)+\alpha_2(a_2)\dashv_2(b_2\vdash_2 c_2)\\
&=\alpha_1(a_1)\dashv_1(b_1\dashv_1 c_1)+\alpha_2(a_2)\dashv_2(b_2\dashv_2 c_2)\\
&=(\alpha_1(a_1)+\alpha_2(a_2))\dashv((b_1+b_2)\dashv(c_1+c_2))\\
&=(\alpha_1+\alpha_2)(a_1+a_2)\dashv((b_1+b_2)\dashv(c_1+c_2)).
\qedhere \end{align*}
\end{proof}
\begin{prop}
Let $(S, \dashv, \vdash, \alpha, \beta)$ be a BiHom-left-symmetric dialgebra, and suppose
$\alpha^{2}=\beta^{2}=\alpha\circ\beta=\beta\circ\alpha=id.$ Then, $(S, \dashv,
\vdash, \alpha, \beta)\cong(S,\dashv,\vdash, \beta, \alpha).$
\end{prop}
\begin{proof}
We prove only one axiom, as others are proved similarly.
For any $x, y, z\in S,$
\begin{eqnarray*}
\alpha(x)\dashv(y\dashv z)&=&\alpha(x)\dashv(y\vdash z)\Leftrightarrow\\
\alpha(\alpha\beta(x))\dashv (y\dashv z)&=&\alpha(\alpha\beta(x))\dashv(y\vdash z)\Leftrightarrow\\
\alpha^{2}\beta(x)\dashv(y\dashv z)&=&\alpha^{2}\beta(x)\dashv(y\vdash z)\Leftrightarrow\\
\beta(x)\dashv(y\dashv z)&=&\beta(x)\dashv(y\vdash z).
\end{eqnarray*}
Then $(S, \dashv,\vdash, \alpha, \beta)\cong(S, \dashv,\vdash, \beta, \alpha).$
\end{proof}
\begin{thm}\label{th1.1}
 Let $(S, \dashv, \vdash,\alpha,\beta)$ be a BiHom-left-symmetric dialgebra and $\alpha'\beta' : S\rightarrow S$ be two endomorphisms of
$S$ such that any two of the maps $\alpha,\beta,\alpha',\beta'$ commute. Then, $S_{\alpha',\beta'}=(S, \dashv_{\alpha',\beta'}=\dashv\circ(\alpha'\otimes\beta'), \vdash_{\alpha',\beta'}=\vdash\circ(\alpha'\otimes\beta'), \alpha'\circ\alpha,\beta'\circ\beta)$
is a BiHom-left-symmetric dialgebra.

Moreover, suppose that $(S', \dashv', \vdash',\gamma,\delta)$ is another BiHom-left-symmetric dialgebra
 and $\gamma',\delta' : S'\rightarrow S'$ be two endomorphisms of $S'$ such that any two of the maps $\gamma,\delta,\gamma',\delta'$ commute. If $f : S\rightarrow S'$ is a morphism of  BiHom-left-symmetric dialgebras such that $f\circ\alpha'=\gamma'\circ  f,~f\circ\beta'=\delta'\circ  f$, then $ f : S_{\alpha',\beta'}\rightarrow S'_{\gamma',\delta'}$
is a morphism of BiHom-left-symmetric dialgebras.
\end{thm}
\begin{proof}
We only prove \eqref{als1} in $S_{\alpha',\beta'}$, since the other axioms are proved analogously. For any $x, y, z\in S$,
$$\begin{array}{lllllll}\alpha\alpha'(x)\dashv_{\alpha',\beta'}(y\dashv_{\alpha',\beta'} z)&=&\alpha\alpha(x)\dashv_{\alpha',\beta'}(\alpha'^{2}(y)\dashv\beta'(z))\\
&=&\alpha\alpha'^{2}(x)\dashv(\alpha'\beta'(y)\dashv\beta'^{2}(z))
\\
&=&\alpha\alpha'^{2}(x)\dashv(\alpha'\beta'(y)\vdash\beta'^{2}(z))\quad (by~\textsc{\eqref{als1}}~in~S)\\
&=&\alpha\alpha'(x)\dashv_{\alpha',\beta'}(\alpha'(y)\vdash\beta'(z))\\
&=&\alpha\alpha'(x)\dashv_{\alpha',\beta'}(y\vdash_{\alpha',\beta'} z).\end{array}$$
For the second assertion, for any $x,y\in S$,
\begin{eqnarray*}
f(x\dashv_{\alpha',\beta'} y)&=&f(\alpha'(x)\dashv\beta'(y))=f(\alpha'(x))\dashv' f(\beta'(y))\\
&=&\gamma'(f(x))\dashv'\delta'(f(y))=f(x)\dashv'_{\gamma',\delta'} f(y).
\end{eqnarray*}
Similarly, $f(x\vdash_{\alpha',\beta'} y)=f(x)\vdash'_{\gamma',\delta'} f(y)$.
\end{proof}

\begin{defn}
Let $(S, \dashv, \vdash, \alpha_{1}, \alpha_{2})$ be a BiHom-left-symmetric
dialgebra, and $V$  be a vector space.
Let $l_{\dashv}, r_{\dashv}, l_{\vdash}, r_{\vdash} :S \rightarrow gl(V),$ and $\beta_{1}, \beta_{2}: V \rightarrow V$. Then, $ (l_{\dashv}, r_{\dashv}, l_{\vdash}, r_{\vdash}, \beta_{1}, \beta_{2}, V)$ is called a bimodule of $ S $
if the following equations hold for any $ x, y \in S $ and $v\in V$:
$$\begin{array}{rlllllllllll}
l_{\dashv}(\alpha_{1}(x))l_{\dashv}(y)v&=& l_{\dashv}(\alpha_{1}(x))l_{\vdash}(y)v,\cr r_{\dashv}(\alpha_{1}(x)\dashv y)\beta_{1}(v)&=&r\dashv(x\vdash y)\beta_{1}(v),\cr
l_{\dashv}(\alpha_{1}(x))r_{\dashv}(y)v&=&l_{\dashv}(\alpha_{1}(x))r_{\vdash}(y)v,\cr l_{\vdash}(x \vdash y)\beta_{2}(v) &=& l_{\vdash}(x \dashv y)\beta_{2}(v),\cr
r_{\vdash}(\alpha_{2}(x))r_{\vdash}(y)v &=& r_{\vdash}(\alpha_{2}(x))r_{\dashv}(y)v, \cr r_{\dashv}(\alpha_{2}(x))l_{\vdash}(y)v &=& r_{\vdash}(\alpha_{2}(x))l_{\dashv}(y)v,\cr
l_{\dashv}(\alpha_{1}\alpha_{2}(x))l_{\dashv}(\alpha_{1}(y))v&-&
l_{\dashv}(\alpha_{2}(x)\dashv\alpha_{1}(y))\beta_{2}(v)\cr
&=&
l_{\vdash}(\alpha_{1}\alpha_{2}(y))l_{\dashv}(\alpha_{1}(x))v-l_{\dashv}(\alpha_{2}(y)\vdash\alpha_{1}(x))\beta_{2}(v),\cr
 r_{\dashv}(\alpha_{1}(x)\dashv y)\beta_{1}\beta_{2}(v)&-&r_{\dashv}(\alpha_{2}(y))r_{\dashv}(\alpha_{1}(x))\beta_{2}(v)\cr
&=& l_{\vdash}(\alpha_{1}\alpha_{2}(x))r_{\dashv}(y)\beta_{1}(v)-
 r_{\dashv}(\alpha_{2}(x))l_{\vdash}(\alpha_{2}(y))\beta_{1}(v),\cr
 l_{\vdash}(\alpha_{1}\alpha_{2}(x))l_{\vdash}(\alpha_{1}(y))v&-&l_{\vdash}(\alpha_{2}(x)\vdash\alpha_{1}(y))\beta_{2}(v)\cr
&=& l_{\vdash}(\alpha_{1}\alpha_{2}(y))l_{\vdash}(\alpha_{1}(x))v-l_{\vdash}(\alpha_{2}(y)\vdash\alpha_{1}(x))\beta_{2}(v),\cr
 r_{\vdash}(\alpha_{1}(x)\vdash y)\beta_{1}\beta_{2}(v)&-&r_{\vdash}(\alpha_{2}(y))r_{\vdash}(\alpha_{1}(x))\beta_{2}(v)\cr
&=& l_{\vdash}(\alpha_{1}\alpha_{2}(x))r_{\vdash}(y)\beta_{1}(v)-
 r_{\vdash}(\alpha_{2}(x))l_{\vdash}(\alpha_{2}(y))\beta_{1}(v),\cr
  \beta_{1}(l_{\vdash}(x)v)= l_{\vdash}(\alpha_{1}(x))\beta_{1}(v),&&\beta_{1}(r_{\vdash}(x)v)= r_{\vdash}(\alpha_{1}(x))\beta_{1}(v),\cr
\beta_{2}(l_{\vdash}(x)v) = l_{\vdash}(\alpha_{2}(x))\beta_{2}(v),&&\beta_{2}(r_{\vdash}(x)v)= r_{\vdash}(\alpha_{2}(x))\beta_{2}(v),\cr
\beta_{1}(l_{\dashv}(x)v)= l_{\dashv}(\alpha_{1}(x))\beta_{1}(v),&& \beta_{1}(r_{\dashv}(x)v)= r_{\dashv}(\alpha_{1}(x))\beta_{1}(v),\cr
\beta_{2}(l_{\dashv}(x)v) = l_{\dashv}(\alpha_{2}(x))\beta_{2}(v),&&\beta_{2}(r_{\dashv}(x)v)=r_{\dashv}(\alpha_{2}(x))\beta_{2}(v).
\end{array}$$
\end{defn}
\begin{prop}
If $(l_{\dashv}, r_{\dashv}, l_{\vdash}, r_{\vdash}, \beta_{1}, \beta_{2}, V)$ is a bimodule of a BiHom-left-symmetric
dialgebra $(S,\dashv, \vdash, \alpha_{1}, \alpha_{2}),$ then there exists a BiHom-left-symmetric
dialgebra structure on the direct sum $S\oplus V $ of the underlying vector spaces of $ S $ and $V$ given for all $ x, y \in S, u, v \in V $ by
\begin{eqnarray*}
(x + u) \dashv' (y + v) &:=& x \dashv y + l_{\dashv}(x)v + r_{\dashv}(y)u,\cr
(x + u) \vdash' (y + v) &:=& x \vdash y + l_{\vdash}(x)v + r_{\vdash}(y)u, \cr
(\alpha_1+\beta_1)(x+u)&:=&\alpha_1(x)+\beta_1(u),\cr
(\alpha_2+\beta_2)(x+u)&:=&\alpha_2(x)+\beta_2(u).
\end{eqnarray*}
We denote this structure by $ S \times_{l_{\dashv},r_{\dashv}, l_{\vdash}, r_{\vdash}, \alpha_{1}, \alpha_{2}, \beta_{1}, \beta_{2}} V$.
\end{prop}
\begin{proof}
We prove only the axiom \eqref{als1}, as the others are proved similarly.
For any $x_{1},x_{2},x_{3}\in S$ and $v_1, v_2, v_3\in V$,
\begin{align*}
&(\alpha_{1}+\beta_{1})(x_{1}+v_{1})\dashv'((x_2+v_2)\dashv'(x_3+v_3))\\
&\quad=(\alpha_1(x_1)+\beta_1(v_1))\dashv'(x_1\dashv x_3+l_\dashv(x_2)v_3+r_\dashv(x_3)v_2)\\
&\quad =\alpha_1(x_1)\dashv(x_2\dashv x_3)+l_\dashv(\alpha_1(x_1))l_\dashv(x_2)v_3 \\
& \quad\quad +l_\dashv(\alpha_1(x_1))r_\dashv(x_3)v_2+r_\dashv(x_2\dashv x_3)\beta_1(v_1)\\
& \quad=\alpha_1(x_1)\dashv(x_2\vdash x_3)+l_{\dashv}(\alpha_1(x_1))l_\vdash(x_2)v_3\\
&\quad \quad +l_\dashv(\alpha_1(x_1))r_\vdash(x_3)v_2+r_\dashv(x_2\vdash x_3)\beta_1(v)\\
&\quad =(\alpha_{1}+\beta_{1})(x_{1}+v_{1})\dashv'((x_2+v_2)\vdash'(x_3+v_3)).
\qedhere \end{align*}
\end{proof}
\begin{exes}\begin{enumerate}
\item
Let $(S,\dashv, \vdash,\alpha,\beta)$ be a BiHom-left-symmetric dialgebra. Then $(L_{\dashv},R_{\dashv},L_{\vdash},R_{\vdash},\alpha,\beta,S)$ and $(L_{\dashv},0,0,R_{\vdash},\alpha,\beta,S)$ are a bimodules of $(S,\dashv, \vdash,\alpha,\beta)$ where
$L_{\dashv}(a)b=a\dashv b,~R_{\dashv}(a)b=b\dashv a,~L_{\vdash}(a)b=a\vdash b$ and $R_{\vdash}(a)b=b\vdash a$ for all $(a,b)\in S^{2}$. More generally, if $B$ is a two-sided BiHom-ideal of $(S,\dashv, \vdash,\alpha,\beta)$, then $(L_{\dashv},R_{\dashv},L_{\vdash},R_{\vdash},\alpha,\beta,B)$ is a bimodule of $S$, where for all $x\in B$ and $(a,b)\in S^{2}$,
$$L_{\dashv}(a)x=a\dashv x=x\dashv a=R_{\dashv}(a)x, \quad L_{\vdash}(a)x=a\vdash x=x\vdash a=R_{\vdash}(a)x.$$
\item
If $(S,\dashv, \vdash)$ is a left-symmetric dialgebra and $(l_\dashv,r_\dashv,l_\vdash,r_\vdash,V)$ are a bimodules of $S$ in the usual sense, then $(l_\dashv,r_\dashv,l_\vdash,r_\vdash,Id_{V},Id_{V},V)$ is a bimodule of $\mathbb{S}$, where $\mathbb{S}=(S,\dashv, \vdash,Id_{S},
Id_{S})$ is a BiHom-left-symmetric dialgebra.\end{enumerate}
\end{exes}
\begin{prop}\label{viaf}
If $f:(S,\dashv_1, \vdash_1,\alpha_{1},\alpha_2)\longrightarrow(S',\dashv_2, \vdash_2,\beta_{1},\beta_{2})$ is a morphism of BiHom-left-symmetric dialgebras, then $(l_{\dashv_1},r_{\dashv_1},l_{\vdash_1},r_{\vdash_1},\beta_1,\beta_{2},S')$ becomes a bimodule of $S$ via $f$, that is, $l_{\dashv_1}(a)b=f(a)\dashv_2 b,~r_{\dashv_1}(a)b=b \dashv_2 f(a),~l_{\vdash_1}(a)b=f(a)\vdash_2 b$ and $r_{\vdash_1}(a)b=b \vdash_2 f(a)$ for all
$(a,b)\in S\times S'$.
\end{prop}
\begin{proof}
We prove only first axiom, as the other axioms are proved similarly.
For any $x,y\in S, z\in S'$,
\begin{align*}
l_{\dashv_1}(\alpha_1(x))l_{\vdash_1}(y)z&=f(\alpha_1(x))\dashv_2(f(y)\vdash_2 z)\\
&=\beta_1f(x)\dashv_2( f(y)\vdash_2 z)\\
&=\beta_1f(x)\dashv_2( f(y)\dashv_2 z) \quad \textrm{(by \eqref{als1})}\\
&=f(\alpha_1(x))\dashv_2(f(y)\dashv_2 z)\\
&=l_{\dashv_1}(\alpha_1(x))l_{\dashv_1}(y)z.
\qedhere
\end{align*}\end{proof}
\begin{defn}
An abelian extension of BiHom-left-symmetric dialgebra is a short exact sequence of BiHom-left-symmetric dialgebras
$$0\longrightarrow (V,\alpha_{V},\beta_{V})\stackrel{\mbox{i}} \longrightarrow(A,\dashv_{A},\vdash_A,\alpha_{A},\beta_{A})\stackrel{\mbox{$\pi$}}\longrightarrow (B,\dashv_{B},\vdash_B,\alpha_{B},\beta_{B})\longrightarrow 0 ,$$
where $(V,\alpha_{V},\beta_{V})$ is a trivial BiHom-left-symmetric dialgebra, and $i$ and $\pi$ are morphisms of BiHom-left-symmetric dialgebras. Furthermore, if there exists a morphism $s:(B,\dashv_{B},\vdash_B,\alpha_{B},\beta_{B})
\longrightarrow (A,\dashv_{A},\vdash_A,\alpha_{A},\beta_{A})$ such that $\pi\circ s=id_{B}$ then the abelian extension is said to be split and $s$ is called a section of $\pi$.
\end{defn}
\begin{rmk} Consider the  split null extension $S\oplus V$ determined by the bimodule $(l_\dashv,r_\dashv,l_\vdash,r_\vdash,\alpha_{V},\beta_{V},V)$ of BiHom-left-symmetric dialgebra
$(S,\dashv,\vdash,\alpha,\beta)$ in the previous proposition. Write elements $a+v$ of $S\oplus V$ as $(a,v).$ Then there is an injective homomorphism of BiHom-modules
$i :V\rightarrow S\oplus V $ given by $i(v)=(0,v)$ and a surjective homomorphism of BiHom-modules $\pi : S\oplus V\rightarrow S$ given by $\pi(a,v)=a.$
Moreover, $i(V)$ is a two-sided BiHom-ideal of $S\oplus V$  such that $S\oplus V/i(V)\cong S$. On the other hand, there is a morphism of BiHom-left-symmetric dialgebra
$\sigma: S\rightarrow S\oplus V$ given by $\sigma(a)=(a,0)$ which is clearly a section of $\pi.$ Hence, we obtain the abelian split exact sequence of
BiHom-left-symmetric dialgebra and $(l_\dashv,r_\dashv,l_\vdash,r_\vdash,\alpha_{V},\beta_{V},V)$ is a bimodule for $S$ via $\pi.$
 \end{rmk}
\begin{thm}\label{mamm}
Let $(S,\dashv, \vdash,\alpha_1,\alpha_2)$ be a BiHom-left-symmetric dialgebra,  and let $(l_{\dashv},r_{\dashv},l_{\vdash},r_{\vdash},\beta_1,\beta_2,V)$ be a bimodule of $S$. Let $\alpha'_1,\alpha'_2$ be endomorphisms of $S$ such that any two of the maps $\alpha_1,\alpha'_1,\alpha_2,\alpha'_2$ commute
and $\beta'_1,~\beta'_2$ be linear self-maps of $V$ such that any two of the maps $\beta_1,\beta'_1,\beta_2,\beta'_2$ commute. Suppose furthermore that
$$\left\{
   \begin{array}{lllllll}
    \beta'_1\circ l_\dashv=(l_\dashv\circ\alpha'_1)\beta'_1,~~
     \beta'_2\circ l_\dashv=(l_\dashv\circ\alpha'_2)\beta'_2,& \\
         \beta'_1\circ l_\vdash=(l_\vdash\circ\alpha'_1)\beta'_1,~~
     \beta'_2\circ l_\vdash=(l_\vdash\circ\alpha'_2)\beta'_2,&
   \end{array}
 \right.$$
$$\left\{
   \begin{array}{lllllll}
    \beta'_1\circ r_\dashv=(r_\dashv\circ\alpha'_1)\beta'_1,~~
     \beta'_2\circ r_\dashv=(r_\dashv\circ\alpha'_2)\beta'_2,& \\
         \beta'_1\circ r_\vdash=(r_\vdash\circ\alpha'_1)\beta'_1,~~
     \beta'_2\circ r_\vdash=(r_\vdash\circ\alpha'_2)\beta'_2,&
   \end{array}
 \right.$$
and let $S_{\alpha'_1,\alpha'_2}$ be the BiHom-left-symmetric dialgebra $(S,\dashv_{\alpha'_1,\alpha'_2}, \vdash_{\alpha'_1,\alpha'_2},\alpha_1\alpha'_1,\alpha_2\alpha'_2)$ and
$V_{\beta'_1,\beta'_2}=(\widetilde{l}_{\dashv},\widetilde{r}_{\dashv},\widetilde{l}_{\vdash},
\widetilde{r}_{\vdash},\beta_1\beta'_1,\beta_2\beta'_2,V)$ where
\begin{equation}
\begin{array}{l}
\widetilde{l}_{\dashv}=(l_{\dashv}\circ\alpha'_1)\beta'_2, \quad
\widetilde{r}_{\dashv}=(r_{\dashv}\circ\alpha'_2)\beta'_1, \\ \widetilde{l}_{\vdash}=(l_{\vdash}\circ\alpha'_1)\beta'_2, \quad
\widetilde{r}_{\vdash}=(r_{\vdash}\circ\alpha'_2)\beta'_1.
\end{array}
\end{equation}
Then $V_{\beta'_1,\beta'_2}$ is a bimodule of $S_{\alpha'_1,\alpha'_2}$.
\end{thm}
\begin{proof}
We prove only one axiom, as others are proved similarly.
For any $x,y\in S, v\in V$,
\begin{align*}
\widetilde{l}_{\dashv}(\alpha_1\alpha'_1(x))\widetilde{l}_{\dashv}(y)v
&=l_{\dashv}(\alpha_1\alpha'^{2}_1(x))\beta'_2 (l_{\dashv}(\alpha'_{1}(y))\beta'_2(v)\\
&=l_{\dashv}(\alpha_1\alpha'^{2}_1(x))l_{\dashv}(\alpha'_1\alpha'_2(y))\beta'^{2}_2(v)\\
&=l_{\dashv}(\alpha_1\alpha'^{2}_1(x))l_{\vdash}(\alpha'_1\alpha'_2(y))\beta'^{2}_2(v)\\
&=\widetilde{l}_{\dashv}(\alpha_1\alpha'_1(x))\widetilde{l}_{\vdash}(y)v.
\qedhere \end{align*}
\end{proof}
\begin{cor}
Let $(S,\dashv, \vdash,\alpha_1,\alpha_2)$ be a BiHom-left-symmetric dialgebra, and let $(l_{\dashv},r_{\dashv},l_{\vdash},r_{\vdash},\beta_1,\beta_2,V)$ be a bimodule of $S$. Then
$V_{\beta^{q}_1,\beta^{q}_2}$ is a bimodule of $S_{\alpha^{p}_1,\alpha^{p}_2}$ for any nonnegative integers $p$ and $q$.
\end{cor}
\begin{proof}
Apply Theorem \ref{mamm} with $\alpha'_1=\alpha_1^{p},~\alpha'_2=\alpha_{2}^{p}$ and $\beta'_1=\beta_1^{q},~\beta'_2=\beta_{2}^{q}$.
\end{proof}
Assume that $( l_{\dashv}, r_{\dashv}, l_{\vdash}, r_{\vdash}, \beta_{1}, \beta_{2}, V)$ be a bimodule of a BiHom-left-symmetric dialgebra $(S, \dashv, \vdash, \alpha_{1}, \alpha_{2})$ and let $l_{\dashv}^{\ast}, r_{\dashv}^{\ast}, l_{\vdash}^{\ast}, r_{\vdash}^{\ast}:S\rightarrow gl(V^{\ast}).$
Let $\alpha_1^{\ast},\alpha_{2}^{\ast}:S^{\ast}\rightarrow S^{\ast},$ and $\beta_{1}^{\ast},\beta_{2}^{\ast}:V^{\ast}\rightarrow V^{\ast}$
be the dual maps of respectively $\alpha_1,\alpha_2,\beta_1$ and $\beta_2$ such that
$$\begin{array}{llllllll}
  \langle l_{\dashv}^{\ast}(x)u^{\ast},v\rangle =\langle u^{\ast},l_{\dashv}(x)v\rangle,&& \langle r^{\ast}_{\dashv}(x)u^{\ast},v\rangle =\langle u^{\ast},r_{\dashv}(x)v\rangle,\\
   \langle l_{\vdash}^{\ast}(x)u^{\ast},v\rangle =\langle u^{\ast},l_{\vdash}(x)v\rangle,&& \langle r^{\ast}_{\vdash}(x)u^{\ast},v\rangle =\langle u^{\ast},r_{\vdash}(x)v\rangle,\\
    \alpha_{1}^{\ast}(x^{\ast}(y))=x^{\ast}(\alpha_{1}(y)),&& \alpha_{2}^{\ast}(x^{\ast}(y))=x^{\ast}(\alpha_{2}(y)),\\
     \beta_{1}^{\ast}(u^{\ast}(v))=u^{\ast}(\beta_{1}(v)),&& \beta_{2}^{\ast}(u^{\ast}(v))=u^{\ast}(\beta_{2}(v)).
\end{array}$$

The following proposition holds.
\begin{prop}
If $( l_{\dashv}, r_{\dashv}, l_{\vdash}, r_{\vdash}, \beta_{1}, \beta_{2}, V)$ is a bimodule of a BiHom-left-symmetric dialgebra $(S, \dashv, \vdash, \alpha_{1}, \alpha_{2})$, then $( l_{\dashv}^{\ast}, r_{\dashv}^{\ast}, l_{\vdash}^{\ast}, r_{\vdash}^{\ast}, \beta_{1}^{\ast}, \beta_{2}^{\ast}, V^{\ast})$ is a bimodule of $(S, \dashv, \vdash,\alpha_{1}, \alpha_{2})$ provided that for all $x,y\in S$ and $u\in V$,
\begin{align*}
l_{\dashv}(y)l_{\dashv}(\alpha_{1}(x))u&= l_{\vdash}(y)l_{\dashv}(\alpha_{1}(x))u,\cr\beta_{1}( r_{\dashv}(\alpha_{1}(x)\dashv y)(v))u&=\beta_{1}(r\dashv(x\vdash y))u,\cr
r_{\dashv}(y)l_{\dashv}(\alpha_{1}(x))u&=r_{\vdash}(y)l_{\dashv}(\alpha_{1}(x))u,\cr \beta_{2}(l_{\vdash}(x \vdash y))u &= \beta_{2}(l_{\vdash}(x \dashv y))u,\cr
r_{\vdash}(y)r_{\vdash}(\alpha_{2}(x))u &=r_{\dashv}(y) r_{\vdash}(\alpha_{2}(x))u, \cr l_{\vdash}(y)r_{\dashv}(\alpha_{2}(x))u &= l_{\dashv}(y)r_{\vdash}(\alpha_{2}(x))u,\cr
l_{\dashv}(\alpha_{1}(y))l_{\dashv}(\alpha_{1}\alpha_{2}(x))u&-\beta_{2}(l_{\dashv}(\alpha_{2}(x)\dashv\alpha_{1}(y)))u\cr
&=
l_{\dashv}(\alpha_{1}(x))l_{\vdash}(\alpha_{1}\alpha_{2}(y))u-\beta_{2}(l_{\dashv}(\alpha_{2}(y)\vdash\alpha_{1}(x)))u,\cr
 \beta_{2}\beta_{1}(r_{\dashv}(\alpha_{1}(x)\dashv y))u&-\beta_{2}r_{\dashv}(\alpha_{1}(x))r_{\dashv}(\alpha_{2}(y))u\cr
 &=r_{\dashv}(y)\beta_{1}l_{\vdash}(\alpha_{1}\alpha_{2}(x))u-
 \beta_{1}l_{\vdash}(\alpha_{2}(y))r_{\dashv}(\alpha_{2}(y))u,\cr
 l_{\vdash}(\alpha_{1}(y))l_{\vdash}(\alpha_{1}\alpha_{2}(x))u&-\beta_{2}(l_{\vdash}(\alpha_{2}(x)\vdash\alpha_{1}(y)))u\cr
&=
l_{\vdash}(\alpha_{1}(x))l_{\vdash}(\alpha_{1}\alpha_{2}(y))u-\beta_{2}(l_{\vdash}(\alpha_{2}(y)\vdash\alpha_{1}(x)))u,\cr
 \beta_{2}(r_{\vdash}(\alpha_{1}(x)\vdash y)\beta_{1})u&-\beta_2 r_{\vdash}(\alpha_{1}(x))r_{\vdash}(\alpha_{2}(y))u\cr
&=\beta_{1}r_{\vdash}(y)l_{\vdash}(\alpha_{1}\alpha_{2}(x))u-
\beta_{1}l_{\vdash}(\alpha_{2}(y)) r_{\vdash}(\alpha_{2}(y))u.
\end{align*}
\end{prop}
The following theorem is proved in a similar way as for Theorem \ref{matched ass}.
\begin{thm}
Let $(A,\dashv_A,\vdash_A,\alpha_1,\alpha_2)$ and $(B,\dashv_B,\vdash_{B},\beta_1,\beta_2)$ be two BiHom-left-symmetric dialgebras. Suppose that there are linear maps $l_{\dashv_A},r_{\dashv_A},l_{\vdash_A},r_{\vdash_A}:A\rightarrow gl(B)$
and $l_{\dashv_B},r_{\dashv_B},l_{\vdash_B},r_{\vdash_B}:B\rightarrow gl(A)$ such that $( l_{\dashv_{A}}, r_{\dashv_{A}},  l_{\vdash_{A}},
r_{\vdash_{A}}, \beta_{1}, \beta_{2}, B)$ is a bimodule of $A,$ and
$(l_{\dashv_{B}},   r_{\dashv_{B}},  l_{\vdash_{B}},  r_{\vdash_{B}},
\alpha_{1}, \alpha_{2}, A)$  is a bimodule  of $B,$ and for all $x,y\in A,~a,b\in B$, the following equalities hold:
\begin{align}
\label{bieq35A}
&
\begin{array}{l}
l_{\dashv_A}(\alpha_1(x))(a\dashv_B b)=l_{\dashv_A}(\alpha_1(x))(a\vdash_B b),
\end{array}
\\
\label{bieq36A}
&
\begin{array}{l}
r_{\dashv_A}(r_{\dashv_B}(b)x)\beta_1(a)+\beta_1(a)\dashv_B(l_{\dashv_A}(x)b)= \\
\quad r_{\dashv_A}(r_{\vdash_B}(b)x)\beta_1(a)+\beta_1(a)\dashv_B(l_{\vdash_A}(x)b),
\end{array}
\\
\label{bieq37A}
&
\begin{array}{l}
r_{\dashv_A}(l_{\dashv_B}(b)x)\beta_1(a)+\beta_1(a)\dashv_B(r_{\dashv_A}(x)b)=\\
\quad r_{\dashv_A}(l_{\vdash_B}(b)x)\beta_1(a)+\beta_1(a)\dashv_B(r_{\vdash_A}(x)b),
\end{array}
\\
\label{bieq38A}
&
\begin{array}{l}
r_{\vdash_A}(\beta(x))(a\vdash_B b)=r_{\dashv_A}(\alpha_2(x))(a\dashv_B b),
\end{array}
\\
\label{bieq39A}
&
\begin{array}{l}
l_{\vdash_A}(r_{\vdash_B}(a)x)\beta_2(b)+(l_{\vdash_A}(x)a)\vdash_B\beta_2(b)=\\
\quad l_{\vdash_A}(r_{\dashv_B}(a)x)\beta_2(b)+(l_{\dashv_A}(x)a)\vdash_B\beta_2(b),
\end{array}\\
\label{bieq40A}
&\begin{array}{l}
l_{\vdash_A}(l_{\dashv_B}(a)x)\beta_2(b)+(r_{\vdash_A}(x)a)\vdash_B\beta_2(b)= \\
\quad l_{\vdash_A}(l_{\dashv_B}(a)x)\beta_2(b)+(r_{\dashv_A}(x)a)\vdash_B\beta_2(b),
\end{array}\\
\label{bieq41A}
&\begin{array}{l}
l_{\dashv_A}(\alpha_1\alpha_2(x))(\alpha_1(a)\dashv_B b)
-l_{\dashv_A}(r_{\dashv_B}(\beta_1(a))\alpha_2(x))\beta_2(b)\\
\quad -(l_{\dashv_A}(\alpha_1(x))\beta_1(a))\dashv_B\beta_2(b)=\\
\beta_1\beta_2(a)\vdash_B(l_{\dashv_A}(\alpha_1(x))b)+r_{\vdash_A}(r_{\dashv_B}(b)\alpha_1(x))\beta_1\beta_2(a)\\
\quad-(r_{\vdash_A}(\alpha_1(x))\beta_2(a))\dashv_B \beta_2(b) -l_{\dashv_A}(l_{\vdash_B}(\beta_2(a))\alpha_1(x))\beta_2(b),
\end{array}
\\
\label{bieq42A}
&\begin{array}{l}
\beta_1\beta_2(a)\dashv_B(l_{\dashv_A}(\alpha_1(x))b)+r_{\dashv_A}(r_{\dashv_B}(b)\alpha_1(x))\beta_1\beta_2(a)\\
\quad -(r_{\dashv_A}(\alpha_1(x))\beta_2(a))\dashv_B\beta_2(b)=\\
l_{\vdash_A}(\alpha_1\alpha_2(x))(\beta_2(a)\dashv_B b)-(l_{\vdash_A}(\alpha_2(x))\beta_2(a))\dashv_B\beta_2(b)\\
\quad -l_{\dashv_A}(l_{\vdash_B}(\alpha_2(x))\beta_2(a))\beta_2(b),
\end{array}\\
\label{bieq43A}
&\begin{array}{l}
\beta_1\beta_2(a)\dashv_B(r_{\dashv_A}(x)\beta_2(b))+r_{\dashv_A}(l_{\dashv_B}(\beta_2(b))x)\beta_1\beta_2(a)\\
\quad -r_{\dashv_A}(\alpha_2(x))(\beta_2(a)\dashv_B\beta_1(b))=\\
\beta_1\beta_2(b)\vdash_B(r_{\dashv_A}(x)\beta_1(a))+r_{\vdash_A}(l_{\dashv_B}(\beta_1(a))x)\beta_1\beta_2(b) \\
\quad -r_{\dashv_A}(\alpha_2(x))(\beta_2(b)\vdash_B\beta_1(a)),
\end{array}
\\
\label{bieq44A}
&
\begin{array}{l}
l_{\vdash_A}(\alpha_1\alpha_2(x))(\beta_1(a)\vdash_B b)-l_{\vdash_A}(r_{\vdash_B}(\beta_1(a))\alpha_2(x))\beta_2(b)
\\
\quad -(l_{\vdash_A}(\alpha_2(x))\beta_1(a))\vdash_B\beta_2(b)=\\
\beta_1\beta_2(a)\vdash_B(l_{\vdash_A}(\alpha_1(x))b)+r_{\vdash_A}(r_{\vdash_B}(b)\alpha_1(x))\beta_1\beta_2(a)\\
\quad -(r_{\vdash_A}(\alpha_1(x))\beta_2(a))\vdash_B \beta_2(b)
-l_{\vdash_A}(l_{\vdash_B}(\beta_2(a))\alpha_1(x))\beta_2(b),
\end{array} \\
\label{bieq45A}
&
\begin{array}{l}
\beta_1\beta_2(a)\vdash_B(l_{\vdash_A}(\alpha_1(x))b)+r_{\vdash_A}(r_{\vdash_B}(b)\alpha_1(x))\beta_1\beta_2(a)
\\
\quad -(r_{\vdash_A}(\alpha_1(x))\beta_2(a))\vdash_B\beta_2(b)=\\
\quad l_{\vdash_A}(\alpha_1\alpha_2(x))(\beta_2(a)\vdash_B b) -(l_{\vdash_A}(\alpha_2(x))\beta_2(a))\vdash_B\beta_2(b)\\ \quad -l_{\vdash_A}(l_{\vdash_B}(\alpha_2(x))\beta_2(a))\beta_2(b),
\end{array}\\
\label{bieq46A}
&
\begin{array}{l}
\beta_1\beta_2(a)\vdash_B(r_{\vdash_A}(x)\beta_2(b))+r_{\vdash_A}(l_{\vdash_B}(\beta_2(b))x)\beta_1\beta_2(a)\\
\quad -r_{\vdash_A}(\alpha_2(x))(\beta_2(a)\vdash_B\beta_1(b))=\\
\beta_1\beta_2(b)\vdash_B(r_{\vdash_A}(x)\beta_1(a))+
r_{\vdash_A}(l_{\vdash_B}(\beta_1(a))x)\beta_1\beta_2(b)\\
\quad -r_{\vdash_A}(\alpha_2(x))(\beta_2(b)\vdash_B\beta_1(a)),
\end{array}\\
\label{bieq35B}
&
\begin{array}{l}
l_{\dashv_B}(\beta_1(a))(x\dashv_A y)=l_{\dashv_B}(\beta_1(a))(x\vdash_A y),
\end{array}\\
\label{bieq36B}
&\begin{array}{l}
r_{\dashv_B}(r_{\dashv_A}(y)a)\alpha_1(x)+\alpha_1(x)\dashv_A(l_{\dashv_B}(a)y)=\\
\quad
r_{\dashv_B}(r_{\vdash_A}(y)a)\alpha_1(x)+\alpha_1(x)\dashv_A(l_{\vdash_B}(a)y),
\end{array}\\
\label{bieq37B}
&\begin{array}{l}
r_{\dashv_B}(l_{\dashv_A}(y)a)\alpha_1(x)+\alpha_1(x)\dashv_A(r_{\dashv_B}(a)y)=\\
\quad r_{\dashv_B}(l_{\vdash_A}(y)a)\alpha_1(x)+\alpha_1(x)\dashv_A(r_{\vdash_B}(a)y),
\end{array}\\
\label{bieq38B}
&
\begin{array}{l}
r_{\vdash_B}(\alpha(a))(x\vdash_A y)=r_{\dashv_B}(\beta_2(a))(x\dashv_A y),
\end{array}
\\
\label{bieq39B}
&
\begin{array}{l}
l_{\vdash_B}(r_{\vdash_A}(x)a)\alpha_2(y)+(l_{\vdash_B}(a)x)\vdash_A\alpha_2(y)=\\
\quad l_{\vdash_B}(r_{\dashv_A}(x)a)\alpha_2(y)+(l_{\dashv_B}(a)a)\vdash_A\alpha_2(y),
\end{array} \\
\label{bieq40B}
&
\begin{array}{l}
l_{\vdash_B}(l_{\dashv_A}(x)a)\alpha_2(y)+(r_{\vdash_B}(a)x)\vdash_A\alpha_2(y)=\\
\quad l_{\vdash_B}(l_{\dashv_A}(x)a)\alpha_2(y)+(r_{\dashv_B}(a)x)\vdash_A\alpha_2(y),
\end{array}
\\
\label{bieq41B}
&\begin{array}{l}
l_{\dashv_B}(\beta_1\beta_2(a))(\alpha_1(x)\dashv_A y)
-l_{\dashv_B}(r_{\dashv_A}(\alpha_1(x))\beta_2(a))\alpha_2(y)
\\
\quad  -(l_{\dashv_B}(\beta_1(a))\alpha_1(x))\dashv_A\alpha_2(y)=\\
\alpha_1\alpha_2(x)\vdash_A(l_{\dashv_B}(\beta_1(a))y)+r_{\vdash_B}(r_{\dashv_A}(y)\beta_1(a))\alpha_1\alpha_2(x)
\\
\quad  -(r_{\vdash_B}(\beta_1(a))\alpha_2(x))\dashv_A \alpha_2(y)
-l_{\dashv_B}(l_{\vdash_A}(\alpha_2(x))\beta_1(a))\alpha_2(y),
\end{array}
\\
\label{bieq42B}
&\begin{array}{l}
\alpha_1\alpha_2(x)\dashv_A(l_{\dashv_B}(\beta_1(a))y)+r_{\dashv_B}(r_{\dashv_A}(y)\beta_1(a))\alpha_1\alpha_2(x)
\\
\quad -(r_{\dashv_B}(\beta_1(a))\alpha_2(x))\dashv_A\alpha_2(y)=\\
l_{\vdash_B}(\beta_1\beta_2(xa))(\alpha_2(x)\dashv_A y)-(l_{\vdash_B}(\beta_2(a))\alpha_2(x))\dashv_A\alpha_2(y)
\\
\quad -l_{\dashv_B}(l_{\vdash_A}(\beta_2(a))\alpha_2(x))\alpha_2(y),
\end{array}\\
\label{bieq43B}
&\begin{array}{l}
\alpha_1\alpha_2(x)\dashv_A(r_{\dashv_B}(a)\alpha_2(y))
+r_{\dashv_B}(l_{\dashv_A}(\alpha_2(y))a)\alpha_1\alpha_2(x)
\\
\quad
-r_{\dashv_B}(\beta_2(a))(\alpha_2(x)\dashv_A\alpha_1(y))=\\
\alpha_1\alpha_2(y)\vdash_A(r_{\dashv_B}(a)\alpha_1(x))+
r_{\vdash_B}(l_{\dashv_A}(\alpha_1(x))a)\alpha_1\alpha_2(y)
\\
\quad
-r_{\dashv_B}(\beta_2(a))(\alpha_2(y)\vdash_A\alpha_1(x)),
\end{array}\\
\label{bieq44B}
&
\begin{array}{l}
l_{\vdash_B}(\beta_1\beta_2(a))(\alpha_1(x)\vdash_A y) -l_{\vdash_B}(r_{\vdash_A}(\alpha_1(x))\beta_2(a))\alpha_2(y)
\\
\quad -(l_{\vdash_B}(\beta_2(a))\alpha_1(x))\vdash_A\alpha_2(y)= \\
\alpha_1\alpha_2(x)\vdash_A(l_{\vdash_B}(\beta_1(a))y)+
r_{\vdash_B}(r_{\vdash_A}(y)\beta_1(a))\alpha_1\alpha_2(x)
\\
\quad -(r_{\vdash_B}(\beta_1(a))\alpha_2(x))\vdash_A \alpha_2(y)
-l_{\vdash_B}(l_{\vdash_A}(\alpha_2(x))\beta_1(a))\alpha_2(y),
\end{array}
\\
\label{bieq45B}
&
\begin{array}{l}
\alpha_1\alpha_2(x)\vdash_A(l_{\vdash_B}(\beta_1(a))y)+
r_{\vdash_B}(r_{\vdash_A}(y)\beta_1(a))\alpha_1\alpha_2(x)\\
\quad
-(r_{\vdash_B}(\beta_1(a))\alpha_2(x))\vdash_A\alpha_2(y)= \\
l_{\vdash_B}(\beta_1\beta_2(a))(\alpha_2(x)\vdash_A y)-(l_{\vdash_B}(\beta_2(a))\alpha_2(x))\vdash_A\alpha_2(y)\\
\quad
-l_{\vdash_B}(l_{\vdash_A}(\beta_2(a))\alpha_2(x))\alpha_2(y),
\end{array} \\
\label{bieq46B}
&\begin{array}{l}
\alpha_1\alpha_2(x)\vdash_A(r_{\vdash_B}(a)\alpha_2(y))+r_{\vdash_B}(l_{\vdash_A}(\alpha_2(y))a)
\alpha_1\alpha_2(x) \\
\quad -r_{\vdash_B}(\beta_2(a))(\alpha_2(x)\vdash_A\alpha_1(y))= \\
\alpha_1\alpha_2(y)\vdash_A(r_{\vdash_B}(a)\alpha_1(x))+
r_{\vdash_B}(l_{\vdash_A}(\alpha_1(x))a)\alpha_1\alpha_2(y)  \\
\quad -r_{\vdash_B}(\beta_2(a))(\alpha_2(y)\vdash_A\alpha_1(x)).
\end{array}
\end{align}
Then $(A,B,l_{\dashv_A},r_{\dashv_A},l_{\vdash_A},r_{\vdash_A},\beta_1,\beta_2,l_{\dashv_B},r_{\dashv_B},l_{\vdash_B},r_{\vdash_B},\alpha_1,\alpha_2)$ is called a matched pair of BiHom-left-symmetric dialgebras. In this case, there exists a BiHom-left-symmetric dialgebra structure on the direct sum
$A\oplus B$ of the underlying vector spaces of $A$ and $B$ given by
\begin{align*}
(x + a) \dashv(y + b)&:=x \dashv_A y + (l_{\dashv_A}(x)b + r_{\dashv_A}(y)a)+a \dashv_B b + (l_{\dashv_B}(a)y + r_{\dashv_B}(b)x),\cr
(x + a) \vdash (y + b)&:=x \vdash_A y + (l_{\vdash_A}(x)b + r_{\vdash_A}(y)a)+a \vdash_B b + (l_{\vdash_B}(a)y + r_{\vdash_B}(b)x),\cr
(\alpha_{1}\oplus\beta_{1})(x + a)&:=\alpha_{1}(x) + \beta_{1}(a),\cr
(\alpha_{2}\oplus\beta_{2})(x + a)&:=\alpha_{2}(x) + \beta_{2}(a).
\end{align*}
\end{thm}
We denote this BiHom-left-symmetric dialgebra by $A\bowtie^{l_{\dashv_A},r_{\dashv_A},l_{\vdash_A},r_{\vdash_A},\beta_1,\beta_2}_{l_{\dashv_B},r_{\dashv_B},l_{\vdash_B},r_{\vdash_A},\alpha_1,\alpha_2}B$.\\

\subsection{BiHom-associative dialgebra}
\begin{defn}\label{dia}
A BiHom-associative dialgebra is a quintuple $(D, \dashv, \vdash, \alpha, \beta)$ consisting of a vector space $D$ on which the operations $\dashv, \vdash: D\otimes D\rightarrow D$ and $\alpha, \beta: D\rightarrow D$ are linear maps satisfying, for $x, y, z\in D$,
\begin{eqnarray}
\alpha\circ\beta &=&\beta\circ\alpha,\label{dia0}\\
\alpha(x\dashv y)&=&\alpha(x)\dashv\alpha(y), \alpha(x\vdash y)=\alpha(x)\vdash\alpha(y),\label{dia0.0}\\
\beta(x\dashv y)&=&\beta(x)\dashv\beta(y), \beta(x\vdash y)=\beta(x)\vdash\beta(y),\label{dia0.00}\\
(x\vdash y)\dashv\beta(z)&=&\alpha(x)\vdash(y\dashv z), \label{dia1}\\
\alpha(x)\dashv (y\dashv z)&=&(x\dashv y)\dashv\beta(z),\label{dia2}\\
(x\dashv y)\dashv\beta(z)&=&\alpha(x)\dashv(y\vdash z),\label{dia3}\\
(x\vdash y)\vdash\beta(z)&=&\alpha(x)\vdash(y\vdash z),\label{dia4}\\
\alpha(x)\vdash(y\vdash z)&=&(x\dashv y)\vdash\beta(z).\label{dia5}
\end{eqnarray}
\end{defn}
\begin{rmk}\label{bk4}
The following connections between the considered structures hold.
\begin{enumerate}
\item
 If $(D, \dashv, \vdash, \alpha,\beta)$ is a BiHom-associative dialgebra and $\dashv=\vdash=:\mu$, then
$(D, \mu,\alpha,\beta)$ is a BiHom-associative algebra. Any BiHom-associative algebra $(A,\mu,\alpha,\beta)$ is a BiHom-associative dialgebra with $\dashv:=\mu=:\vdash$.
\item
A BiHom-associative dialgebra is a BiHom-X algebra.
\end{enumerate}
\end{rmk}
\begin{prop}
All BiHom-associative dialgebras are BiHom-left symmetric dialgebras.
\end{prop}
\begin{proof}
Let $(D,\dashv,\vdash, \alpha,\beta)$ be a BiHom-associative dialgebra, then \eqref{als1} and \eqref{als2} are satisfied. Since the products $\dashv$ and $\vdash$ are associative with the condition \eqref{dia1}, the equalities \eqref{als3} and \eqref{als4} are established.
\end{proof}
\begin{rmk}
Any BiHom-left symmetric algebra is a BiHom-left symmetric dialgebra in which $\dashv=\vdash$. A nonassociative BiHom-left symmetric algebra is not a BiHom-left symmetric dialgebra.
\end{rmk}
\begin{prop}\label{iibb}
 A BiHom-left-symmetric dialgebra $S$ is a BiHom-associative dialgebra if and only if both products of $S$ are BiHom-associative.
\end{prop}
\begin{proof}
If a BiHom-left-symmetric dialgebra $S$ is a BiHom-associative dialgebra, then both products $\dashv$ and $\vdash$ defined over $S$
are BiHom-associative according to Definition \ref{dia}.
Conversely, if each product of a BiHom-left-symmetric dialgebra
is BiHom-associative, then by Definition \ref{gls}, $S$ is a BiHom-associative dialgebra.
\end{proof}

\begin{defn}
An averaging operator over a BiHom-associative algebra  $(A, \mu, \alpha,\beta)$ is a linear map $\gamma : A\rightarrow A$ such
that $\alpha\circ\gamma=\gamma\circ\alpha$ and $\beta\circ\gamma=\gamma\circ\beta$, and for all $x, y\in A$,
\begin{equation}
 \gamma(\mu(\gamma(x), y)=\mu(\gamma(x), \gamma(y))=\gamma(\mu(x, \gamma(y))). \label{avo1}
\end{equation}
 \end{defn}

\begin{thm}\label{bk2}
 Let $(A,\cdot)$ be an associative algebra and $\alpha,\beta : A\rightarrow A$ two averaging operators such that
$(A, \cdot,\alpha,\beta)$ be a BiHom-associative algebra. For any $x, y\in A$, define new operations on $A$ by
$$x\vdash y=\alpha(x)\cdot \beta(y)\quad\mbox{and}\quad x\dashv y:=\beta(x)\cdot\alpha(y).$$
Then $(A, \dashv, \vdash, \alpha,\beta)$ is a BiHom-associative dialgebra.
\end{thm}
\begin{proof}
 We prove only one axiom, as others are proved similarly. For any $x, y, z\in A$,
\begin{align*}
 \alpha(x)\dashv(y\dashv z)&-(x\dashv y)\dashv\beta(z)
=\alpha\beta(x)\cdot\alpha(\beta(y)\cdot\alpha(z))-\beta(\beta(x)\cdot\alpha(y))\cdot\alpha\beta(z)\nonumber\\
&= \alpha\beta(x)\cdot(\alpha\beta(y)\cdot\alpha(z))-(\beta(x)\cdot\alpha\beta(y))\cdot\alpha\beta(z)
\quad\quad\;(\mbox{by}\;\eqref{avo1})\nonumber\\
&= \alpha\beta(x)\cdot(\alpha\beta(y)\cdot\alpha(z))-\alpha\beta(x)\cdot(\alpha\beta(y)\cdot\alpha(z))=0.
\;\;(\mbox{by}\;\eqref{aca})\nonumber
\end{align*}
This proves the second axiom in Definition \ref{dia}.
\end{proof}
\begin{defn}
Let $(D, \dashv, \vdash, \alpha_{1}, \alpha_{2})$ be a BiHom-associative
dialgebra, and $V$  be a vector space.
Let $l_{\dashv}, r_{\dashv}, l_{\vdash}, r_{\vdash} : D \rightarrow gl(V),$ and $\beta_{1}, \beta_{2}: V \rightarrow V$ be six linear maps. Then, $(l_{\dashv}, r_{\dashv}, l_{\vdash}, r_{\vdash}, \beta_{1}, \beta_{2}, V$) is called a bimodule of $ D $
 if for any $x, y \in D $ and $v\in V$:
$$\begin{array}{lllllllllll}
l_{\dashv}(x\vdash y)\beta_{2}(v)&=&l_{\vdash}(\alpha_{1}(x))l_{\dashv}(y)v,&r_{\dashv}(\alpha_{2}(x))l_{\vdash}(y)v&=&l_{\vdash}(\alpha_{1}(y))r_{\dashv}(x)v,\\r_{\dashv}(\alpha_{2}(x))
r_{\vdash}(y)(v)&=&r_{\vdash}(y\dashv x)\beta_{1}(v),&
l_{\dashv}(x\dashv y)\beta_{2}(v)&=&l_{\dashv}(\alpha_{1}(x))l_{\dashv}(y)v,\\
r_{\dashv}(\alpha_{2}(x))l_{\dashv}(y)v&=&l_{\dashv}(\alpha_{1}(y))r_{\dashv}(x)v,&r_{\dashv}(\alpha_{2}(x))
r_{\dashv}(y)(v)&=&r_{\dashv}(y\dashv x)\beta_{1}(v),\\
l_{\dashv}(x\dashv y)\beta_{2}(v)&=&l_{\dashv}(\alpha_{1}(x))l_{\vdash}(y)v,&r_{\vdash}(\alpha_{2}(x))l_{\vdash}(y)v&=&l_{\vdash}(\alpha_{1}(y))r_{\vdash}(x)v,\\r_{\vdash}(\alpha_{2}(x))
r_{\vdash}(y)(v)&=&r_{\vdash}(y\vdash x)\beta_{1}(v),&
l_{\vdash}(x\vdash y)\beta_{2}(v)&=&l_{\vdash}(\alpha_{1}(x))l_{\vdash}(y)v,\\
r_{\dashv}(\alpha_{2}(x))l_{\vdash}(y)v&=&l_{\vdash}(\alpha_{1}(y))r_{\dashv}(x)v, &r_{\dashv}(\alpha_{2}(x))
r_{\vdash}(y)(v)&=&r_{\vdash}(y\dashv x)\beta_{1}(v),\\
l_{\dashv}(x\dashv y)\beta_{2}(v)&=&l_{\vdash}(\alpha_{1}(x))l_{\dashv}(y)v, &r_{\vdash}(\alpha_{2}(x))l_{\dashv}(y)v&=&l_{\vdash}(\alpha_{1}(y))r_{\vdash}(x)v,\\r_{\vdash}(\alpha_{2}(x))
r_{\dashv}(y)(v)&=&r_{\vdash}(y\vdash x)\beta_{1}(v), &
 \beta_{1}(l_{\vdash}(x)v)&=& l_{\vdash}(\alpha_{1}(x))\beta_{1}(v),\\ \beta_{1}(r_{\vdash}(x)v)&=& r_{\vdash}(\alpha_{1}(x))\beta_{1}(v),&
\beta_{2}(l_{\vdash}(x)v) &=& l_{\vdash}(\alpha_{2}(x))\beta_{2}(v),\cr\beta_{2}(r_{\vdash}(x)v)&=& r_{\vdash}(\alpha_{2}(x))\beta_{2}(v),&
\beta_{1}(l_{\dashv}(x)v)&=& l_{\dashv}(\alpha_{1}(x))\beta_{1}(v),\cr \beta_{1}(r_{\dashv}(x)v)&=& r_{\dashv}(\alpha_{1}(x))\beta_{1}(v),&
\beta_{2}(l_{\dashv}(x)v) &=& l_{\dashv}(\alpha_{2}(x))\beta_{2}(v),\\\beta_{2}(r_{\dashv}(x)v)&=&r_{\dashv}(\alpha_{2}(x))\beta_{2}(v).
\end{array}$$
\end{defn}
\begin{prop}
Let $(l_{\dashv}, r_{\dashv}, l_{\vdash}, r_{\vdash}, \beta_{1}, \beta_{2}, V)$ be a bimodule of a BiHom-associative
dialgebra $(D,\dashv, \vdash, \alpha_{1}, \alpha_{2}).$ Then, there exists a BiHom-associative
dialgebra structure on the direct sum $D\oplus V $ of the underlying vector spaces of $ D $ and $V$ given for all $ x, y \in D, u, v \in V $ by
\begin{eqnarray*}
(x + u) \dashv' (y + v) &:=& x \dashv y + l_{\dashv}(x)v + r_{\dashv}(y)u, \cr
(x + u) \vdash' (y + v) &:=& x \vdash y + l_{\vdash}(x)v + r_{\vdash}(y)u, \cr
(\alpha_{1}+\beta_{1})(x+u)&:=&\alpha_1(x)+\beta_1(u), \cr
(\alpha_{2}+\beta_{2})(x+u)&:=&\alpha_2(x)+\beta_2(u).
\end{eqnarray*}
We denote such a BiHom-associative dialgebra by $D \times_{l_{\dashv},r_{\dashv}, l_{\vdash}, r_{\vdash}, \alpha_{1}, \alpha_{2}, \beta_{1}, \beta_{2}} V$.
\end{prop}
\begin{proof}
We prove only one axiom, as others are proved similarly.
For any $x_{1},x_{2},x_{3}\in D$ and $v_1, v_2, v_3\in V$,
\begin{align*}
&((x_1+v_1)\vdash'(x_2+v_2))\dashv'(\alpha_2+\beta_{2})(x_3+v_3)\\
&\quad=(x_1\vdash x_2+l_{\vdash}(x_1)v_2+r_{\vdash}(x_2)v_1)\dashv'(\alpha_2(x_3)+\beta_2(v_3))\\
&\quad=(x_1\vdash x_2)\dashv\alpha_2(x_3)+l_{\dashv}(x_1\vdash x_2)\beta_2(v_3)\\
&\quad \quad +r_\dashv(\alpha_2(x_3))l_\vdash(x_1)v_1+r_{\dashv}(\alpha_{2}(x_3))r_{\vdash}(x_2)v_1.
\\
&(\alpha_1+\beta_1)(x_1+v_1)\vdash'((x_{2}+v_{2})\dashv'(x_3+v_3))\\
&\quad =(\alpha_1(x_1)+\beta_1(v_1))\vdash'(x_2\dashv x_3+l_{\dashv}(x_{2})v_3+r_{\dashv}(x_3)v_2)\\
&\quad =\alpha_1(x_1)\vdash(x_2\dashv x_3)+l_{\vdash}(\alpha_1(x_1))l_{\dashv}(x_2)v_3\\
&\quad \quad +l_{\vdash}(\alpha_1(x_1))r_{\dashv}(x_{3})v_2+r_\vdash(x_2+x_3)\beta_1(v_1),
\end{align*}
which implies that
\begin{align*}
((x_1+v_1)\vdash'(x_2+v_2))\dashv' & (\alpha_2+\beta_{2})(x_3+v_3)= \\
&(\alpha_1+\beta_1)(x_1+v_1)\vdash'((x_{2}+v_{2})\dashv'(x_3+v_3)).
\qedhere
\end{align*}
\end{proof}
\begin{exes}
Some examples can be obtained as follows. \\
1) Let $(D,\dashv, \vdash,\alpha,\beta)$ be a BiHom-associative dialgebra. Then $(L_{\dashv},R_{\dashv},L_{\vdash},R_{\vdash},\alpha,\beta,D)$ is a bimodule of $D$, where
$L_{\dashv}(a)b=a\dashv b,~R_{\dashv}(a)b=b\dashv a,~L_{\vdash}(a)b=a\vdash b$ and $R_{\vdash}(a)b=b\vdash a$ for all $(a,b)\in D^{ 2}$. More generally, if $B$ is a two-sided BiHom-ideal of $(D,\dashv, \vdash,\alpha,\beta)$, then $(L_{\dashv},R_{\dashv},L_{\vdash},R_{\vdash},\alpha,\beta,B)$ is a bimodule of $D$, where the structure maps are $L_{\dashv}(a)x=a\dashv x=x\dashv a=R_{\dashv}(a)x$ and $L_{\vdash}(a)x=a\vdash x=x\vdash a=R_{\vdash}(a)x$ for all $x\in B$ and $(a,b)\in D^{ 2}$. \\
2)
If $(D,\dashv, \vdash)$ is an associative dialgebra and $(l_{\dashv},r_{\dashv},l_{\vdash},r_{\vdash},\alpha,\beta,V)$ is a bimodule of $D$, then $(l_{\dashv},r_{\dashv},l_{\vdash},r_{\vdash},Id_D,Id_D,V)$ is a bimodule of $\mathbb{D}$ where $\mathbb{D}=(D,\dashv, \vdash,Id_{D},
Id_{D})$ is a BiHom-associative dialgebra.
\end{exes}
\begin{prop}
If $f:(D_1,\dashv_1, \vdash_1,\alpha_1,\alpha_2)\longrightarrow(D_2,\dashv_2, \vdash_2,\beta_1,\beta_{2})$ is a morphism of BiHom-associative dialgebra, then $(l_{\dashv},r_{\dashv},l_{\vdash},r_{\vdash},\beta_1,\beta_{2},D_2)$ becomes a bimodule of $D_1$ via $f$, that is, the structure maps are defined as $l_{\dashv_1}(a)b=f(a)\dashv_2 b,~r_{\dashv_1}(a)b=b \dashv_2 f(a),~l_{\vdash_1}(a)b=f(a)\vdash_2 b$ and $r_{\vdash_1}(a)b=b \vdash_2 f(a)$ for all
$(a,b)\in D_1\times D_2$.
\end{prop}
\begin{proof}
It is obtained in a similar way as for Proposition \ref{viaf}.
\end{proof}
\begin{thm}
Let $(A, \dashv_{A}, \vdash_{A}, \alpha_{1}, \alpha_{2})$ and $(B, \dashv_{B}, \vdash_{B}, \beta_{1}, \beta_{2})$
 be BiHom-associative dialgebras. Suppose that there are linear maps
$ l_{\dashv_{A}},   r_{\dashv_{A}},  l_{\vdash_{A}},  r_{\vdash_{A}} : A \rightarrow gl(B),$
and $ l_{\dashv_{B}},   r_{\dashv_{B}},  l_{\vdash_{B}},  r_{\vdash_{B}} : B \rightarrow gl(A)$
such that
\begin{align*}
& ( l_{\dashv_{A}},   r_{\dashv_{A}},  l_{\vdash_{A}},
r_{\vdash_{A}}, \beta_{1}, \beta_{2}, B) \ \mbox{is a bimodule of} \ A, \\
& (l_{\dashv_{B}},   r_{\dashv_{B}},  l_{\vdash_{B}},  r_{\vdash_{B}},
\alpha_{1}, \alpha_{2}, A) \ \mbox{is a bimodule} \ of B,
\end{align*}
and for any $ x, y \in A, ~a, b \in B $,
\begin{eqnarray}
 \label{bieq101}
r_{\dashv_{A}}(\alpha_{2}(x))(a \vdash_{B} b) = r_{\vdash_{A}}(l_{\dashv_{B}}(b)x)\beta_{1}(a) +
\beta_{1}(a)\vdash_{B} (r_{\dashv_{A}}(x)b), \\
\label{bieq102}
\begin{array}{ll}
l_{\dashv_{A}}(l_{\vdash_{B}}(a)x)\beta_{2}(b) & + (r_{\vdash_{A}}(x)a) \dashv_{B}\beta_{2}(b)=\cr
& \beta_{1}(a)\vdash_{B} (l_{\dashv_{A}}(x)b) + r_{\vdash_{A}}(r_{\dashv_{B}}(b)x)\beta_{1}(a),
\end{array} \\
\label{bieq103}
l_{\dashv_{A}}(\alpha_{1}(x))(a \dashv_{B} b) = ( l_{\vdash_{A}}(x)a) \dashv_{B}\beta_{2}(b) +
 l_{\dashv_{A}}(r_{\vdash_{B}}(a)x)\beta_{2}(b), \\
 \label{bieq104}
r_{\dashv_{A}}(\alpha_{2}(x))(a \dashv_{B} b) = r_{\dashv_{A}}(l_{\dashv_{B}}(b)x)\beta_{1}(a) +
\beta_{1}(a)\dashv_{B} (r_{\dashv_{A}}(x)b), \\
\label{bieq105}
\begin{array}{ll}
l_{\dashv_{A}}(l_{\dashv_{B}}(a)x)\beta_{2}(b) & + (r_{\dashv_{A}}(x)a) \dashv_{B}\beta_{2}(b)=\cr
& \beta_{1}(a)\dashv_{B} (l_{\dashv_{A}}(x)b) + r_{\dashv_{A}}(r_{\dashv_{B}}(b)x)\beta_{1}(a),
\end{array} \\
\label{bieq106}
l_{\dashv_{A}}(\alpha_{1}(x))(a \dashv_{B} b) = ( l_{\dashv_{A}}(x)a) \dashv_{B}\beta_{2}(b) +
 l_{\dashv_{A}}(r_{\dashv_{B}}(a)x)\beta_{2}(b),
\\
 \label{bieq107}
r_{\dashv_{A}}(\alpha_{2}(x))(a \dashv_{B} b) = r_{\dashv_{A}}(l_{\vdash_{B}}(b)x)\beta_{1}(a) +
\beta_{1}(a)\dashv_{B} (r_{\vdash_{A}}(x)b), \\
\label{bieq108}
\begin{array}{ll}
l_{\dashv_{A}}(l_{\dashv_{B}}(a)x)\beta_{2}(b) & + (r_{\dashv_{A}}(x)a) \dashv_{B}\beta_{2}(b)=\cr
& \beta_{1}(a)\dashv_{B} (l_{\vdash_{A}}(x)b) + r_{\dashv_{A}}(r_{\vdash_{B}}(b)x)\beta_{1}(a),
\end{array} \\
\label{bieq109}
l_{\dashv_{A}}(\alpha_{1}(x))(a \vdash_{B} b) = ( l_{\dashv_{A}}(x)a) \dashv_{B}\beta_{2}(b) +
 l_{\dashv_{A}}(r_{\dashv_{B}}(a)x)\beta_{2}(b),
\\
\label{bieq110}
r_{\vdash_{A}}(\alpha_{2}(x))(a \vdash_{B} b) = r_{\vdash_{A}}(l_{\vdash_{B}}(b)x)\beta_{1}(a) +
\beta_{1}(a)\vdash_{B} (r_{\vdash_{A}}(x)b), \\
\label{bieq111}
\begin{array}{ll}
l_{\vdash_{A}}(l_{\vdash_{B}}(a)x)\beta_{2}(b) & + (r_{\vdash_{A}}(x)a) \vdash_{B}\beta_{2}(b)=\cr
& \beta_{1}(a)\vdash_{B} (l_{\vdash_{A}}(x)b) + r_{\vdash_{A}}(r_{\vdash_{B}}(b)x)\beta_{1}(a),
\end{array} \\
\label{bieq112}
l_{\vdash_{A}}(\alpha_{1}(x))(a \vdash_{B} b) = ( l_{\vdash_{A}}(x)a) \vdash_{B}\beta_{2}(b) +
 l_{\vdash_{A}}(r_{\vdash_{B}}(a)x)\beta_{2}(b),
\\
\label{bieq113}
r_{\vdash_{A}}(\alpha_{2}(x))(a \dashv_{B} b) = r_{\vdash_{A}}(l_{\vdash_{B}}(b)x)\beta_{1}(a) +
\beta_{1}(a)\vdash_{B} (r_{\vdash_{A}}(x)b), \\
\label{bieq114}
\begin{array}{ll}
l_{\vdash_{A}}(l_{\dashv_{B}}(a)x)\beta_{2}(b) & + (r_{\dashv_{A}}(x)a) \vdash_{B}\beta_{2}(b)=\cr
& \beta_{1}(a)\vdash_{B} (l_{\vdash_{A}}(x)b) + r_{\vdash_{A}}(r_{\vdash_{B}}(b)x)\beta_{1}(a),
\end{array} \\
\label{bieq115}
l_{\vdash_{A}}(\alpha_{1}(x))(a \vdash_{B} b) = ( l_{\dashv_{A}}(x)a) \vdash_{B}\beta_{2}(b) +
 l_{\vdash_{A}}(r_{\dashv_{B}}(a)x)\beta_{2}(b),
\\
 \label{bieq116}
r_{\dashv_{B}}(\beta_{2}(a))(x \vdash_{A} y) = r_{\vdash_{B}}(l_{\dashv_{A}}(y)a)\alpha_{1}(x) +
\alpha_{1}(x)\vdash_{A} (r_{\dashv_{B}}(a)y), \\
\label{bieq117}
\begin{array}{ll}
l_{\dashv_{B}}(l_{\vdash_{A}}(x)a)\alpha_{2}(y) & + (r_{\vdash_{B}}(a)x) \dashv_{A}\alpha_{2}(y)=\cr
& \alpha_{1}(x)\vdash_{B} (l_{\dashv_{B}}(a)y) + r_{\vdash_{B}}(r_{\dashv_{A}}(y)a)\alpha_{1}(x),
\end{array} \\
\label{bieq118}
l_{\dashv_{B}}(\beta_{1}(a))(x \dashv_{A} y) = ( l_{\vdash_{B}}(a)x) \dashv_{A}\alpha_{2}(y) +
 l_{\dashv_{B}}(r_{\vdash_{A}}(x)a)\alpha_{2}(y),
\\
 \label{bieq119}
r_{\dashv_{B}}(\beta_{2}(a))(x \dashv_{A} y) = r_{\dashv_{B}}(l_{\dashv_{A}}(y)a)\alpha_{1}(x) +
\alpha_{1}(x)\dashv_{A} (r_{\dashv_{B}}(a)y), \\
\label{bieq120}
\begin{array}{ll}
l_{\dashv_{B}}(l_{\dashv_{A}}(x)a)\alpha_{2}(y) & + (r_{\dashv_{B}}(a)x) \dashv_{A}\alpha_{2}(y)=\cr
& \alpha_{1}(x)\dashv_{B} (l_{\dashv_{B}}(a)y) + r_{\dashv_{B}}(r_{\dashv_{A}}(y)a)\alpha_{1}(x),
\end{array} \\
\label{bieq121}
l_{\dashv_{B}}(\beta_{1}(a))(x \dashv_{A} y) = ( l_{\dashv_{B}}(a)x) \dashv_{A}\alpha_{2}(y) +
 l_{\dashv_{B}}(r_{\dashv_{A}}(x)a)\alpha_{2}(y),
\\
 \label{bieq122}
r_{\dashv_{B}}(\beta_{2}(a))(x \dashv_{A} y) = r_{\dashv_{B}}(l_{\vdash_{A}}(y)a)\alpha_{1}(x) +
\alpha_{1}(x)\dashv_{A} (r_{\vdash_{B}}(a)y), \\
\label{bieq123}
\begin{array}{ll}
l_{\dashv_{B}}(l_{\dashv_{A}}(x)a)\alpha_{2}(y) & + (r_{\dashv_{B}}(a)x) \dashv_{A}\alpha_{2}(y)=\cr
& \alpha_{1}(x)\dashv_{B} (l_{\vdash_{B}}(a)y) + r_{\dashv_{B}}(r_{\vdash_{A}}(y)a)\alpha_{1}(x),
\end{array} \\
\label{bieq124}
l_{\dashv_{B}}(\beta_{1}(a))(x \vdash_{A} y) = ( l_{\dashv_{B}}(a)x) \dashv_{A}\alpha_{2}(y) +
 l_{\dashv_{B}}(r_{\dashv_{A}}(x)a)\alpha_{2}(y),
\\
\label{bieq125}
r_{\vdash_{B}}(\beta_{2}(a))(x \vdash_{A} y) = r_{\vdash_{B}}(l_{\vdash_{A}}(y)a)\alpha_{1}(x) +
\alpha_{1}(x)\vdash_{A} (r_{\vdash_{B}}(a)y), \\
\label{bieq126}
\begin{array}{ll}
l_{\vdash_{B}}(l_{\vdash_{A}}(x)a)\alpha_{2}(y) & + (r_{\vdash_{B}}(a)x) \vdash_{A}\alpha_{2}(y)=\cr
& \alpha_{1}(x)\vdash_{B} (l_{\vdash_{B}}(a)y) + r_{\vdash_{B}}(r_{\vdash_{A}}(y)a)\alpha_{1}(x),
\end{array} \\
\label{bieq127}
l_{\vdash_{B}}(\beta_{1}(a))(x \vdash_{A} y) = ( l_{\vdash_{B}}(a)x) \vdash_{A}\alpha_{2}(y) +
 l_{\vdash_{B}}(r_{\vdash_{A}}(x)a)\alpha_{2}(y),
\\
\label{bieq128}
r_{\vdash_{B}}(\beta_{2}(a))(x \dashv_{A} y) = r_{\vdash_{B}}(l_{\vdash_{A}}(y)a)\alpha_{1}(x) +
\alpha_{1}(x)\vdash_{A} (r_{\vdash_{B}}(a)y), \\
\label{bieq129}
\begin{array}{ll}
l_{\vdash_{B}}(l_{\dashv_{A}}(x)a)\alpha_{2}(y) & + (r_{\dashv_{B}}(a)x) \vdash_{A}\alpha_{2}(y)=\cr
& \alpha_{1}(x)\vdash_{B} (l_{\vdash_{B}}(a)y) + r_{\vdash_{B}}(r_{\vdash_{A}}(y)a)\alpha_{1}(x),
\end{array} \\
\label{bieq130}
l_{\vdash_{B}}(\beta_{1}(a))(x \vdash_{A} y) = ( l_{\dashv_{B}}(a)x) \vdash_{A}\alpha_{2}(y) +
 l_{\vdash_{B}}(r_{\dashv_{A}}(x)a)\alpha_{2}(y).
\end{eqnarray}
Then, there is a BiHom-associative dialgebra structure on the direct sum $ A \oplus B $ of the underlying vector spaces of
 $ A $ and $ B $ given by
\begin{eqnarray*}
(x + a) \dashv ( y + b ) &:=& (x \dashv_{A} y + r_{\dashv_{B}}(b)x + l_{\dashv_{B}}(a)y)\cr
&&\quad +(l_{\dashv_{A}}(x)b + r_{\dashv_{A}}(y)a + a \dashv_{B} b ), \cr
(x + a) \vdash ( y + b ) &:=& (x \vdash_{A} y + r_{\vdash_{B}}(b)x + l_{\vdash_{B}}(a)y)\cr
&&\quad + (l_{\vdash_{A}}(x)b + r_{\vdash_{A}}(y)a + a \vdash_{B} b ),\cr
(\alpha_1+\beta_1)(x+a)&:=&\alpha_1(x)+\beta_1(a),\cr
(\alpha_2+\beta_2)(x+a)&:=&\alpha_2(x)+\beta_2(a).
\end{eqnarray*}
\end{thm}
\begin{proof}
The proof is obtained in a similar way as for Theorem \ref{matched ass}.
\end{proof}
Let $ A \bowtie^{l_{\dashv_{A}}, r_{\dashv_{A}},
l_{\vdash_{A}}, r_{\vdash_{A}}, \beta_{1}, \beta_{1}}_{l_{\dashv_{B}}, r_{\dashv_{B}}, l_{\vdash_{B}},
r_{\vdash_{B}}, \alpha_{1}, \alpha_{2}} B $ denote this BiHom-associative dialgebra.

\section{Bimodules and matched pairs of BiHom- tridendriform algebras}
\label{sec:homtridendriformcoloralgebras}
In this section, we recall definitions of BiHom-dendriform and BiHom tridendriform algebras given in \cite{LiuMakhMenPan:Rota-BaxteropsBiHomassalg}.
Next we study the concept of bimodules and matched pairs of BiHom-tridendriform alge-
bra and we give some related properties.

\begin{defn}
A BiHom-dendriform algebra is a quintuple $(A, \dashv, \vdash, \alpha, \beta)$ consisting of a vector space $A$ on which the operations $\dashv, \vdash: A\otimes A\rightarrow A$ and $\alpha, \beta: A\rightarrow A$ are linear maps satisfying
\begin{eqnarray}
&&\alpha\circ\beta=\beta\circ\alpha,\\
&&\alpha(x\dashv y)=\alpha(x)\dashv\alpha(y), \alpha(x\vdash y)=\alpha(x)\vdash\alpha(y),\\
&&\beta(x\dashv y)=\beta(x)\dashv\beta(y), \beta(x\vdash y)=\beta(x)\vdash\beta(y),\\
&&(x \dashv y)\dashv \beta(z) = \alpha(x)\dashv (y \dashv z + y \vdash z), \\
&&(x\vdash y)\dashv\beta(z)=\alpha(x)\vdash(y \dashv z), \\
&&\alpha(x)\vdash (y \vdash z ) = (x \dashv y + x \vdash y)\vdash\beta(z).
\end{eqnarray}
\end{defn}
\begin{rmk}
BiHom-dendriform algebras are BiHom-X algebras.
\end{rmk}
\begin{prop}
If $(A, \dashv, \vdash, \alpha, \beta)$ is a BiHom-dendriform algebra, then $(A, \ast, \alpha, \beta)$ is a multiplicative BiHom-associative algebra,
where $ x \ast y = x \dashv y + x \vdash y $.
\end{prop}
\begin{proof}
For all $x, y, z\in A$,
\begin{align*}
(x\ast y)\ast\beta(z)&=(x\dashv y)\dashv\beta(z) + (x\dashv y)\vdash\beta(z) + (x\vdash y)\dashv\beta(z) + (x\vdash y)\vdash\beta(z)\cr
&=(x\dashv y)\dashv\beta(z) + (x\vdash y)\dashv\beta(z) + (x\dashv y)\vdash\beta(z) + (x\vdash y)\vdash\beta(z)\cr
&= (x\dashv y)\dashv\beta(z) + (x\vdash y)\dashv\beta(z) + (x\ast y)\vdash\beta(z)\cr
&= \alpha(x)\dashv(y\ast z) + \alpha(x)\vdash(y\dashv z) + \alpha(x)\vdash(y\vdash z)\cr
&= \alpha(x)\dashv(y\ast z) + \alpha(x)\vdash(y\ast z)= \alpha(x)\ast(y\ast z), \cr
\alpha(x\ast y)&=\alpha(x\vdash y) + \alpha(x\dashv y)= \alpha(x)\vdash\alpha(y) + \alpha(x)\dashv\alpha(y)= \alpha(x)\ast\alpha(y),\\
\beta(x\ast y)&=\beta(x\vdash y) + \beta(x\dashv y)= \beta(x)\vdash\beta(y) + \beta(x)\dashv\beta(y)= \beta(x)\ast\beta(y).
\qedhere \end{align*}
\end{proof}

\begin{defn}[\cite{HounkonnouHoundedjiSilvestrov:DoubleconstrbiHomFrobalg}]
Let $(A, \dashv, \vdash, \alpha_{1}, \alpha_{2})$ be a BiHom-dendriform algebra,  and $V$  be a vector space.
Let $l_{\dashv}, r_{\dashv}, l_{\vdash}, r_{\vdash} : A \rightarrow gl(V),$ and $\beta_{1}, \beta_{2}: V \rightarrow V$ be six linear maps. Then, ( $ l_{\dashv}, r_{\dashv}, l_{\vdash}, r_{\vdash}, \beta_{1}, \beta_{2}, V$) is called a bimodule of $ A $ if for any $ x, y \in A, v\in V$ and
 $ x \ast y = x \dashv y + x \vdash y,$ $l_{\ast} = l_{\dashv} + l_{\vdash},$
 $r_{\ast} = r_{\dashv} + r_{\vdash}, $ the following equalities hold:
 $$
\begin{array}{llllllllll}
l_{\dashv}(x \dashv y)\beta_{2}(v)&=& l_{\dashv}(\alpha_{1}(x))l_{\ast}(y)v,& r_{\dashv}(\alpha_{2}(x))l_{\dashv}(y)v&=&l_{\dashv}(\alpha_{1}(y))r_{\ast}(x)v,\cr
r_{\dashv}(\alpha_{2}(y))r_{\dashv}(y)v &=& r_{\dashv}(x\ast y)\beta_{1}(v),& l_{\dashv}(x \vdash y)\beta_{2}(v) &=& l_{\vdash}(\alpha_{1}(x))l_{\dashv}(y)v,\cr
r_{\dashv}(\alpha_{2}(x))l_{\vdash}(y)v &=& l_{\vdash}(\alpha_{1}(y))r_{\dashv}(x)v,& r_{\dashv}(\alpha_{2}(x))r_{\vdash}(y)v &=& r_{\vdash}(y\dashv x)\beta_{1}(v),\cr
l_{\vdash}(x\ast y)\beta_{2}(v) &=& l_{\vdash}(\alpha_{1}(x))l_{\vdash}(y)v,& r_{\vdash}(\alpha_{2}(x))l_{\ast}(y)v&=& l_{\vdash}(\alpha_{1}(y))r_{\vdash}(x)v,\cr
r_{\vdash}(\alpha_{2}(x))r_{\ast}(y)v &=& r_{\vdash}(y \vdash x)\beta_{1}(v),&
\beta_{1}(l_{\vdash}(x)v)&=& l_{\vdash}(\alpha_{1}(x))\beta_{1}(v),\\ \beta_{1}(r_{\vdash}(x)v)&=& r_{\vdash}(\alpha_{1}(x))\beta_{1}(v),&
\beta_{2}(l_{\vdash}(x)v) &=& l_{\vdash}(\alpha_{2}(x))\beta_{2}(v),\cr\beta_{2}(r_{\vdash}(x)v)&=& r_{\vdash}(\alpha_{2}(x))\beta_{2}(v),&
\beta_{1}(l_{\dashv}(x)v)&=& l_{\dashv}(\alpha_{1}(x))\beta_{1}(v),\cr \beta_{1}(r_{\dashv}(x)v)&=& r_{\dashv}(\alpha_{1}(x))\beta_{1}(v),&
\beta_{2}(l_{\dashv}(x)v) &=& l_{\dashv}(\alpha_{2}(x))\beta_{2}(v),\\
\beta_{2}(r_{\dashv}(x)v)&=&r_{\dashv}(\alpha_{2}(x))\beta_{2}(v).
\end{array}$$
\end{defn}
\begin{prop}[\cite{HounkonnouHoundedjiSilvestrov:DoubleconstrbiHomFrobalg}]
Let $(l_{\dashv}, r_{\dashv}, l_{\vdash}, r_{\vdash}, \beta_{1}, \beta_{2}, V)$ be a bimodule of a BiHom-dendri\-form algebra $(A,\dashv, \vdash, \alpha_{1}, \alpha_{2}).$
Then, on the direct sum $A\oplus V $ of the underlying vector spaces of $ A $ and $V$, there exists a BiHom-dendriform algebra structure  given for all $ x, y \in A, u, v \in V $ by
\begin{align*}
(x + u) \dashv' (y + v) &:= x \dashv y + l_{\dashv}(x)v + r_{\dashv}(y)u\cr
(x + u) \vdash' (y + v) &:= x \vdash y + l_{\vdash}(x)v + r_{\vdash}(y)u, \cr
(\alpha_1+\beta_1)(x+a)&:=\alpha_1(x)+\beta_1(a), \cr
(\alpha_2+\beta_2)(x+a)&:=\alpha_2(x)+\beta_2(a).
\end{align*}
We denote it by $ A \times_{l_{\dashv},r_{\dashv}, l_{\vdash}, r_{\vdash}, \alpha_{1}, \alpha_{2}, \beta_{1}, \beta_{2}} V$.
\end{prop}
\begin{thm}[\cite{HounkonnouHoundedjiSilvestrov:DoubleconstrbiHomFrobalg}]
Let $(A,\dashv_A,\vdash_A,\alpha_1,\alpha_2)$ and $(B,\dashv_B,\vdash_{B},\beta_1,\beta_2)$ be two BiHom-dendriform algebras. Suppose that there are linear maps $l_{\dashv_A},r_{\dashv_A},l_{\vdash_A},r_{\vdash_A}:A\rightarrow gl(B)$
and $l_{\dashv_B},r_{\dashv_B},l_{\vdash_B},r_{\vdash_B}:B\rightarrow gl(A)$ such that for all $x,y\in A,~a,b\in B$ and for
\begin{align*}x\ast_A y=x\dashv_A y + x \vdash_A y,~~~l_{A} = l_{\dashv_A} +
l_{\vdash_A},~~~r_{A} = r_{\dashv_A} + r_{\vdash_A},\\
a\ast_B b= a\dashv_B b + a \vdash_B b,~~~l_{_B} = l_{\dashv_B} +
l_{\vdash_B},~~~r_{B} = r_{\dashv_B} + r_{\vdash_B},
\end{align*}
the following equalities hold:
\begin{align}
\label{bieq35Adend}
&r_{\dashv_{A}}(\alpha_{2}(x))(a \dashv_{B} b) = \beta_{1}(a)\dashv_{B}( r_{A}(x)b) + r_{\dashv_{A}}(l_{B}(x)\beta_{1}(a)),
\\
\label{bieq36Adend}
&\begin{array}{lll}
l_{\dashv_{A}}(l_{\dashv_{B}}(x))\beta_{2}(b) &+& (r_{\dashv_{A}}(x)a) \dashv_{B}\beta_{2}(b)= \cr
&&\beta_{1}(a) \dashv_{B} (l_{\dashv_{A}}(x)b) + r_{\dashv_{A}}(r_{\dashv_{B}}(b)x)\beta_{1}(a),
\end{array} \\
\label{bieq37Adend}
&l_{\dashv_{A}}(\alpha_{1}(x))(a \ast_{B} b) = (l_{\dashv_{A}}(x)a) \ast_{B} \beta_{2}(b) +
 l_{\dashv_{A}}(r_{\dashv_{A}}(a)x)\beta_{2}(b),
\\
\label{bieq38Adend}
&r_{\dashv_{A}}(\alpha_{2}(x))(a \vdash_{B} b) = r_{\vdash_{A}}(l_{\dashv_{B}}(b)x)\beta_{1}(a) +
\beta_{1}(a)\vdash_{B} (r_{\dashv_{A}}(x)b), \\
\label{bieq39Adend}
&\begin{array}{lll}
l_{\dashv_{A}}(l_{\vdash_{B}}(a)x)\beta_{2}(b) &+& (r_{\vdash_{A}}(x)a) \dashv_{B}\beta_{2}(b)=\cr
&& \beta_{1}(a)\vdash_{B} (l_{\dashv_{A}}(x)b) + r_{\vdash_{A}}(r_{\dashv_{B}}(b)x)\beta_{1}(a),
\end{array} \\
\label{bieq40Adend}
&l_{\vdash_{A}}(\alpha_{1}(x))(a \dashv_{B} b) = ( l_{\vdash_{A}}(x)a) \dashv_{B}\beta_{2}(b) +
 l_{\dashv_{A}}(r_{\vdash_{B}}(a)x)\beta_{2}(b),
\\
\label{bieq41Adend}
&r_{\vdash_{A}}(\alpha_{2}(x))(a \ast_{B} b)= \beta_{1}(a)\vdash_{B} (r_{\vdash_{A}}(x)b) +
 r_{\vdash_{A}}(l_{\vdash_{B}}(b)x)\beta_{1}(a),
\\
\label{bieq42Adend}
&\begin{array}{lll}
\beta_{1}(a)\vdash_{B} (l_{\vdash_{A}}(x)b) &+& r_{\vdash_{A}}(r_{\vdash_{B}}(b)x)\beta_{1}(a)=\cr
&&l_{\vdash_{A}}(l_{B}(a)x)\beta_{2}(b) + (r_{A}(x)a) \vdash_{B}\beta_{2}(b),
\end{array} \\
\label{bieq43Adend}
&l_{\vdash_{A}}(\alpha_{1}(x))(a \vdash_{B} b) = (l_{A}(x)a) \vdash_{B}\beta_{2}(b) + l_{\vdash_{A}}(r_{B}(a)x)\beta_{2}(b),
\\
\label{bieq44Adend}
&r_{\dashv_{B}}(\beta_{2}(a))(x \dashv_{A} y) = \alpha_{1}(x)\dashv_{A} (r_{B}(a)y) + r_{\dashv_{B}}(l_{A}(y)a)\alpha_{1}(x),
\\
\label{bieq45Adend}
&\begin{array}{lll}
l_{\dashv_{B}}(l_{\dashv_{A}}(x)a)\alpha_{2}(y) &+& (r_{\dashv_{B}}(a)x) \dashv_{A}\alpha_{2}(y)=\cr
&&\alpha_{1}(x)\dashv_{A} (l_{B}(a)y) + r_{\dashv_{B}}(r_{A}(y)a)\alpha_{1}(x),
\end{array} \\
\label{bieq46Adend}
&l_{\dashv_{B}}(\beta_{1}(a))(x \ast_{A} y) = (l_{\dashv_{B}}(a)x) \dashv_{A}\alpha_{2}(y) +
l_{\dashv_{B}}(r_{\dashv_{A}}(x)a)\alpha_{2}(y),
\\
\label{bieq47Adend}
&r_{\dashv_{B}}(\beta_{2}(a))(x \vdash_{A} y) = r_{\vdash_{B}}(l_{\dashv_{B}}(y)a)\alpha_{1}(x) +
 \alpha_{1}(x)\vdash_{A} (r_{\dashv_{B}}(a)y),
\\
\label{bieq48Adend}
&\begin{array}{lll}
l_{\dashv_{B}}(l_{\vdash_{A}}(x)a)\alpha_{2}(y) &+& (r_{\vdash_{B}}(a)x) \dashv_{A}\alpha_{2}(y)=\cr
&&\alpha_{1}(x)\vdash_{A} (l_{\dashv_{B}}(a)y) + r_{\vdash_{B}}(r_{\dashv_{A}}(y)a)\alpha_{1}(x),
\end{array} \\
\label{bieq49Adend}
&l_{\vdash_{B}}(\beta_{1}(a))(x \dashv_{A} y) = (l_{\vdash_{B}}(a)x) \dashv_{A}\alpha_{2}(y) +
l_{\dashv_{B}}(r_{\vdash_{A}}(x)a)\alpha_{2}(y),
\\
\label{bieq50Adend}
&r_{\vdash_{B}}(\beta_{2}(a))(x \ast_{A} y)= \alpha_{1}(x)\vdash_{A} (r_{\vdash_{B}}(a)y) +
r_{\vdash_{B}}(l_{\vdash_{A}}(y)a)\alpha_{1}(x),
\\
\label{bieq51Adend}
&\begin{array}{lll}
\alpha_{1}(x)\vdash_{A} (l_{\vdash_{B}}(a)y) &+& r_{\vdash_{B}}(r_{\vdash_{A}}(y)a)\alpha_{1}(x)=\cr
&& l_{\vdash_{B}}(l_{A}(x)a)\alpha_{2}(y) + (r_{B}(a)x) \vdash_{A}\alpha_{2}(y),
\end{array} \\
\label{bieq52Adend}
&l_{\vdash_{B}}(\beta_{1}(a))(x \vdash_{A} y) = (l_{B}(a)x) \vdash_{A}\alpha_{2}(y) +
l_{\vdash_{B}}(r_{A}(x)a)\alpha_{2}(y).
\end{align}
Then $(A,B,l_{\dashv_A},r_{\dashv_A},l_{\vdash_A},r_{\vdash_A},\beta_1,\beta_2,l_{\dashv_B},r_{\dashv_B},l_{\vdash_B},r_{\vdash_B},\alpha_1,\alpha_2)$ is called a matched pair of BiHom-dendriform algebras. In this case, on the direct sum
$A\oplus B$ of the underlying vector spaces of $A$ and $B$, there exists a BiHom-dendriform algebra structure given by
\begin{align*}
(x + a) \dashv(y + b)&:=x \dashv_A y + (l_{\dashv_A}(x)b + r_{\dashv_A}(y)a)+a \dashv_B b + (l_{\dashv_B}(a)y + r_{\dashv_B}(b)x),\cr
(x + a) \vdash (y + b)&:=x \vdash_A y + (l_{\vdash_A}(x)b + r_{\vdash_A}(y)a)+a \vdash_B b + (l_{\vdash_B}(a)y + r_{\vdash_B}(b)x),\cr
(\alpha_{1}\oplus\beta_{1})(x + a)&:=\alpha_{1}(x) + \beta_{1}(a),\cr
(\alpha_{2}\oplus\beta_{2})(x + a)&:=\alpha_{2}(x) + \beta_{2}(a).
\end{align*}
\end{thm}

We denote this BiHom-dendriform algebra by $A\bowtie^{l_A,r_A,\beta_1,\beta_2}_{l_B,r_B,\alpha_1,\alpha_2}B$.
\begin{defn}\label{LMMP}
 A  BiHom-tridendriform algebra is a sextuple $(T, \dashv, \vdash, \cdot,\alpha,\beta)$ consisting of a linear space $T$,
 three bilinear maps  $\dashv, \vdash, \cdot : T\otimes T\rightarrow T$, and two linear maps $\alpha,\beta : T\rightarrow T$ satisfying for any $x, y, z\in T$,
\begin{eqnarray}
\alpha\circ\beta &=& \beta\circ\alpha,\label{t0}\\
\alpha(x\dashv y) &=& \alpha(x)\dashv\alpha(y), \alpha(x\vdash y)=\alpha(x)\vdash\alpha(y),\label{t0.0}\\
\beta(x\dashv y) &=& \beta(x)\dashv\beta(y), \beta(x\vdash y)=\beta(x)\vdash\beta(y),\label{t0.00}\\
\beta(x\cdot y)&=& \beta(x)\cdot\beta(y), \beta(x\cdot y)=\beta(x)\cdot\beta(y),\label{t0.000}\\
(x\dashv y)\dashv\beta(z)&=& \alpha(x)\dashv(y\dashv z+y\vdash z+y\cdot z),\label{t1}\\
(x\vdash y)\dashv\beta(z)&=& \alpha(x)\vdash(y\dashv z),\label{1.2}\\
\alpha(x)\vdash(y\vdash z)&=& (x\dashv y+x\vdash y+x\cdot y)\vdash\beta(z),\\
(x\dashv y)\cdot\beta(z)&=& \alpha(x)\cdot(y\vdash z),\\
(x\vdash y)\cdot\beta(z)&=& \alpha(x)\vdash(y\cdot z),\\
(x\cdot y)\dashv\beta(z)&=&\alpha(x)\cdot(y\dashv z),\label{1.6}\\
(x\cdot y)\cdot\beta(z)&=&\alpha(x)\cdot(y\cdot z).
\end{eqnarray}
 \end{defn}
\begin{rmk}
BiHom-tridendriform algebras are a BiHom-X algebras.
Also, when the BiHom-associative product ''$\cdot$'' is identically null, we get a BiHom-dendriform algebra.
\end{rmk}
\begin{prop}
Let $(T, \dashv, \vdash, \cdot,\alpha,\beta)$ be a  BiHom-tridendriform algebra. \\
Then $(T, \dashv, \vdash',\alpha,\beta)$ is a BiHom-dendriform algebra, where for any $x, y\in T$,
$$x\vdash' y:=x\vdash y+x\cdot y.$$
\end{prop}
\begin{proof}
We prove only one axiom, as others are proved similarly. For any $x, y, z\in T$,
\begin{align*}
(x\vdash'y)\dashv\beta(z)&=(x\vdash y+x\cdot y)\dashv\beta(z)\\
&=\alpha(x)\vdash(y\dashv z)+\alpha(x)\cdot(y\dashv z) \quad \mbox{(by~\eqref{1.2}~and~\eqref{1.6})}\\
&=\alpha(x)\vdash'(y\dashv z).
\qedhere \end{align*}
\end{proof}
\begin{thm}\label{car1}
Let $(A, \cdot,\alpha, \beta,R)$ be some Rota-Baxter BiHom-associative algebra of weight $\lambda$,
and three new operations $\dashv, \vdash$ and $\ast$ on $A$ are defined by
$$x\dashv y:=x\cdot R(y), \quad x\vdash y:=R(x)\cdot y, \quad x\ast y:=\lambda x\cdot y.$$
Then $(A, \dashv, \vdash, \ast, \alpha,\beta)$ is a BiHom-tridendriform algebra.
\end{thm}
\begin{proof}
 We prove only one axiom, as others are proved similarly. For any $x, y, z\in A$,
\begin{align*}
(x\vdash y)\dashv\beta(z)&=(R(x)\cdot y)\dashv\beta(z)\\
&=(R(x)\cdot y)\cdot R(\beta(z))=\alpha(R(x))\cdot (y\cdot R(z))\\
&=\alpha(R(x))\cdot (y\dashv z)=\alpha(x)\vdash (y\dashv z)=\alpha(x)\vdash (y\dashv z).
\qedhere \end{align*}
\end{proof}


In Theorem \ref{tp}, we associate a BiHom-associative algebra to any BiHom-tridendri\-form algebra.
\begin{thm}\label{tp}
If $(T, \dashv, \vdash, \cdot,\alpha,\beta)$ is a BiHom-tridendriform algebra, and $$x\ast y=x\vdash y+x\dashv y+ x\cdot y,$$ then $(T, \ast,\alpha,\beta)$ is a BiHom-associative algebra.
\end{thm}
\begin{proof}
For any $x, y, z\in T$,
 \begin{multline*}
(x\ast y)\ast \beta(z)-\alpha(x)\ast(y\ast z)=
(x\vdash y)\vdash\beta(z)+(x\dashv y)\vdash\beta(z)\\
+(x\cdot y)\vdash\beta(z)
+(x\vdash y)\dashv\beta(z)
+(x\dashv y)\dashv\beta(z)
+(x\cdot y)\dashv\beta(z)\\
+(x\vdash y)\cdot\beta(z)
+(x\dashv y)\cdot\beta(z)
+(x\cdot y)\cdot\beta(z)\\
-\alpha(x)\vdash(y\vdash z)-\alpha(x)\vdash(y\dashv z)-\alpha(x)\vdash(y\cdot z) \\
-\alpha(x)\dashv(y\vdash z)
-\alpha(x)\dashv(y\dashv z)
-\alpha(x)\dashv(y\cdot z) \\
-\alpha(x)\cdot(y\vdash z)-\alpha(x)\cdot(y\dashv z)
-\alpha(x)\cdot(y\cdot z).
 \end{multline*}
The left hand side vanishes by axioms in Definition \ref{LMMP}.
This proves that $(T, \ast, \alpha,\beta)$ is a BiHom-associative algebra.
\end{proof}

Now, we introduce the notion of bimodule of BiHom-tridendriform algebra.
\begin{defn}
Let $(T, \dashv, \vdash,\cdot, \alpha_{1}, \alpha_{2})$ be a BiHom-tridendriform algebra,  and $V$  be a vector space.
Let $l_{\dashv}, r_{\dashv}, l_{\vdash}, r_{\vdash},l_{\cdot}, r_{\cdot} : T \rightarrow gl(V),$ and $\beta_{1}, \beta_{2}: V \rightarrow V$ be eight linear maps. Then, ( $ l_{\dashv}, r_{\dashv}, l_{\vdash}, r_{\vdash},l_{\cdot}, r_{\cdot}, \beta_{1}, \beta_{2}, V$) is called a bimodule of $T $
if the following equations hold for any $ x, y \in T $ and $v\in V$:
$$\begin{array}{llllllll}
l_{\dashv}(x \dashv y)\beta_{2}(v)&=& l_{\dashv}(\alpha_{1}(x))l_{\ast}(y)v,&& r_{\dashv}(\alpha_{2}(x))l_{\dashv}(y)v&=&l_{\dashv}(\alpha_{1}(y))r_{\ast}(x)v,\\
r_{\dashv}(\alpha_{2}(y))r_{\dashv}(y)v &=& r_{\dashv}(x\ast y)\beta_{1}(v),&& l_{\dashv}(x \vdash y)\beta_{2}(v) &=& l_{\vdash}(\alpha_{1}(x))l_{\dashv}(y)v,\\
r_{\dashv}(\alpha_{2}(x))l_{\vdash}(y)v &=& l_{\vdash}(\alpha_{1}(y))r_{\dashv}(x)v,&& r_{\dashv}(\alpha_{2}(x))r_{\vdash}(y)v &=& r_{\vdash}(y\dashv x)\beta_{1}(v),\\
l_{\vdash}(x\ast y)\beta_{2}(v) &=& l_{\vdash}(\alpha_{1}(x))l_{\vdash}(y)v,&& r_{\vdash}(\alpha_{2}(x))l_{\ast}(y)v&=& l_{\vdash}(\alpha_{1}(y))r_{\vdash}(x)v,\\
r_{\vdash}(\alpha_{2}(x))r_{\ast}(y)v &=& r_{\vdash}(y \vdash x)\beta_{1}(v),&& l_{\cdot}(x\dashv y)\beta_{2}(v)&=&l_{\cdot}(\alpha_{1}(x))l_{\vdash}(y)v,\\
r_{\cdot}(\alpha_{2}(x))l_{\dashv}(y)v&=&l_{\cdot}(\alpha_{1}(y))r_{\vdash}(x)v,&& r_{\cdot}(\alpha_{2}(x))r_{\dashv}(y)v&=&r_{\cdot}(y\vdash x)\beta_{1}(v),\\
 l_{\cdot}(x\vdash y)\beta_{2}(v)&=&l_{\vdash}(\alpha_{1}(x))l_{\cdot}(y)v,&& r_{\cdot}(\alpha_{2}(x))l_{\vdash}(y)v&=&l_{\vdash}(\alpha_{1}(y))r_{\cdot}(x)v,\\
 r_{\cdot}(\alpha_{2}(x))r_{\vdash}(y)v&=&r_{\vdash}(y\cdot x)\beta_{1}(v),&& l_{\dashv}(x\cdot y)\beta_{2}(v)&=&l_{\cdot}(\alpha_{1}(x))l_{\dashv}(y)v, \\ r_{\dashv}(\alpha_{2}(x))l_{\cdot}(y)v&=&l_{\cdot}(\alpha_{1}(y))r_{\dashv}(x)v,&& r_{\dashv}(\alpha_{2}(x))r_{\cdot}(y)v&=&r_{\cdot}(y\dashv x)\beta_{1}(v),\\
  l_{\cdot}(x\cdot y)\beta_{2}(v)&=&l_{\cdot}(\alpha_{1}(x))l_{\cdot}(y)v,&& r_{\cdot}(\alpha_{2}(x))l_{\cdot}(y)v&=&l_{\cdot}(\alpha_{1}(y))r_{\cdot}(x)v, \\ r_{\cdot}(\alpha_{2}(x))r_{\cdot}(y)v&=&r_{\cdot}(y\cdot x)\beta_{1}(v),&&
\beta_{1}(l_{\vdash}(x)v)&=& l_{\vdash}(\alpha_{1}(x))\beta_{1}(v),\\ \beta_{1}(r_{\vdash}(x)v)&=& r_{\vdash}(\alpha_{1}(x))\beta_{1}(v),&&
\beta_{2}(l_{\vdash}(x)v) &=& l_{\vdash}(\alpha_{2}(x))\beta_{2}(v),\cr\beta_{2}(r_{\vdash}(x)v)&=& r_{\vdash}(\alpha_{2}(x))\beta_{2}(v),&&
\beta_{1}(l_{\dashv}(x)v)&=& l_{\dashv}(\alpha_{1}(x))\beta_{1}(v),\cr \beta_{1}(r_{\dashv}(x)v)&=& r_{\dashv}(\alpha_{1}(x))\beta_{1}(v),&&
\beta_{2}(l_{\dashv}(x)v) &=& l_{\dashv}(\alpha_{2}(x))\beta_{2}(v),\\\beta_{2}(r_{\dashv}(x)v)&=&r_{\dashv}(\alpha_{2}(x))\beta_{2}(v),
\end{array}$$
where $ x \ast y = x \dashv y + x \vdash y+x\cdot y, l_{\ast} = l_{\dashv} + l_{\vdash}+l_{\cdot}, r_{\ast} = r_{\dashv} + r_{\vdash}+r_{\cdot} $.
\end{defn}
\begin{prop}
Let $(l_{\dashv}, r_{\dashv}, l_{\vdash}, r_{\vdash},l_{\cdot}, r_{\cdot}, \beta_{1}, \beta_{2}, V)$ be a bimodule of a BiHom-tri\-dendriform algebra $(T,\dashv, \vdash,\cdot, \alpha_{1}, \alpha_{2}).$ Then, on the direct sum $T\oplus V $ of the underlying vector spaces of $T$ and $V$, there exists a BiHom-tridendriform algebra structure given, for all $ x, y \in T, u, v \in V $, by
\begin{eqnarray*}
(x + u) \dashv' (y + v) &:=& x \dashv y + l_{\dashv}(x)v + r_{\dashv}(y)u, \cr
(x + u) \vdash' (y + v) &:=& x \vdash y + l_{\vdash}(x)v + r_{\vdash}(y)u,\cr
(x + u) \cdot (y + v) &:=& x \cdot y + l_{\cdot}(x)v + r_{\cdot}(y)u.
\end{eqnarray*}
We denote it by $ T\times_{l_{\dashv},r_{\dashv}, l_{\vdash}, r_{\vdash}, l_{\cdot}, r_{\cdot},\alpha_{1}, \alpha_{2}, \beta_{1}, \beta_{2}} V$.
\end{prop}

\begin{proof}
We prove only one axiom, as others are proved similarly.
For any $x_{1},x_{2},x_{3}\in T$ and $v_1, v_2, v_3\in V$,
\begin{align*}
&((x_1+v_1)\vdash'(x_2+v_2))\dashv'(\alpha_2+\beta_{2})(x_3+v_3)\\
&\quad =(x_1\vdash x_2+l_{\vdash}(x_1)v_2+r_{\vdash}(x_2)v_1)\dashv'(\alpha_2(x_3)+\beta_2(v_3))\\
&\quad =(x_1\vdash x_2)\dashv\alpha_2(x_3)+l_{\dashv}(x_1\vdash x_2)\beta_2(v_3)\\
&\quad\quad+r_\dashv(\alpha_2(x_3))l_\vdash(x_1)v_1+r_{\dashv}(\alpha_{2}(x_3))r_{\vdash}(x_2)v_1.
\\
&(\alpha_1+\beta_1)(x_1+v_1)\vdash'((x_{2}+v_{2})\dashv'(x_3+v_3))\\
&\quad =(\alpha_1(x_1)+\beta_1(v_1))\vdash'(x_2\dashv x_3+l_{\dashv}(x_{2})v_3+r_{\dashv}(x_3)v_2)\\
&\quad =\alpha_1(x_1)\vdash(x_2\dashv x_3)+l_{\vdash}(\alpha_1(x_1))l_{\dashv}(x_2)v_3\\
&\quad \quad+l_{\vdash}(\alpha_1(x_1))r_{\dashv}(x_{3})v_2+r_\vdash(x_2+x_3)\beta_1(v_1).
\end{align*}
We deduce that $((x_1+v_1)\vdash'(x_2+v_2))\dashv'(\alpha_2+\beta_{2})(x_3+v_3)=((x_1+v_1)\vdash'(x_2+v_2))\dashv'(\alpha_2+\beta_{2})(x_3+v_3)$. This ends the proof.
\end{proof}

\begin{exes}
Some examples of bimodules of BiHom-tridendriform algebra can be constructed as follows.
\\
1)
Let $(T,\dashv, \vdash,\cdot,\alpha,\beta)$ be a BiHom-tridendriform algebra. Then $$(L_{\dashv},R_{\dashv},L_{\vdash},R_{\vdash},L_\cdot,R_\cdot,\alpha,\beta,T)
\ \mbox{is a bimodule of}\ T,$$
where for all $(a,b)\in T^{\times 2}$,
\begin{alignat*}{4}
L_{\dashv}(a)b&=a\dashv b, & \quad R_{\dashv}(a)b&=b\dashv a, \\
L_{\vdash}(a)b&=a\vdash b, & \quad R_{\vdash}(a)b&=b\vdash a, \\
L_{\cdot}(a)b&=a\cdot b, & \quad R_{\cdot}(a)b&=b\cdot a.
\end{alignat*}
More generally, if $B$ is a two-sided BiHom-ideal of $(T,\dashv, \vdash,\cdot,\alpha,\beta)$, then \\ $$(L_{\dashv},R_{\dashv},L_{\vdash},R_{\vdash},L_\cdot,R_\cdot,\alpha,\beta,B)
\ \mbox{is a bimodule of}\  T,$$
where for all $x\in B$ and $(a,b)\in T^{\times 2}$,
\begin{align*}
L_{\dashv}(a)x &= a\dashv x=x\dashv a=R_{\dashv}(a)x, \\
L_{\vdash}(a)x &= a\vdash x=x\vdash a=R_{\vdash}(a)x, \\
L_{\cdot}(a)x &= a\cdot x=x\cdot a=R_{\cdot}(a)x.
\end{align*}
2) If $(l_{\dashv},r_{\dashv},l_{\vdash},r_{\vdash},l_\cdot,r_\cdot,V)$ is a bimodule of a BiHom-tridendriform algebra $(T,\dashv, \vdash,\cdot)$ is a BiHom-tridendriform algebra, then $(l_{\dashv},r_{\dashv},l_{\vdash},r_{\vdash},l_\cdot,r_\cdot,Id_{V},Id_{V},V)$ is a bimodule of $\mathbb{T}$, where $\mathbb{T}=(T,\dashv, \vdash,\cdot,Id_{T},
Id_{T})$ is a BiHom-tridendriform algebra.
\end{exes}
\begin{prop}
If $f:(T,\dashv_1, \vdash_1,\cdot_1,\alpha_{1},\alpha_2)\longrightarrow(T',\dashv_2, \vdash_2,\cdot_2,\beta_{1},\beta_{2})$ is a morphism of BiHom-tridendriform algebras, then $(l_{\dashv_1},r_{\dashv_1},l_{\vdash_1},r_{\vdash_1},l_{\cdot_1},r_{\cdot_1},\beta_1,\beta_{2},T')$ becomes a bimodule of $T$ via $f$, that is, for all
$(a,b)\in T\times T'$,
\begin{alignat*}{4}
l_{\dashv_1}(a)b &=f(a)\dashv_2 b, &\quad r_{\dashv_1}(a)b &=b \dashv_2 f(a),\\
l_{\vdash_1}(a)b &=f(a)\vdash_2 b, &\quad r_{\vdash_1}(a)b &=b \vdash_2 f(a),\\
l_{\cdot_1}(a)b  &=f(a)\cdot_2 b,  &\quad r_{\cdot_1}(a)b  &=b \cdot_2 f(a).
\end{alignat*}
\end{prop}
\begin{proof}
We prove only one axiom, since other axioms are proved similarly.
For any $x,y\in T$ and $z\in T'$,
\begin{align*}
l_{\vdash_1}(x\dashv_1 y)\beta_2(z)&=f(x\dashv_1 y)\dashv_2\beta_2(z)\\
&=(f(x)\dashv_2 f(y))\dashv_2 \beta_2(z)=\beta_1 f(x)\dashv_2(f(y)\ast_{2} z)\\
&=f(\alpha_1(x))\dashv_2 l_{\ast_1}(y)z=l_{\dashv_1}(\alpha_{1}(x))l_{\ast_1}(y)z.
\qedhere
\end{align*}
\end{proof}
\begin{defn}
An abelian extension of BiHom-tridendriform algebra is a short exact sequence of BiHom-tridendriform algebra
$$0\longrightarrow (V,\alpha_{V},\beta_{V})\stackrel{\mbox{i}} \longrightarrow(T,\dashv_T, \vdash_T,\cdot_T,\alpha_{T},\beta_{T})\stackrel{\mbox{$\pi$}}\longrightarrow (T',\dashv_{T'}, \vdash_{T'},\cdot_{T'},\alpha_{T'},\beta_{T'})\longrightarrow 0 ,$$
where $(V,\alpha_{V},\beta_{V})$ is a trivial BiHom-tridendriform algebra, $i$ and $\pi$ are morphisms of BiHom-algebras. Furthermore, if there is a morphism $s:(T',\dashv_{T'}, \vdash_{T'},\cdot_{T'},\alpha_{T'},\beta_{T'})
\longrightarrow (T,\dashv_T, \vdash_T,\cdot_T,\alpha_{T},\beta_{T})$ such that $\pi\circ s=id_{T'}$, then the abelian extension is said to be split and $s$ is called a section of $\pi$.
\end{defn}
\begin{rmk} Consider the  split null extension $T\oplus V$ determined by the bimodule\\ $(l_{\dashv},r_{\dashv},l_{\vdash},r_{\vdash},l_\cdot,r_\cdot,\alpha_V,\beta_V,V)$ for the BiHom-tridendriform algebra
$(T,\dashv_T, \vdash_T,\cdot_T,\alpha,\beta)$ in the previous proposition. Write  elements $a+v$ of $T\oplus V$ as $(a,v).$ Then there is an injective homomorphism of BiHom-modules
$i :V\rightarrow T\oplus V $ given by $i(v)=(0,v)$ and a surjective homomorphism of BiHom-modules $\pi : T\oplus V\rightarrow T$ given by $\pi(a,v)=a.$
Moreover, $i(V)$ is a two-sided BiHom-ideal of $T\oplus V$  such that $T\oplus V/i(V)\cong T$. On the other hand, there is a morphism of BiHom-algebras
$\sigma: T\rightarrow T\oplus V$ given by $\sigma(a)=(a,0)$ which is clearly a section of $\pi.$ Hence, we obtain the abelian split exact sequence of
BiHom-tridendriform algebra and $(l_{\dashv},r_{\dashv},l_{\vdash},r_{\vdash},l_\cdot,r_\cdot, \alpha_V,\beta_{V},V)$ is a bimodule for $T$ via $\pi.$
 \end{rmk}
\begin{prop}\label{propa}
Let ($ l_{\dashv}, r_{\dashv}, l_{\vdash}, r_{\vdash},l_{\cdot}, r_{\cdot}, \beta_{1}, \beta_{2}, V$) be a bimodule of a BiHom-triden\-driform algebra $(T, \dashv, \vdash,\cdot, \alpha_{1}, \alpha_{2})$. Let $(T, \ast, \alpha_{1}, \alpha_{2})$ be the associated BiHom-associative algebra. Then,
$( l_{\dashv}+l_{\vdash}+l_{\cdot}, r_{\dashv}+r_{\vdash}+r_{\cdot},\beta_{1}, \beta_{2}, V)$ is a bimodule of $(T, \ast, \alpha_{1}, \alpha_{2})$.
\end{prop}
\begin{proof}
We prove only one axiom. The other axioms are proved similarly. For any $x, y \in A, v \in V$,
\begin{align*}
&l_\ast(x\ast y)\beta_2(v)=(l_\dashv+l_\vdash+l_\cdot)(x\ast y)\beta_2(v) =(l_\dashv+l_\vdash+l_\cdot)(x\dashv y+x\vdash y+x\cdot y)\beta_2(v)\\
&\quad=l_\dashv(x\dashv y)\beta_2(v)+l_\dashv(x\vdash y)\beta_2(v)+l_\dashv(x\cdot y)\beta_2(v)+l_\vdash(x\ast y)\beta_2(v)\\
&\quad\quad +l_\cdot(x\dashv y)\beta_2(v)+l_\cdot(x\vdash y)\beta_2(v)+l_\cdot(x\cdot y)\beta_2(v)\\
&\quad =l_\dashv(\alpha_1(x))l_\ast(y)v+l_\vdash(\alpha_1(x))l_\dashv(y)v+l_\cdot(\alpha_1(x))l_\dashv(y)v+l_\vdash(\alpha_1(x))l_\vdash(y)v\\
&\quad\quad+l_\cdot(\alpha_1(x))l_\vdash(y)v+l_\vdash(\alpha_1(x))l_\cdot(y)v+l_\cdot(\alpha_1(x))l_\cdot(y)v\\
&\quad=(l_\dashv+l_\vdash+l_\cdot)(\alpha_1(x))(l_\dashv+l_\vdash+l_\cdot)(y)v=l_\ast(\alpha_1(x))l_\ast(y)v.
\qedhere \end{align*}
\end{proof}
\begin{thm}\label{mamm1}
Let $(T,\dashv, \vdash,\cdot,\alpha_1,\alpha_2)$ be a BiHom-tridendriform algebra, and   $V_{\beta_1,\beta_2}=(l_{\dashv},r_{\dashv},l_{\vdash},r_{\vdash},l_{\cdot},r_{\cdot},\beta_1,\beta_2,V)$ be a bimodule of $T$. Let $\alpha'_1,\alpha'_2$ be two endomorphisms of $T$ such that any two of the maps $\alpha_1,\alpha'_1,\alpha_2,\alpha'_2$ commute
and $\beta'_1,~\beta'_2$ be linear maps of $V$ such that any two of the maps $\beta_1,\beta'_1,\beta_2,\beta'_2$ commute. Suppose furthermore that
$$\left\{
   \begin{array}{lllllll}
    \beta'_1\circ l_\dashv=(l_\dashv\circ\alpha'_1)\beta'_1,~~
     \beta'_2\circ l_\dashv=(l_\dashv\circ\alpha'_2)\beta'_2,& \\
         \beta'_1\circ l_\vdash=(l_\vdash\circ\alpha'_1)\beta'_1,~~
     \beta'_2\circ l_\vdash=(l_\vdash\circ\alpha'_2)\beta'_2,&\\
      \beta'_1\circ l_\cdot=(l_\cdot\circ\alpha'_1)\beta'_1,~~
     \beta'_2\circ l_\cdot=(l_\cdot\circ\alpha'_2)\beta'_2,& \\
   \end{array}
 \right.$$
$$\left\{
   \begin{array}{lllllll}
    \beta'_1\circ r_\dashv=(r_\dashv\circ\alpha'_1)\beta'_1,~~
     \beta'_2\circ r_\dashv=(r_\dashv\circ\alpha'_2)\beta'_2,& \\
         \beta'_1\circ r_\vdash=(r_\vdash\circ\alpha'_1)\beta'_1,~~
     \beta'_2\circ r_\vdash=(r_\vdash\circ\alpha'_2)\beta'_2,&\\
      \beta'_1\circ r_\cdot=(r_\cdot\circ\alpha'_1)\beta'_1,~~
     \beta'_2\circ r_\cdot=(r_\cdot\circ\alpha'_2)\beta'_2,&
   \end{array}
 \right.$$
and write $T_{\alpha'_1,\alpha'_2}=(T,\dashv_{\alpha'_1,\alpha'_2}, \vdash_{\alpha'_1,\alpha'_2},\cdot_{\alpha'_1,\alpha'_2},\alpha_1\alpha'_1,\alpha_2\alpha'_2)$
for the BiHom-tridendriform algebra, and
$V_{\beta'_1,\beta'_2}=(\widetilde{l}_{\dashv},\widetilde{r}_{\dashv},\widetilde{l}_{\vdash},
\widetilde{r}_{\vdash},\widetilde{l}_{\cdot},\widetilde{r}_{\cdot},\beta_1\beta'_1,\beta_2\beta'_2,V)$, where
\begin{equation}
\begin{array}{llll}
\widetilde{l}_{\dashv}&=(l_{\dashv}\circ\alpha'_1)\beta'_2,  &\widetilde{r}_{\dashv}&=(r_{\dashv}\circ\alpha'_2)\beta'_1, \\ \widetilde{l}_{\vdash}&=(l_{\vdash}\circ\alpha'_1)\beta'_2,
&\widetilde{r}_{\vdash}&=(r_{\vdash}\circ\alpha'_2)\beta'_1,\\ \widetilde{l}_{\cdot}&=(l_{\cdot}\circ\alpha'_1)\beta'_2, &\widetilde{r}_{\cdot}&=(r_{\cdot}\circ\alpha'_2)\beta'_1.
\end{array}
\end{equation}
This gives the BiHom-module $V_{\beta'_1,\beta'_2}$ the structure of $T_{\alpha'_1,\alpha'_2}$-bimodule.
\end{thm}
\begin{proof}
We prove only one axiom, since other axioms are proved similarly.
For any $x,y\in T$ and $v\in V$,
\begin{align*}
\widetilde{l}_{\dashv}(x\dashv_{\alpha'_1\alpha'_2}y)\beta_2\beta'_2(v)
&=\widetilde{l}_{\dashv}(\alpha'(x)\dashv_{\alpha'_1\alpha'_2}\alpha'(y))\beta_2\beta'_2(v)\\
&=l_{\dashv}(\alpha'{2}_{1}^(x)\dashv\alpha'_1\alpha'_2(y))\beta_2\beta'{2}_{2}^(v)\\
&=l_{\dashv}(\alpha_1\alpha'^{2}_{1}(x))l_{\ast}(\alpha'_1\alpha'_2(y))\beta'^{2}_{2}(v)\\
&=\widetilde{l}_{\dashv}(\alpha_1\alpha'_1(x))l_{\ast}(\alpha'_1(y))\beta'_2(v)\\
&=\widetilde{l}_{\dashv}(\alpha_1\alpha'_1(x))\widetilde{l}_{\ast}(y)v.
\qedhere
\end{align*}
\end{proof}
Let ($ l_{\dashv}, r_{\dashv}, l_{\vdash}, r_{\vdash},l_{\cdot}, r_{\cdot}, \beta_{1}, \beta_{2}, V$) be a bimodule of a BiHom-tridendriform algebra $(T, \dashv, \vdash,\cdot, \alpha_{1}, \alpha_{2})$ and $l_{\dashv}^{\ast}, r_{\dashv}^{\ast}, l_{\vdash}^{\ast}, r_{\vdash}^{\ast},l_{\cdot}^{\ast}, r_{\cdot}^{\ast}:T\rightarrow gl(V^{\ast}).$ Let  $\alpha_1^{\ast},\alpha_{2}^{\ast}:T^{\ast}\rightarrow T^{\ast},~~\beta_{1}^{\ast},\beta_{2}^{\ast}:V^{\ast}\rightarrow V^{\ast}$ be the dual maps of respectively $\alpha_1,\alpha_2,\beta_1$ and $\beta_2$ such that
$$\begin{array}{llllllll}
  \langle l_{\dashv}^{\ast}(x)u^{\ast},v\rangle =\langle u^{\ast},l_{\dashv}(x)v\rangle,&& \langle r^{\ast}_{\vdash}(x)u^{\ast},v\rangle =\langle u^{\ast},r_{\vdash}(x)v\rangle,\\
   \langle l_{\vdash}^{\ast}(x)u^{\ast},v\rangle =\langle u^{\ast},l_{\vdash}(x)v\rangle,&& \langle r^{\ast}_{\vdash}(x)u^{\ast},v\rangle =\langle u^{\ast},r_{\vdash}(x)v\rangle,\\
    \langle l_{\cdot}^{\ast}(x)u^{\ast},v\rangle =\langle u^{\ast},l_{\cdot}(x)v\rangle,&& \langle r^{\ast}_{\cdot}(x)u^{\ast},v\rangle =\langle u^{\ast},r_{\cdot}(x)v\rangle,\\
    \alpha_{1}^{\ast}(x^{\ast}(y)=x^{\ast}(\alpha_{1}(y)),&& \alpha_{2}^{\ast}(x^{\ast}(y)=x^{\ast}(\alpha_{2}(y)),\\
     \beta_{1}^{\ast}(u^{\ast}(v)=u^{\ast}(\beta_{1}(v)),&& \beta_{2}^{\ast}(u^{\ast}(v)=u^{\ast}(\beta_{2}(v)).
\end{array}$$
\begin{prop}
Let ($ l_{\dashv}, r_{\dashv}, l_{\vdash}, r_{\vdash},l_{\cdot}, r_{\cdot}, \beta_{1}, \beta_{2}, V$) be a bimodule of a BiHom-tri\-dendriform algebra $(T, \dashv, \vdash,\cdot, \alpha_{1}, \alpha_{2})$. Then ($ l_{\dashv}^{\ast}, r_{\dashv}^{\ast}, l_{\vdash}^{\ast}, r_{\vdash}^{\ast},l_{\cdot}^{\ast}, r_{\cdot}^{\ast}, \beta_{1}^{\ast}, \beta_{2}^{\ast}, V^{\ast}$) is a bimodule of
$(T, \dashv, \vdash,\cdot, \alpha_{1}, \alpha_{2})$ provided that

$$\begin{array}{llllllll}
\beta_{2}(l_{\dashv}(x \dashv y))u&=& l_{\ast}(y)l_{\dashv}(\alpha_{1}(x))u,&& l_{\dashv}(y)r_{\dashv}(\alpha_{2}(x))u&=&r_{\ast}(x)l_{\dashv}(\alpha_{1}(y))u,\\
r_{\dashv}(y)r_{\dashv}(\alpha_{2}(y))u &=& \beta_{1}(r_{\dashv}(x\ast y))u,&& \beta_{2}(l_{\dashv}(x \vdash y))u &=&l_{\dashv}(y) l_{\vdash}(\alpha_{1}(x))u,\\
l_{\vdash}(y)r_{\dashv}(\alpha_{2}(x))u &=& r_{\dashv}(x)l_{\vdash}(\alpha_{1}(y))u,&& r_{\vdash}(y)r_{\dashv}(\alpha_{2}(x))u &=& \beta_{1}(r_{\vdash}(y\dashv x))u,\\
\beta_{2}(l_{\vdash}(x\ast y))u &=& l_{\vdash}(y)l_{\vdash}(\alpha_{1}(x))u,&& l_{\ast}(y)r_{\vdash}(\alpha_{2}(x))u&=& r_{\vdash}(x)l_{\vdash}(\alpha_{1}(y))u,\\
r_{\ast}(y)r_{\vdash}(\alpha_{2}(x))u &=& \beta_{1}(r_{\vdash}(y \vdash x))u,&& \beta_{2}(l_{\cdot}(x\dashv y))u&=&l_{\vdash}(y)l_{\cdot}(\alpha_{1}(x))u,\\
l_{\dashv}(y)r_{\cdot}(\alpha_{2}(x))u&=&r_{\vdash}(x)l_{\cdot}(\alpha_{1}(y))u,&&r_{\dashv}(y) r_{\cdot}(\alpha_{2}(x))u&=&\beta_{1}(r_{\cdot}(y\vdash x))u,\\
 \beta_{2}(l_{\cdot}(x\vdash y))u&=&l_{\cdot}(y)l_{\vdash}(\alpha_{1}(x))u,&& l_{\vdash}(y)r_{\cdot}(\alpha_{2}(x))u&=&r_{\cdot}(x)l_{\vdash}(\alpha_{1}(y))u,\\
r_{\vdash}(y) r_{\cdot}(\alpha_{2}(x))u&=&\beta_{1}(r_{\vdash}(y\cdot x))u,&& \beta_{2}(l_{\dashv}(x\cdot y))u&=&l_{\dashv}(y)l_{\cdot}(\alpha_{1}(x))u, \\ l_{\cdot}(y)r_{\dashv}(\alpha_{2}(x))u&=&r_{\dashv}(x)l_{\cdot}(\alpha_{1}(y))u,&&r_{\cdot}(y) r_{\dashv}(\alpha_{2}(x))u&=&\beta_{1}(r_{\cdot}(y\dashv x))u,\\
  \beta_{2}(l_{\cdot}(x\cdot y))u&=&l_{\cdot}(y)l_{\cdot}(\alpha_{1}(x))u,&& l_{\cdot}(y)r_{\cdot}(\alpha_{2}(x))u&=&r_{\cdot}(x)l_{\cdot}(\alpha_{1}(y))u, \\ r_{\cdot}(y)r_{\cdot}(\alpha_{2}(x))u&=&\beta_{1}(r_{\cdot}(y\cdot x))u.
\end{array}$$
where $ x \ast y = x \dashv y + x \vdash y+x\cdot y, l_{\ast} = l_{\dashv} + l_{\vdash}+l_{\cdot}, r_{\ast} = r_{\dashv} + r_{\vdash}+r_{\cdot}, $
for all $x,y\in T$ and $u\in V$.
\end{prop}

\begin{thm}
Let $(A, \dashv_{A}, \vdash_{A},\cdot_A, \alpha_{1}, \alpha_{2})$ and $(B, \dashv_{B}, \vdash_{B}, \cdot_B,\beta_{1}, \beta_{2})$
 be two BiHom-tridendriform algebras. Suppose that there are linear maps
\begin{align*}
& l_{\dashv_{A}},   r_{\dashv_{A}},  l_{\vdash_{A}}, r_{\vdash_{A}},l_{\cdot_A},r_{\cdot_A} : A \rightarrow gl(B), \\
& l_{\dashv_{B}},   r_{\dashv_{B}},  l_{\vdash_{B}},  r_{\vdash_{B}},l_{\cdot_B},r_{\cdot_B} : B \rightarrow gl(A),
\end{align*}
such that
\begin{eqnarray*}
(l_{\dashv_{A}},   r_{\dashv_{A}},  l_{\vdash_{A}},
r_{\vdash_{A}},l_{\cdot_A},r_{\cdot_A}, \beta_{1}, \beta_{2}, B) \ \mbox{is a bimodule of}\ A, \\
(l_{\dashv_{B}},   r_{\dashv_{B}},  l_{\vdash_{B}},  r_{\vdash_{B}},l_{\cdot_B},r_{\cdot_B}
\alpha_{1}, \alpha_{2}, A) \ \mbox{is a bimodule  of}\ B,
\end{eqnarray*}
and for
\begin{align*}
&x\ast_A y=x\dashv_A y+x\vdash_A y+x\cdot_A y,~l_{A} = l_{\dashv_{A}} +
 l_{\vdash_{A}}+l_{\cdot_{A}}, r_{A} =  r_{\dashv_{A}} +  r_{\vdash_{A}}+r_{\cdot_{A}},\\
 &a\ast_B b=a\dashv_B b+a\vdash_B b+a\cdot_B b,~l_{B} =
 l_{\dashv_{B}} +  l_{\vdash_{B}}+l_{\cdot_{B}} , r_{B} =  r_{\dashv_{B}} +  r_{\vdash_{B}}+r_{\cdot_{B}}.\end{align*}
and for any $ x, y \in A,~ a, b \in B $,
\begin{eqnarray}
\label{bieq201}
r_{\dashv_{A}}(\alpha_{2}(x))(a \dashv_{B} b) = r_{A(l_{B}}(b)x)\beta_{1}(a) +
\beta_{1}(a)\dashv_{B} (r_{\dashv_{A}}(x)b), \\
\label{bieq202}
\begin{array}{ll}
l_{\dashv_{A}}(l_{\dashv_{B}}(a)x)\beta_{2}(b) & + (r_{\dashv_{A}}(x)a) \dashv_{B}\beta_{2}(b)=\cr
& \beta_{1}(a)\dashv_{B} (l_{\dashv_{A}}(x)b) + r_{\dashv_{A}}(r_{\dashv_{B}}(b)x)\beta_{1}(a),
\end{array} \\
\label{bieq203}
l_{\dashv_{A}}(\alpha_{1}(x))(a \ast_{B} b) = ( l_{\dashv_{A}}(x)a) \ast_{B}\beta_{2}(b) +
 l_{\dashv_{A}}(r_{\dashv_{B}}(a)x)\beta_{2}(b), \\
 \label{bieq204}
r_{\dashv_{A}}(\alpha_{2}(x))(a \vdash_{B} b) = r_{\vdash_{A}}(l_{\dashv_{B}}(b)x)\beta_{1}(a) +
\beta_{1}(a)\vdash_{B} (r_{\dashv_{A}}(x)b), \\
\label{bieq205}
\begin{array}{ll}
l_{\dashv_{A}}(l_{\vdash_{B}}(a)x)\beta_{2}(b) & + (r_{\vdash_{A}}(x)a) \dashv_{B}\beta_{2}(b)=\cr
& \beta_{1}(a)\vdash_{B} (l_{\dashv_{A}}(x)b) + r_{\vdash_{A}}(r_{\dashv_{B}}(b)x)\beta_{1}(a),
\end{array} \\
\label{bieq206}
l_{\dashv_{A}}(\alpha_{1}(x))(a \dashv_{B} b) = ( l_{\vdash_{A}}(x)a) \dashv_{B}\beta_{2}(b) +
 l_{\dashv_{A}}(r_{\vdash_{B}}(a)x)\beta_{2}(b), \\
\label{bieq207}
r_{\vdash_{A}}(\alpha_{2}(x))(a \ast_{B} b) = r_{\vdash_{A}}(l_{\vdash_{B}}(b)x)\beta_{1}(a) +
\beta_{1}(a)\vdash_{B} (r_{\vdash_{A}}(x)b), \\
\label{bieq208}
\begin{array}{ll}
l_{\vdash_{A}}(l_{A}(a)x)\beta_{2}(b) & + (r_{A}(x)a) \vdash_{B}\beta_{2}(b)=\cr
& \beta_{1}(a)\vdash_{B} (l_{\vdash_{A}}(x)b) + r_{\vdash_{A}}(r_{\vdash_{B}}(b)x)\beta_{1}(a),
\end{array} \\
\label{bieq209}
l_{\vdash_{A}}(\alpha_{1}(x))(a \vdash_{B} b) = ( l_{\vdash_{A}}(x)a) \vdash_{B}\beta_{2}(b) +
 l_{A}(r_{B}(a)x)\beta_{2}(b), \\
 \label{bieq210}
r_{\cdot_{A}}(\alpha_{2}(x))(a \dashv_{B} b) = r_{\cdot_{A}}(l_{\vdash_{B}}(b)x)\beta_{1}(a) +
\beta_{1}(a)\cdot_{B} (r_{\vdash_{A}}(x)b), \\
\label{bieq211}
\begin{array}{ll}
l_{\cdot_{A}}(l_{\dashv_{B}}(a)x)\beta_{2}(b) & + (r_{\dashv_{A}}(x)a) \cdot_{B}\beta_{2}(b)=\cr
& \beta_{1}(a)\cdot_{B} (l_{\vdash_{A}}(x)b) + r_{\cdot_{A}}(r_{\vdash_{B}}(b)x)\beta_{1}(a),
\end{array} \\
\label{bieq212}
l_{\cdot_{A}}(\alpha_{1}(x))(a \vdash_{B} b) = ( l_{\dashv_{A}}(x)a) \cdot_{B}\beta_{2}(b) +
 l_{\cdot_{A}}(r_{\dashv_{B}}(a)x)\beta_{2}(b), \\
\label{bieq213}
r_{\cdot_{A}}(\alpha_{2}(x))(a \vdash_{B} b) = r_{\vdash_{A}}(l_{\cdot_{B}}(b)x)\beta_{1}(a) +
\beta_{1}(a)\vdash_{B} (r_{\cdot_{A}}(x)b), \\
\label{bieq214}
\begin{array}{ll}
l_{\cdot_{A}}(l_{\vdash_{B}}(a)x)\beta_{2}(b) & + (r_{\vdash_{A}}(x)a) \cdot_{B}\beta_{2}(b)=\cr
& \beta_{1}(a)\vdash_{B} (l_{\cdot_{A}}(x)b) + r_{\vdash_{A}}(r_{\cdot_{B}}(b)x)\beta_{1}(a),
\end{array} \\
\label{bieq215}
l_{\vdash_{A}}(\alpha_{1}(x))(a \cdot_{B} b) = ( l_{\vdash_{A}}(x)a) \cdot_{B}\beta_{2}(b) +
 l_{\cdot_{A}}(r_{\vdash_{B}}(a)x)\beta_{2}(b), \\
 \label{bieq216}
r_{\dashv_{A}}(\alpha_{2}(x))(a \cdot_{B} b) = r_{\cdot_{A}}(l_{\dashv_{B}}(b)x)\beta_{1}(a) +
\beta_{1}(a)\cdot_{B} (r_{\dashv_{A}}(x)b), \\
\label{bieq217}
\begin{array}{ll}
l_{\dashv_{A}}(l_{\cdot_{B}}(a)x)\beta_{2}(b) & + (r_{\cdot_{A}}(x)a) \dashv_{B}\beta_{2}(b)=\cr
& \beta_{1}(a)\cdot_{B} (l_{\dashv_{A}}(x)b) + r_{\cdot_{A}}(r_{\dashv_{B}}(b)x)\beta_{1}(a),
\end{array} \\
\label{bieq218}
l_{\cdot_{A}}(\alpha_{1}(x))(a \dashv_{B} b) = ( l_{\cdot_{A}}(x)a) \dashv_{B}\beta_{2}(b) +
 l_{\dashv_{A}}(r_{\cdot_{B}}(a)x)\beta_{2}(b), \\
 \label{bieq219}
r_{\cdot_{A}}(\alpha_{2}(x))(a \cdot_{B} b) = r_{\cdot_{A}}(l_{\cdot_{B}}(b)x)\beta_{1}(a) +
\beta_{1}(a)\cdot_{B} (r_{\cdot_{A}}(x)b), \\
\label{bieq220}
\begin{array}{ll}
l_{\cdot_{A}}(l_{\cdot_{B}}(a)x)\beta_{2}(b) & + (r_{\cdot_{A}}(x)a) \cdot_{B}\beta_{2}(b)=\cr
& \beta_{1}(a)\cdot_{B} (l_{\cdot_{A}}(x)b) + r_{\cdot_{A}}(r_{\cdot_{B}}(b)x)\beta_{1}(a),
\end{array} \\
\label{bieq221}
l_{\cdot_{A}}(\alpha_{1}(x))(a \cdot_{B} b) = ( l_{\cdot_{A}}(x)a) \cdot_{B}\beta_{2}(b) +
 l_{\cdot_{A}}(r_{\cdot_{B}}(a)x)\beta_{2}(b),\\
 \label{bieq222}
r_{\dashv_{B}}(\beta_{2}(a))(x \dashv_{A} y) = r_{B(l_{A}}(y)a)\alpha_{1}(x) +
\alpha_{1}(x)\dashv_{A} (r_{\dashv_{B}}(a)y), \\
\label{bieq223}
\begin{array}{ll}
l_{\dashv_{B}}(l_{\dashv_{A}}(x)a)\alpha_{2}(y) & + (r_{\dashv_{B}}(a)x) \dashv_{A}\alpha_{2}(y)=\cr
& \alpha_{1}(x)\dashv_{B} (l_{\dashv_{B}}(a)y) + r_{\dashv_{B}}(r_{\dashv_{A}}(y)a)\alpha_{1}(x),
\end{array} \\
\label{bieq224}
l_{\dashv_{B}}(\beta_{1}(a))(x \ast_{A} y) = ( l_{\dashv_{B}}(a)x) \ast_{A}\alpha_{2}(y) +
 l_{\dashv_{B}}(r_{\dashv_{A}}(x)a)\alpha_{2}(y),\\
 \label{bieq225}
r_{\dashv_{B}}(\beta_{2}(a))(x \vdash_{A} y) = r_{\vdash_{B}}(l_{\dashv_{A}}(y)a)\alpha_{1}(x) +
\alpha_{1}(x)\vdash_{A} (r_{\dashv_{B}}(a)y), \\
\label{bieq226}
\begin{array}{ll}
l_{\dashv_{B}}(l_{\vdash_{A}}(x)a)\alpha_{2}(y) & + (r_{\vdash_{B}}(a)x) \dashv_{A}\alpha_{2}(y)=\cr
& \alpha_{1}(x)\vdash_{B} (l_{\dashv_{B}}(a)y) + r_{\vdash_{B}}(r_{\dashv_{A}}(y)a)\alpha_{1}(x),
\end{array} \\
\label{bieq227}
l_{\dashv_{B}}(\beta_{1}(a))(x \dashv_{A} y) = ( l_{\vdash_{B}}(a)x) \dashv_{A}\alpha_{2}(y) +
 l_{\dashv_{B}}(r_{\vdash_{A}}(x)a)\alpha_{2}(y),\\
\label{bieq228}
r_{\vdash_{B}}(\beta_{2}(a))(x \ast_{A} y) = r_{\vdash_{B}}(l_{\vdash_{A}}(y)a)\alpha_{1}(x) +
\alpha_{1}(x)\vdash_{A} (r_{\vdash_{B}}(a)y), \\
\label{bieq229}
\begin{array}{ll}
l_{\vdash_{B}}(l_{B}(x)a)\alpha_{2}(y) & + (r_{B}(a)x) \vdash_{A}\alpha_{2}(y)=\cr
& \alpha_{1}(x)\vdash_{B} (l_{\vdash_{B}}(a)y) + r_{\vdash_{B}}(r_{\vdash_{A}}(y)a)\alpha_{1}(x),
\end{array} \\
\label{bieq230}
l_{\vdash_{B}}(\beta_{1}(a))(x \vdash_{A} y) = ( l_{\vdash_{B}}(a)x) \vdash_{A}\alpha_{2}(y) +
 l_{B}(r_{A}(x)a)\alpha_{2}(y),\\
 \label{bieq231}
r_{\cdot_{B}}(\beta_{2}(a))(x \dashv_{A} y) = r_{\cdot_{B}}(l_{\vdash_{A}}(y)x)\alpha_{1}(x) +
\alpha_{1}(x)\cdot_{A} (r_{\vdash_{B}}(a)y), \\
\label{bieq232}
\begin{array}{ll}
l_{\cdot_{B}}(l_{\dashv_{A}}(x)a)\alpha_{2}(y) & + (r_{\dashv_{B}}(a)x) \cdot_{A}\alpha_{2}(y)=\cr
& \alpha_{1}(x)\cdot_{B} (l_{\vdash_{B}}(a)y) + r_{\cdot_{B}}(r_{\vdash_{A}}(y)a)\alpha_{1}(x),
\end{array} \\
\label{bieq233}
l_{\cdot_{B}}(\beta_{1}(a))(x \vdash_{A} y) = ( l_{\dashv_{B}}(a)x) \cdot_{A}\alpha_{2}(y) +
 l_{\cdot_{B}}(r_{\dashv_{A}}(x)a)\alpha_{2}(y),\\
\label{bieq234}
r_{\cdot_{B}}(\beta_{2}(a))(x \vdash_{A} y) = r_{\vdash_{B}}(l_{\cdot_{A}}(y)a)\alpha_{1}(x) +
\alpha_{1}(x)\vdash_{A} (r_{\cdot_{B}}(a)y), \\
\label{bieq235}
\begin{array}{ll}
l_{\cdot_{B}}(l_{\vdash_{A}}(x)a)\alpha_{2}(y) & + (r_{\vdash_{B}}(a)x) \cdot_{A}\alpha_{2}(y)=\cr
& \alpha_{1}(x)\vdash_{B} (l_{\cdot_{B}}(a)y) + r_{\vdash_{B}}(r_{\cdot_{A}}(y)a)\alpha_{1}(x),
\end{array} \\
\label{bieq236}
l_{\vdash_{B}}(\beta_{1}(a))(x \cdot_{A} y) = ( l_{\vdash_{B}}(a)x) \cdot_{A}\alpha_{2}(y) +
 l_{\cdot_{B}}(r_{\vdash_{A}}(x)a)\alpha_{2}(y),\\
 \label{bieq237}
r_{\dashv_{B}}(\beta_{2}(a))(x \cdot_{A} y) = r_{\cdot_{B}}(l_{\dashv_{A}}(y)a)\alpha_{1}(x) +
\alpha_{1}(x)\cdot_{A} (r_{\dashv_{B}}(a)y), \\
\label{bieq238}
\begin{array}{ll}
l_{\dashv_{B}}(l_{\cdot_{A}}(x)a)\alpha_{2}(y) & + (r_{\cdot_{B}}(a)x) \dashv_{A}\alpha_{2}(y)=\cr
& \alpha_{1}(x)\cdot_{B} (l_{\dashv_{B}}(a)y) + r_{\cdot_{B}}(r_{\dashv_{A}}(y)a)\alpha_{1}(x),
\end{array} \\
\label{bieq239}
l_{\cdot_{B}}(\beta_{1}(a))(x \dashv_{A} y) = ( l_{\cdot_{B}}(a)x) \dashv_{A}\alpha_{2}(y) +
 l_{\dashv_{B}}(r_{\cdot_{A}}(x)a)\alpha_{2}(y),\\
 \label{bieq240}
r_{\cdot_{B}}(\beta_{2}(a))(x \cdot_{A} y) = r_{\cdot_{B}}(l_{\cdot_{A}}(y)a)\alpha_{1}(x) +
\alpha_{1}(x)\cdot_{A} (r_{\cdot_{B}}(a)y), \\
\label{bieq241}
\begin{array}{ll}
l_{\cdot_{B}}(l_{\cdot_{A}}(x)a)\alpha_{2}(y) & + (r_{\cdot_{B}}(a)x) \cdot_{A}\alpha_{2}(y)=\cr
& \alpha_{1}(x)\cdot_{B} (l_{\cdot_{B}}(a)y) + r_{\cdot_{B}}(r_{\cdot_{A}}(y)a)\alpha_{1}(x),
\end{array} \\
\label{bieq242}
l_{\cdot_{B}}(\beta_{1}(a))(x \cdot_{A} y) = ( l_{\cdot_{B}}(a)x) \cdot_{A}\alpha_{2}(y) +
 l_{\cdot_{B}}(r_{\cdot_{A}}(x)x)\alpha_{2}(y).
  \end{eqnarray}
 Then, there is a BiHom-tridendriform algebra structure on the direct sum $ A \oplus B $ of the underlying vector spaces of
 $ A $ and $ B $ given for any $ x, y \in A, a, b \in B $ by
\begin{eqnarray*}
(x + a) \dashv ( y + b ) &:=& (x \dashv_{A} y + r_{\dashv_{B}}(b)x + l_{\dashv_{B}}(a)y)\cr
&+&(l_{\dashv_{A}}(x)b + r_{\dashv_{A}}(y)a + a \dashv_{B} b ), \cr
(x + a) \vdash ( y + b ) &:=& (x \vdash_{A} y + r_{\vdash_{B}}(b)x + l_{\vdash_{B}}(a)y)\cr
&+& (l_{\vdash_{A}}(x)b + r_{\vdash_{A}}(y)a + a \vdash_{B} b ),\cr
(x + a) \cdot ( y + b ) &:=& (x \cdot_{A} y + r_{\cdot_{B}}(b)x + l_{\cdot_{B}}(a)y)\cr
&+&(l_{\cdot_{A}}(x)b + r_{\cdot_{A}}(y)a + a \cdot_{B} b ).
\end{eqnarray*}
\end{thm}
\begin{proof}
The proof is obtained in a similar way as for Theorem \ref{matched ass}.
\end{proof}
Let $ A \bowtie^{l_{\dashv_{A}}, r_{\dashv_{A}},
l_{\vdash_{A}}, r_{\vdash_{A}},l_{\cdot_{A}}, r_{\cdot_{A}}, \beta_{1}, \beta_{1}}_{l_{\dashv_{B}}, r_{\dashv_{B}}, l_{\vdash_{B}},
r_{\vdash_{B}},l_{\cdot_{B}},
r_{\cdot_{B}}, \alpha_{1}, \alpha_{2}} B $ denote this BiHom-tridendriform algebra.

\begin{defn}
Let $ (A, \dashv_{A}, \vdash_{A},\cdot_{A}, \alpha_{1}, \alpha_{2}) $ and $  (B, \dashv_{B}, \vdash_{B},\cdot_{B}, \beta_{1}, \beta_{2}) $
be two BiHom-tridendriform algebras. Suppose there exist linear maps
$$ l_{\dashv_{A}}, r_{\dashv_{A}}, l_{\vdash_{A}}, r_{\vdash_{A}}, l_{\cdot_{A}}, r_{\cdot_{A}} : A \rightarrow gl(B),$$
$$ l_{\dashv_{B}}, r_{\dashv_{B}}, l_{\vdash_{B}}, r_{\vdash_{B}},l_{\cdot_{B}}, r_{\cdot_{B}} : B \rightarrow gl(A) $$
 such that $(l_{\dashv_{A}}, r_{\dashv_{A}}, l_{\vdash_{A}}, r_{\vdash_{A}},l_{\cdot_{A}}, r_{\cdot_{A}}, \beta_{1}, \beta_{2})$ is a bimodule of $ A,$
and $$(l_{\dashv_{B}}, r_{\dashv_{B}}, l_{\vdash_{B}}, r_{\vdash_{B}},l_{\cdot_{B}}, r_{\cdot_{B}},  \alpha_{1}, \alpha_{2}) \ \mbox{is a bimodule of}\  B.$$
If  \eqref{bieq201} - \eqref{bieq242} are satisfied, then $$(A, B, l_{\dashv_{A}},
r_{\dashv_{A}}, l_{\vdash_{A}}, r_{\vdash_{A}},l_{\cdot_{A}}, r_{\cdot_{A}}, \beta_{1}, \beta_{2}, l_{\dashv_{B}}, r_{\dashv_{B}}, l_{\vdash_{B}},
 r_{\vdash_{B}}, l_{\cdot_{B}}, r_{\cdot_{B}},\alpha_{1}, \alpha_{2})$$ is called a matched pair of BiHom-tridendriform algebras.
\end{defn}

\begin{cor}
Let $$(A, B, l_{\dashv_{A}}, r_{\dashv_{A}}, l_{\vdash_{A}}, r_{\vdash_{A}},l_{\cdot_{A}}, r_{\cdot_{A}}, \beta_{1}, \beta_{2},
 l_{\dashv_{B}}, r_{\dashv_{B}}, l_{\vdash_{B}}, r_{\vdash_{B}},l_{\cdot_{B}}, r_{\cdot_{B}}, \alpha_{1}, \alpha_{2}) $$
be a matched pair of BiHom-tridendriform algebras.
Then, $$(A, B, l_{\dashv_{A}} + l_{\vdash_{A}}+l_{\cdot_{A}}, r_{\dashv_{A}} + r_{\vdash_{A}}+r_{\cdot_{A}},\beta_{1}, \beta_{2},
l_{\dashv_{B}} + l_{\vdash_{B}}+l_{\cdot_{B}},  r_{\dashv_{B}} + r_{\vdash_{B}}+ r_{\cdot_{B}}, \alpha_{1}, \alpha_{2})$$ is a matched pair of the associated
BiHom-associative algebras $(A, \ast_{A}, \alpha_{1}, \alpha_{2})$ and  $(B, \ast_{B}, \beta_{1}, \beta_{2})$.
\end{cor}
\begin{proof}
Let $(A, B, l_{\dashv_{A}}, r_{\dashv_{A}}, l_{\vdash_{A}}, r_{\vdash_{A}},l_{\cdot_{A}}, r_{\cdot_{A}}, \beta_1,\beta_2,
 l_{\dashv_{B}}, r_{\dashv_{B}}, l_{\vdash_{B}}, r_{\vdash_{B}}, l_{\cdot_{B}}, r_{\cdot_{B}},\alpha_1,\alpha_2)$ be a matched pair of a BiHom-tridendriform algebras
 $$(A, \dashv_{A}, \vdash_{A}, \cdot_A,\alpha_1,\alpha_2)\quad \mbox{and}\quad (B, \dashv_{B}, \vdash_{B},\cdot_B,\beta_1,\beta_2).$$
 In view of Proposition \ref{propa}, the linear maps
 $l_{\dashv_{A}} + l_{\vdash_{A}}+l_{\cdot_A}, r_{\dashv_{A}} + r_{\vdash_{A}}+r_{\cdot_A}:A\rightarrow gl(B)$ and $l_{\dashv_{B}} + l_{\vdash_{B}}+l_{\cdot_B},  r_{\dashv_{B}} + r_{\vdash_{B}}+r_{\cdot_B}:B\rightarrow gl(A)$
 are bimodules of the underlying BiHom-associative algebras $(A,\ast_A, \alpha_1,\alpha_2)$ and $(B,\ast_B,\beta_1,\beta_2)$, respectively. Thus, \eqref{3}-\eqref{5} are equivalent to \eqref{bieq201}-\eqref{bieq221}, and
 \eqref{6}-\eqref{8} are equivalent to \eqref{bieq222}-\eqref{bieq242}.
\end{proof}


\end{document}